\documentclass[a4paper, 11pt]{amsart}
\usepackage[margin= 1in]{geometry}
\usepackage{graphicx}
\usepackage{amsfonts,amsmath,amssymb,amsthm, mathtools}

\usepackage[dvipsnames]{xcolor}
\usepackage{hyperref}
\usepackage{booktabs,tabularx}
\hypersetup{colorlinks=true,linkcolor=blue,citecolor=ForestGreen,urlcolor=blue}
\providecommand{\rank}{\operatorname{rank}}
\providecommand{\spn}{\operatorname{span}}

\providecommand{\R}{\mathbb R}
\providecommand{\C}{\mathbb C}
\providecommand{\m}{\mathfrak m}
\providecommand{\n}{\mathfrak n}
\providecommand{\h}{\mathfrak h}
\providecommand{\hh}{\widehat h}
\providecommand{\cA}{\widehat{\mathcal A}}
\providecommand{\cN}{\mathcal N}
\providecommand{\codim}{\operatorname{codim}}
\providecommand{\Id}{\operatorname{Id}}
\providecommand{\dd}{\mathrm d}
\mathtoolsset{showonlyrefs}
\newcommand{\rnk}{\mathrm{rank}\text{ }}

\newcommand{\ad}{\mathrm{ad}}
\newcommand{\rt}{\rho^i_}
\newcommand{\imad}{\mathrm{Im} \text{ } \mathrm{ad}}
\newtheorem{theorem}{Theorem}[section]
\newtheorem{Coro}[theorem]{Corollary}
\newtheorem{Lemma}[theorem]{Lemma}

\newtheorem{Prop}[theorem]{Proposition}
\newtheorem{Conj}[theorem] {Conjecture}
\theoremstyle{definition}
\newtheorem{Def}[theorem]{Definition}

\theoremstyle{remark}
\newtheorem{rk}[theorem]{Remark}
\newtheorem{example}[theorem]{Example}
\newtheorem{lemma}[theorem]{Lemma}

\theoremstyle{definition}

\theoremstyle{remark}

\numberwithin{equation}{section}
\title[Projectively Equivalent Sub-Riemannian Metrics on Step 2 Distributions]{Levi-Civita's type theorem for projectively equivalent sub-Riemannian metrics on step 2 distributions}
\date{\today}
\thanks{ I.\ Zelenko is partly supported by NSF grant DMS 2105528 and Simons Foundation Collaboration Grant for Mathematicians 524213.}
\author{Christopher Sinkule}
\address{Christopher Sinkule\\
	Department of Mathematics\\
	Texas A\&M University\\
	College Station\\
	Texas \ 77843\\
	USA}
\email{cpsinkule@exchange.tamu.edu}
\author{Igor Zelenko}
\address{Igor Zelenko\\
         Department of Mathematics\\
         Texas A\&M University\\
         College Station\\
         Texas \ 77843\\
         USA}
\email{zelenko@math.tamu.edu}
\urladdr{\url{http://www.math.tamu.edu/~zelenko}}
\date{}
\begin{document}
\begin{abstract}
This paper extends the classical results of \emph{Eisenhart (1923)} and \emph{Levi-Civita (1896)} in Riemannian geometry concerning the classification of pairs of germs of Riemannian metrics with the same geodesics either as parametrized curves (\emph{affine equivalence}) or as unparametrized curves (\emph{projective equivalence}), to the setting of \emph{sub-Riemannian metrics} on arbitrary \emph{step 2 distributions}: at every stable point in the projective case, and at every point in the affine case. In particular, this classification extends both the results of Zelenko (2004)\cite{Z} for contact and even-contact distributions and the affine result of Lin and Zelenko (2024) \cite{Zaifeng} for step 2 distributions whose Tanaka symbols satisfy the more restrictive condition of ad-surjectivity. Moreover, our projective statement covers the ad-surjective case as well.
In the affine case, we prove that if a sub-Riemannian metric from the considered class admits another metric sharing the same geodesics, but not constantly proportional to it, then locally it decomposes as a \emph{product of two sub-Riemannian metrics}.
In the projective case, we establish the \emph{2019 conjecture of Jean–Maslovskaya–Zelenko} \cite{jmz2019} on Levi-Civita pairs, showing in particular that the underlying distribution admits a \emph{product structure}, that the \emph{squares of each component} are \emph{involutive}, and that the restrictions of each sub-Riemannian metric in the pair to each leaf of the corresponding foliation are \emph{conformal} to each other under the appropriate pullback. The proof is based on a new \emph{graph-theoretic formalism}, in which appropriately compatible moving frames are represented by the \emph{vertices and edges} of certain graphs. As a byproduct, we introduce new \emph{discrete invariants} of step 2 nilpotent graded Lie algebras, called the \emph{characteristic tuple}. For left-invariant metrics on step 2 Carnot groups, this tuple is the growth of the osculating flags of the \emph{Jacobi curves} \cite{jac1} of generic normal extremals, i.e., the sequence of the lengths of the columns of their Young diagrams \cite{li2} starting from the second column; for other sub-Riemannian metrics with the same Tanaka symbol, this growth can be faster, but not slower. The compatible graphs are used on the dense open set where this tuple is locally constant. The bracket relations obtained there are identities between smooth tensors, so they extend by continuity across the points where the tuple jumps, and the normal form is then obtained near every stable point. Consequently, our classification holds without any assumption of local constancy of the characteristic tuple.
 \end{abstract}
\maketitle
\section{Introduction}
\label{intro_sec}
This paper is devoted to the classification of pairs of germs of sub-Riemannian metrics on step $2$ distributions which share the same (variational) geodesics, either  as unparametrized curves (\emph{projective equivalence}) or as parametrized curves (\emph{affine equivalence})  (see Theorems \ref{mainthm} and \ref{mainthm_eisen}  below). These results extend the classical theorems of \emph{Eisenhart (1923)} and \emph{Levi-Civita (1896)} for Riemannian metrics, which can be viewed as a special case of sub-Riemannian metrics where the underlying distribution coincides with the tangent bundle. In particular, our results confirm the phenomenon of direct product (or separation of variables), previously observed in the Riemannian setting: at the level of sub-Riemannian metrics in the affine case and at the level of the underlying distributions in the projective case, as stated in Corollaries \ref{direct_prod_affine_cor} and \ref{direct_prod_proj_cor} below.
Although step $2$ distributions still form a rather restrictive subclass among all distributions, the results of this paper represent substantial generalizations of the known classification results previously established for contact and even-contact distributions in \cite{Z}, as well as for the so-called ad-surjective step $2$ distributions in \cite{Zaifeng}. Moreover, the graph-theoretic approach developed in this work both significantly simplifies and clarifies the proof in the particular ad-surjective case given in \cite{Zaifeng}, and suggests that the method may be further extended to distributions of higher step through the use of \emph{hypergraphs}.
\subsection{Projective and affine equivalence of sub-Riemannian metrics}
In more detail, recall that a distribution $D$ on a manifold $M$ is a subbundle of the tangent bundle $TM$.
A sub-Riemannian manifold/structure is a triple $(M, D,g)$, where $M$ is a smooth manifold, $D$ is a  bracket-generating distribution, and for any $q$, $g(q)$ is an inner product on $D(q)$ which depends smoothly on $q$. We say that $g$ is a sub-Riemannian metric on $(M, D)$. As already mentioned, a Riemannian manifold/structure/metric appears as the particular case where $D=TM$.
In the sequel, we will assume that the distribution $D$ is bracket-generating, i.e. for every point $q\in M$ the iterative Lie brackets of vector fields tangent to a distribution $D$ (i.e., of sections of $D$) span the tangent space $T_qM$. In more detail,
one can define a filtration
\begin{equation}
\label{filtr}
D=D^{1}\subset D^{2}\subset\ldots D^j\subset \ldots
\end{equation}
of the tangent bundle, called a \emph{weak derived flag}, as follows:  set $D=D^{1}$ and define recursively
\begin{equation}
\label{Dpower1}
D^{j}=D^{j-1}+[D,D^{j-1}], \quad  j>1.
\end{equation}
If $X_1,\ldots, X_m$ are $m$ vector fields constituting a local basis of a distribution $D$,
then $D^{j}(q)$ is the linear span of all iterated Lie brackets of these vector fields, of length not greater than  $j$,  evaluated at a point $q$,
\begin{equation}
\label{Dpower2}
D^{j}(q)=  {\rm span}\{\left[X_{i_1},\ldots [X_{i_{s-1}},X_{i_s}]\ldots\right](q):  (i_1,\ldots,i_s)\in [1:m]^s, s\in[1:j] \}
\end{equation}
(here given a positive integer $n$ we denote by $[1:n]$ the set $\{1, \ldots, n\}$).
A distribution $D$ is called \emph{bracket-generating} (or \emph{completely nonholonomic}) if for any $q$ there exists $\mu\in\mathbb N$ such that $D^{\mu}(q)=T_q M$. The smallest such integer, denoted by $\mu(q)$, is called the \emph{degree of nonholonomy} of $D$ at the point $q$. If $D$ is equiregular, i.e., the dimensions of the spaces $D^j(q)$ do not depend on $q$, and its degree of nonholonomy is equal to $\mu$ at every point, one says that $D$ is a \emph{step $\mu$ distribution}. Note that a distribution whose degree of nonholonomy is equal to $2$ at every point is automatically equiregular.
Since we work locally, the assumption of bracket generation is not too restrictive: if a distribution is not bracket-generating then in a neighborhood $U$ of a generic point there exists a positive integer $\mu$ such that $D^{\mu+1}=D^{\mu}\subsetneqq TM$. So,  $D^\mu$ is a proper involutive subbundle of $TU$ and the distribution $D$ is bracket-generating on each integral submanifold of  $D^\mu$ in $U$. So, we can restrict ourselves to these integral submanifolds instead of $U$.
In the present paper, we consider geodesics defined via the variational approach, which is natural from the optimal-control perspective. For a comparison with the differential-geometric viewpoint, i.e., ``shortest'' versus ``straightest'' curves, see \cite{Alekseevsky2020}; note that, if ``straightest'' refers to all geodesics of an affine connection on $TM$, without a horizontality restriction, then both notions coincide only in the Riemannian case.
A horizontal curve $\gamma:[a, b]\rightarrow  M$ is an absolutely continuous curve tangent to $D$, i.e., $\gamma'(t) \in D \left(\gamma(t)\right)$. In the sequel, the manifold $M$ is assumed to be connected. By the Rashevskii-Chow theorem the assumption that  $D$ is bracket-generating guarantees that the space of horizontal curves connecting two given points $q_0$ and $q_1$ is not empty.  Consider the following energy-minimizing problem:
\begin{equation}
\label{energy}
\begin{split}
~& E(\gamma) = \displaystyle\int_a^b g\left(\gamma'(t),\gamma'(t)\right)\, dt\rightarrow \textnormal{min},\\
~&\gamma'(t)\in D \left(\gamma(t)\right) \quad \textnormal{ a.e. } t,\\
~& \gamma(a)=q_0, \quad \gamma(b)=q_1.
\end{split}
\end{equation}
Necessary optimality conditions for this problem are supplied by the Pontryagin Maximum Principle (\cite{AgrSach, Pont}) of optimal control theory, which defines special curves in the cotangent bundle $T^*M$, called the \emph{Pontryagin extremals}, so that a minimizer of the optimal control problem \eqref{energy} is a projection from $T^*M$ to $M$ of some \emph{Pontryagin extremal}.
In general, there are two	types of  Pontryagin  extremals for optimal control problems, \emph{normal} and \emph{abnormal} (\cite{ABB, AgrSach}): for the former, the Lagrange multiplier near the functional is non-zero, and, for the latter, it is zero.  In particular, abnormal extremals, as unparametrized curves,  depend only on the distribution $D$ and not on a metric $g$ defined thereon. This indicates that only normal extremals are essential for the considered problems of affine/projective equivalence (see Proposition  \ref{th:orb.to.proj}
for the precise formulation). Therefore, we give an explicit description of normal extremals only.  They  are the  integral curves of the Hamiltonian vector field  $\vec h$ on $T^*M$ corresponding to the Hamiltonian
\begin{equation}
\label{Hamiltonian}
h(p,q)=\frac12||\, p|_{D(q)}||^2, \quad q\in M, p\in T^*_qM,
\end{equation}
and lying on a nonzero level set of $h$. Here $||\, p|_{D(q)}||$ denotes the operator norm of the functional  $p|_{D(q)}$, i.e.,
$$||\, p|_{D(q)}||=\max \{p(v): v\in D(q), g(q)(v,v)=1\}.$$
The Hamiltonian $h$ defined by \eqref{Hamiltonian} is
called the \emph{Hamiltonian, associated with the metric $g$} or shortly the \emph{sub-Riemannian Hamiltonian}.
\begin{Def}
\label{geod}
The \emph{(variational) sub-Riemannian geodesics} are projections of the Pontryagin extremals of the optimal control problem \eqref{energy}.
\end{Def}
Note that in the Riemannian case the geodesics given by Definition \ref{geod} coincide with the usual Riemannian geodesics.
\begin{Def}
Let $M$ be a manifold and $D$ be a bracket-generating distribution on $M$. Two sub-Riemannian metrics $g_1$ and $g_2$ on $(M, D)$ are called \emph{projectively equivalent} at $q_0\in M$ if they have the same geodesics, up to a reparametrization, in a neighborhood of $q_0$, i.e., for every geodesic $\gamma(t)$ of $g_1$ there exists a reparametrization $t=\varphi(\tau)$ such that $\gamma\bigl(\varphi(\tau)\bigr)$ is a geodesic of $g_2$, and conversely, with $g_1$ and $g_2$ interchanged. Here a reparametrization is a diffeomorphism between intervals of the real line. The metrics $g_1$ and $g_2$ are called \emph{affinely equivalent} at $q_0$ if they are projectively equivalent at $q_0$ and the reparametrizations $\varphi(\tau)$ above are affine functions, i.e., they are of the form $\varphi(\tau)=a\tau+b$.
\end{Def}
We will write   $g_1\stackrel{p}{\sim} g_2$ and $g_1\stackrel{a}{\sim} g_2$ in the case of projective and affine equivalence, respectively.
Obviously, for a sub-Riemannian metric $g$ on $(M, D)$  and a positive constant $c$  the metrics $c g$ and $g$ are affinely equivalent. The metric $c g$ will be said a \emph{constantly proportional metric} to the metric $g$.
	\subsection{Review of classical results in Riemannian case} The problem of projective and affine equivalence of Riemannian metrics is classical, and the former was solved completely already in the nineteenth century: all pairs of projectively equivalent Riemannian metrics satisfying certain natural regularity assumptions were described by Levi-Civita in \cite{LeviCivita1896}, generalizing the earlier result of Dini \cite{Dini1870} in the $2$-dimensional case.
    All pairs of locally affinely equivalent Riemannian metrics were described by Eisenhart in \cite{Eisenhart1923} about a quarter of a century later.
    We formulate these classical theorems both for completeness and as preparation for the proper sub-Riemannian case. We begin by introducing some terminology that will also be used in the sub-Riemannian setting.
\begin{Def}
\label{transition_def}
    Given two distinct sub-Riemannian metrics $g_1$ and $g_2$ on $D$, we define the \emph{transition operator} to be the linear map $S(q):D(q)\to D(q)$ such that $g_2(X,Y)=g_1(S(q)X,Y)$ for every $X, Y \in D(q)$. A point $q_0\in M$ is called \emph{ stable w.r.t. the pair $(g_1, g_2)$} if  $S_q$ has the same number of distinct eigenvalues in a neighborhood of $q_0$ (or, equivalently, the tuple of multiplicities of the eigenvalues is constant in a neighborhood of $q_0$).
\end{Def}
\begin{Def}
\label{net_def}
A collection of foliations $\{\mathcal F_i\}_{i=1}^k$ on a manifold  $M$ forms a \emph{net of foliations} if, for every point $q$, the tangent spaces at $q$ to the leaves $\mathcal F_i(q)$ passing through $q$ define a splitting of $T_qM$, and there is a coordinate system in which the foliations $\mathcal F_i$ are parallel to the corresponding coordinate subspaces.  In this case, for every point $q$ the natural projection $\pi_i^{q}: U\to \mathcal F_i(q)\cap U$ is well defined on a sufficiently small product neighborhood $U$ of $q$. In what follows, all such projections are understood in this local sense.
\end{Def}
 Near a stable point $q_0$, decompose $TM$ into distributions $D_i$ of eigenspaces corresponding to the eigenvalues  $\lambda_i$, $1\leq i\leq k$, of the transition operator.
    \medskip
	\begin{theorem}[Levi-Civita, 1896, \cite{LeviCivita1896}]
Two Riemannian metrics $g_1$ and $g_2$ are projectively equivalent in a neighborhood of a stable point $q_0$ if and only if the following two conditions hold:
\begin{enumerate}
    \item The distributions $D_i$ are involutive, and their integral foliations $\mathcal F_i$ form a net near $q_0$.
    \item For each $i$, there exist a Riemannian metric $b_i$ and a positive smooth function $\mu_i$ on the leaf $\mathcal F_i(q_0)$ such that $\mu_i$ is constant if $\dim D_i>1$ and $\mu_i(q_0)\neq\mu_j(q_0)$ for all $i\neq j$, and, setting $\pi_i:=\pi_i^{q_0}$,
    \[
    g_1=\sum_{i=1}^k \gamma_i\,\pi_i^*b_i,
    \qquad
    g_2=\sum_{i=1}^k \lambda_i\,\gamma_i\,\pi_i^*b_i,
    \]
    where
    \[
    \lambda_i=\frac{1}{\pi_i^*\mu_i}\prod_{l=1}^k\frac{1}{\pi_l^*\mu_l},
    \qquad
    \gamma_i=\prod_{l\neq i}\left|\pi_l^*\mu_l-\pi_i^*\mu_i\right|.
    \]
\end{enumerate}
\end{theorem}
A pair of Riemannian metrics $(g_1,g_2)$ satisfying the conditions of the preceding theorem is called a \emph{Levi-Civita pair} at $q_0$.
\begin{theorem}[Eisenhart, 1923, \cite{Eisenhart1923}]
Two Riemannian metrics $g_1$ and $g_2$ are affinely equivalent in a neighborhood of a point $q_0\in M$ if and only if they form a Levi-Civita pair at $q_0$ for which all functions $\mu_i$ are constant. Equivalently, $g_1$ and $g_2$ are affinely equivalent in a neighborhood of $q_0$ if and only if there exist a net of foliations $\{\mathcal F_i\}_{i=1}^k$ and Riemannian metrics $h_i$ on the leaves $\mathcal F_i(q_0)$ such that
\[
g_1=\sum_{i=1}^k \pi_i^*h_i,
\qquad
g_2=\sum_{i=1}^k C_i\,\pi_i^*h_i
\]
in a neighborhood of $q_0$, for some positive constants $C_1,\ldots,C_k$.
\end{theorem}
\begin{Coro}
\label{Riemann_direct}
If a Riemannian metric $g$ is affinely nonrigid near a point $q_0$, i.e., if it admits an affinely equivalent metric that is not constantly proportional to it, then $g$ admits a nontrivial product structure in a neighborhood $U$ of $q_0$. More precisely, there exist positive-dimensional manifolds $U_1$ and $U_2$ such that $U=U_1\times U_2$, together with Riemannian metrics $g_1$ and $g_2$ on $U_1$ and $U_2$, respectively, for which
\[
g=\pi_1^*g_1+\pi_2^*g_2,
\]
where $\pi_i:U\to U_i$, $i=1,2$, are the canonical projections.
\end{Coro}
\medskip
\subsection{The generalized Levi-Civita pairs conjecture}
A class of projectively equivalent sub-Riemannian metrics that provides a natural generalization of Levi-Civita pairs was introduced in \cite{jmz2019}.
\begin{Def}
\label{Gen_Levi-Civita}
One says that two sub-Riemannian metrics $g_1$ and $g_2$ on a distribution $D$ form a \emph{generalized Levi-Civita pair} in a neighborhood  of a stable point $q_0$ if there exists a local net of foliations $\{\mathcal F_i\}_{i=1}^k$ such that
the following three conditions hold:
    \begin{enumerate}
    \item If $D_i$ is  the distribution of the eigenspaces of the transition operator corresponding to the eigenvalue $\lambda_i$, then $D_i\subset T \mathcal F_i$ and $D_i$ is bracket generating within every leaf of $\mathcal F_i$;
    \item $d\,\pi_i^{q_0} \bigl(D_i(q) \bigr)=D_i\bigl(\pi_i^{q_0}(q)\bigr) $;
\item On the leaf $\mathcal F_i(q_0)$ there is a sub-Riemannian metric $b_i$ with underlying distribution  $D_i$ (considered as a distribution on $\mathcal F_i(q_0)$) and a positive smooth function $\mu_i$ such that $\mu_i$ is constant if $\dim D_i>1$, $\mu_i(q_0)\neq\mu_j(q_0)$ for all $i\neq j$, and, setting $\pi_i:=\pi_i^{q_0}$,
$$g_1 =\sum_{i=1}^k \gamma_i\,\pi_i^*b_i, \quad
g_2=\sum_{i=1}^k\lambda_i\, \gamma_i\, \pi_i^* b_i,$$
where
$$\lambda_i=\cfrac{1}{\pi_i^*\mu_i}\displaystyle{\prod_{l=1}^k\cfrac{1}{\pi_l^*\mu_l}},
\qquad
\gamma_i=\displaystyle{\prod_{l\neq
i}\Bigl|\pi_l^*\mu_l-\pi_i^*\mu_i\Bigr|}.$$
\end{enumerate}
\end{Def}
In \cite[Proposition 5.5]{jmz2019}, it was shown that the metrics in every generalized Levi-Civita pair are projectively equivalent. This led to the following two conjectures concerning the converse.
\begin{Conj}[generalized Levi-Civita theorem]
\label{conj1}
Two sub-Riemannian metrics in a neighborhood of a stable point are projectively equivalent if and only if they form a generalized Levi-Civita pair.
\end{Conj}
\begin{Conj}[generalized Eisenhart theorem]
\label{conj2}
Two sub-Riemannian metrics are locally affinely equivalent if and only if they form a generalized Levi-Civita pair with constant functions $\mu_i$.
\end{Conj}
In the affine case the stable-point hypothesis is automatic, since the eigenvalues of the transition operator are constant \cite[Proposition 4.7]{jmz2019}. Moreover, affine equivalence implies projective equivalence. Therefore, if Conjecture \ref{conj1} holds under a given collection of regularity assumptions, then an affinely equivalent pair satisfying those assumptions is a generalized Levi-Civita pair. Writing $B_i:=\pi_i^*\mu_i$, the identities $\lambda_i=B_i^{-1}\prod_{r=1}^k B_r^{-1}$ imply that $B_j/B_i=\lambda_i/\lambda_j$, so all ratios $B_i/B_j$ are constant. Substituting $B_i=c_i B_1$ into the formula for $\lambda_1$ gives $\lambda_1=C B_1^{-(k+1)}$ for a positive constant $C$; hence $B_1$, and therefore every $B_i$, is constant. Conversely, a generalized Levi-Civita pair with constant $\mu_i$'s is affinely equivalent. Thus Conjecture \ref{conj1} implies Conjecture \ref{conj2} under the same regularity assumptions.

Let us mention two relevant classification results obtained prior to this work:
\begin{enumerate}
    \item Conjecture \ref{conj1}, and hence Conjecture \ref{conj2}, was confirmed in \cite{Z} for sub-Riemannian metrics on contact and even-contact distributions. Although the formulation in \cite{Z} does not explicitly use generalized Levi-Civita pairs, it is equivalent to the formulation above; the contact or even-contact hypothesis supplies the additional regularity.
    \item Conjecture \ref{conj2}, but not Conjecture \ref{conj1}, was confirmed in \cite{Zaifeng} for sub-Riemannian metrics on special step $2$ distributions whose Tanaka symbols are $\mathrm{ad}$-surjective at every point; see Remark \ref{ad-surjective_rem} for the definition of $\mathrm{ad}$-surjectivity.
\end{enumerate}
Another important initial step toward the proof of Conjecture \ref{conj1} in the general case was taken in \cite[Theorem 7.1]{jmz2019}, which proves the conjecture at the level of nilpotent approximations at a point (i.e., a left-invariant sub-Riemannian structure on the Lie group whose Lie algebra is equal to the Tanaka symbol at that point). Namely, if $D$ admits a pair of projectively equivalent, non-conformal sub-Riemannian metrics $(g_1,g_2)$ in a neighborhood of a stable point (see Definition \ref{transition_def}), then their nilpotent approximations at this point constitute a Levi-Civita pair with constant coefficients. Consequently (see also Theorem \ref{Tanaka_decomp_prop} below), the Tanaka symbol $\mathfrak{m}$ of $D$ is decomposable (i.e., a direct sum of fundamental graded Lie algebras), and this decomposition aligns with the eigenspace decomposition of the transition operator of the pair $(g_1,g_2)$.

Our main results establish these conjectures for step $2$ distributions.
\begin{theorem}
\label{mainthm}
Let $D$ be a step $2$ distribution and let $q_0$ be stable with respect to a pair of sub-Riemannian metrics $(g_1,g_2)$. Then Conjecture \ref{conj1} holds
in a neighborhood of $q_0$.
If the metrics are projectively equivalent, the leaves of the foliations $\mathcal F_i$ in Definition \ref{Gen_Levi-Civita} are integral submanifolds of $D_i^2$, where $D_i$ is the eigendistribution of the transition operator corresponding to $\lambda_i$.
\end{theorem}
The affine statement follows from Theorem \ref{mainthm} and the constancy of the eigenvalues of the transition operator for affinely equivalent metrics \cite[Proposition 4.7]{jmz2019}.
\begin{theorem}
\label{mainthm_eisen}
Let $D$ be a step $2$ distribution and let $q_0\in M$. Then Conjecture \ref{conj2} holds in a neighborhood of $q_0$.
For each affinely equivalent pair $(g_1,g_2)$ near $q_0$, the leaves of the foliations $\mathcal F_i$ in Definition \ref{Gen_Levi-Civita} are integral submanifolds of $D_i^2$, where $D_i$ is the eigendistribution of the transition operator corresponding to $\lambda_i$.
\end{theorem}
The conformal case gives the following consequence.
\begin{Coro}[Sub-Riemannian Weyl theorem]
\label{Weyl_thm}
Let $D$ be a step $2$ distribution on a connected manifold $M$, $\dim M>1$. If two sub-Riemannian metrics on $D$ are projectively equivalent and conformal, then they are constantly proportional.
\end{Coro}
\begin{proof}
Write $g_2=\alpha^2g_1$, where $\alpha$ is a positive smooth function on $M$. The transition operator is $\alpha^2\mathrm{id}_D$, so it has a single eigenvalue throughout $M$ and every point is stable. Theorem \ref{mainthm} therefore gives a generalized Levi-Civita pair with one factor, $k=1$. A rank-one distribution is involutive; hence a step $2$ distribution has rank at least two. It follows from Definition \ref{Gen_Levi-Civita} that $\mu_1$ is constant. For $k=1$, the empty product defining $\gamma_1$ equals $1$, and the formula for the eigenvalue gives $\lambda_1=\mu_1^{-2}$. Thus $\alpha^2=\mu_1^{-2}$ is locally constant. Since $M$ is connected, $\alpha$ is constant on $M$, as required.
\end{proof}
In Section \ref{weyl_comparison_sec}, we show that, for step $2$ distributions, Corollary \ref{Weyl_thm} is stronger than the sub-Riemannian Weyl theorems of \cite{JMZ21}: we give examples of step $2$ sub-Riemannian structures that do not satisfy the hypotheses of these theorems, while Corollary \ref{Weyl_thm} requires no additional hypotheses. See also the description of Section \ref{weyl_comparison_sec} in Subsection \ref{outline_subsec}.

We also record consequences of the main theorems concerning the direct product structure of affinely or stably projectively nonrigid sub-Riemannian metrics, in analogy with Corollary \ref{Riemann_direct}. An \emph{affinely nonrigid sub-Riemannian metric} is a metric that admits an affinely equivalent metric that is not constantly proportional to it. Similarly, a \emph{stably projectively nonrigid sub-Riemannian metric} is a metric that admits a projectively equivalent metric that is not constantly proportional to it and for which every point in the neighborhood under consideration is stable with respect to the pair.
We first clarify what we mean by a direct product structure.
We say that a distribution $D$ on a manifold $M$ \emph{admits a product structure} if there exist positive-dimensional manifolds $M_1$ and $M_2$, endowed with positive-rank distributions $D_1$ and $D_2$, respectively, such that the following conditions hold:
\begin{enumerate}
\item $M=M_1\times M_2$;
\item if $\pi_i:M\to M_i$, $i=1,2$, are the canonical projections and $\pi_i^*D_i$ denotes the pullback distribution
\[
\pi_i^*D_i(q)=\{v\in T_qM:d\pi_i(q)v\in D_i(\pi_i(q))\},
\]
then
\begin{equation}
\label{proddist}
D(q)=\pi_1^*D_1(q)\cap\pi_2^*D_2(q).
\end{equation}
\end{enumerate}
In this case, we write $(M,D)=(M_1,D_1)\times(M_2,D_2)$.
A sub-Riemannian structure $(M,D,g)$ \emph{admits a product structure} if there exist sub-Riemannian structures $(M_1,D_1,g_1)$ and $(M_2,D_2,g_2)$ such that $(M,D)=(M_1,D_1)\times(M_2,D_2)$ and
\begin{equation}
\label{prodmetric}
g=\pi_1^*g_1+\pi_2^*g_2,
\end{equation}
where $\pi_i:M\to M_i$, $i=1,2$, are the canonical projections.
As a direct consequence of Theorem \ref{mainthm_eisen}, we obtain the following.
\begin{Coro}
\label{direct_prod_affine_cor}
Let $(M,D,g)$ be a sub-Riemannian structure with $D$ of step $2$, and let $q_0\in M$. Suppose that $g$ admits an affinely equivalent metric $\bar g$ near $q_0$ that is not constantly proportional to $g$. Then $(M,D,g)$ admits a product structure near $q_0$.
\end{Coro}
\begin{proof}
Apply Theorem \ref{mainthm_eisen} to the pair $(g,\bar g)$. On a sufficiently small product neighborhood determined by the net of foliations, the eigenspace decomposition $D=\bigoplus_iD_i$ and the projection condition in Definition \ref{Gen_Levi-Civita} identify $D$ with the product of the distributions on the factor leaves. The functions $\mu_i$, and therefore the weights $\gamma_i$, are constant. Consequently, the formula $g=\sum_i\gamma_i\pi_i^*b_i$ expresses $g$ as the sum of the pullbacks of the factor metrics $\gamma_i b_i$. If there were only one factor, the same definition would give $\bar g=\mu_1^{-2}g$ with $\mu_1$ constant, contrary to the hypothesis. There are therefore at least two positive-dimensional factors. Taking one factor and grouping the others gives the required product of two sub-Riemannian structures.
\end{proof}
As a direct consequence of Theorem \ref{mainthm} we have
\begin{Coro}
\label{direct_prod_proj_cor}
Let $D$ be a step $2$ distribution on a manifold $M$ with $\dim M>1$, and let $q_0\in M$. Suppose that a sub-Riemannian metric $g$ on $D$ is stably projectively nonrigid near $q_0$. Then $D$ admits a product structure near $q_0$.
\end{Coro}
\begin{proof}
Choose a projectively equivalent metric $\bar g$ as in the definition of stable projective nonrigidity and work in a connected stable neighborhood of $q_0$. If the transition operator had only one eigenvalue there, the two metrics would be conformal and hence constantly proportional by Corollary \ref{Weyl_thm}, contrary to the choice of $\bar g$. Thus there are at least two eigendistributions. Theorem \ref{mainthm} provides the net of foliations in Definition \ref{Gen_Levi-Civita}. Its projection condition, together with $D=\bigoplus_iD_i$, identifies $D$ with the product of the positive-rank distributions on the factor leaves. Grouping one factor and the remaining factors gives the asserted product structure near $q_0$.
\end{proof}

\subsection{The outline of the paper}
\label{outline_subsec}
The proof of Theorems \ref{mainthm} and \ref{mainthm_eisen} has two stages. The graph calculations are carried out on the open dense regular locus of the characteristic tuple, where they give bracket relations between the eigenspace distributions of the transition operator and their squares. These relations are identities between smooth tensors, so Proposition \ref{tuple_continuation_prop} extends them to the entire stable neighborhood, and the generalized Levi-Civita normal form is then obtained near every stable point by rotating frames on the eigendistributions. A complementary result, Proposition \ref{eigenconst}, states that each transition eigenvalue is constant along its eigendistribution whenever the latter has rank at least two. Its proof in Appendix \ref{app:eigen} uses the step $2$ assumption and stability directly.

In Section \ref{orb_sec}, we recall the primary tools used in the proof: the reduction of projective equivalence to the orbital equivalence of the corresponding sub-Riemannian Hamiltonian flows, and the fundamental algebraic system deduced from it in \cite{jmz2019}. This system has infinitely many rows, and its coefficients are polynomials on the fibers of the cotangent bundle. The existence of an orbital map above a nonempty open subset of every fiber implies that all maximal minors of the augmented matrix of this system vanish identically on that fiber.

Section \ref{quasinormal_sec} develops the main graph-theoretic tools for the proof of Theorem \ref{mainthm}. We first recall the Tanaka symbol of a distribution at a point, a graded nilpotent Lie algebra of the same step as the distribution, which encodes the tensorial part of the iterated Lie brackets of its sections. In Subsection \ref{char_tuple_subsec}, we introduce the characteristic tuple of a step $2$ fundamental graded Lie algebra. To our knowledge, this discrete invariant has not appeared in the literature before, and we believe that it is of independent interest. As shown in Appendix \ref{app:jacobi}, the tuple gives the growth of the osculating flags of the generic Jacobi curves in the flat model, i.e., the lengths of the columns of their Young diagrams starting from the second column; for a general metric with the same symbol, this growth can be faster, but not slower. We then assign a special graph, called normalizable, to every step $2$ fundamental graded Lie algebra, and we construct the corresponding moving frames, called quasi-normal. The normalizable graph of a nonabelian symbol is a union of nested out-stars, which may share vertices, and of isolated vertices; it is determined, up to isomorphism, by the dimension of the horizontal layer and by the characteristic tuple. Further, we use the decomposition of the Tanaka symbol according to the eigenspaces of the transition operator, which follows from \cite{jmz2019}; this decomposition is smooth on every stable neighborhood. On the open dense regular locus of the characteristic tuple, it yields smooth frames compatible with the graphs of all eigenspace summands; note that these graphs may differ from one neighborhood to another. In such a frame, the vector field corresponding to an edge is the Lie bracket of the vector fields corresponding to the endpoints of this edge.

The remaining sections are devoted to the proof of the main theorems. Section \ref{proof_outline_sec} begins with the terminology for the structure functions of quasi-normal frames used in the rest of the paper (Definition \ref{cross_component_def}) and contains a more detailed synopsis of this proof; here we describe only its main line. The first step, carried out in Section \ref{principal_part_sec}, is to reduce the problem to a much simpler system. For this purpose, we apply an exact column operation, which does not change the maximal minors of the augmented matrix, and use the grading of polynomials on the fibers by the horizontal, or vertex, degree. In this way, we obtain the principal system $\mathfrak A\Psi=\tilde{\mathfrak b}$. The entries of its coefficient matrix are linear in the vertex variables, while its last column is quadratic in them. These are the first relevant homogeneous components; when the quadratic column vanishes, the cubic component is used as well.

The next step, carried out in Sections \ref{root_monomial_sec} and \ref{distinguished_minors_sec}, consists in establishing the rules for the coefficients of the so-called root monomials and in the analysis of the determinant of the core sub-matrix of the principal part. The product of all root monomials appearing in the entries of the core sub-matrix is called the distinguished monomial. Using a saturation argument, which forces its allocation row by row, we show that this monomial appears in the core determinant with coefficient $\pm1$. This proves that the coefficient matrix of the principal system has full column rank and, together with the identities for the augmented minors, that this system is rationally consistent.

Section \ref{transverse_vanishing_sec}, which together with the next section constitutes the technical core of the paper, deals with the horizontal cross-component structure functions of quasi-normal frames, i.e., the structure functions whose inputs and output are horizontal, but do not all belong to the same component (see Definition \ref{cross_component_def} for this terminology). We prove that such a structure function vanishes whenever its output belongs to a component different from the components of both of its inputs. For this purpose, we augment the core sub-matrix by one more row of the principal system and by its last column, and we compute the coefficients of suitable monomials in the resulting maximal minors, which vanish. We first treat the structure functions one of whose inputs has depth zero, i.e., is either the first root of its component, which is the center of the first out-star of the normalizable graph, or an isolated vertex of this graph. For inputs of positive depth, we augment the core sub-matrix by the row of the first root in the layer whose number exceeds the depth of the deeper input by one, and we consider the monomials obtained by multiplying the distinguished monomial by the edge variables along the path from the first root to the deeper input through the successive roots. Counting the available edge factors, we show that the coefficient of such a monomial in the augmented minor is, up to a nonzero factor, the structure function to be eliminated plus a multiple of a structure function one of whose inputs is a root. This determines the order of elimination in Proposition \ref{c2}: at each depth, the structure functions having a root as an input are eliminated first. The case when the inputs belong to a component whose summand of the Tanaka symbol is abelian is treated separately. Finally, the structure functions whose inputs and output belong to three distinct components are eliminated by multiplying the first-layer equations of each block of the principal system by the corresponding vertex variables and adding them up; this leads to a linear system whose matrix is formed by the differences of the eigenvalues of the transition operator and is invertible.

In Section \ref{other_sec}, we prove the vanishing of the remaining cross-component structure functions listed in Proposition \ref{outline_transverse_vanishing}, namely of those with a horizontal input and an edge input from the same component and with an output, central or horizontal, from another component. For this purpose, we use a different form of the first two layers of the principal system, the operator form derived in Appendix \ref{opp:fas}. Up to this point, the layers of the fundamental algebraic system and of its principal part have been described entry by entry, by recursive formulas in terms of the structure functions of the frame and of the fiber coordinates. In the operator form, the block of each layer corresponding to an eigendistribution $D_i$ is described invariantly, in terms of the corresponding summand of the Tanaka symbol, the metric $g_1$, and the covector. Namely, the covector determines, first, the vector of $D_i$ which is $g_1$-dual to its restriction to $D_i$, and, second, a linear functional on the central part of the summand, given by the restriction of the covector to the span of the edge fields of the corresponding component; by means of the Levi bracket and the metric $g_1$, this functional defines a skew-symmetric operator on $D_i$. The first layer of the block is then the transpose of the operator of taking the Lie bracket, in the Tanaka symbol, with the above vector of $D_i$, and the second layer is the transpose of the same operator for the image of this vector under the skew-symmetric operator. The key Lemma \ref{beta2} states that if a system formed by two layers of this form, with zero right-hand side in the first layer, has a rational solution and its second right-hand side is a polynomial that does not depend on the above functional on the central part, then this right-hand side vanishes. By the results of Section \ref{transverse_vanishing_sec}, the right-hand sides of the first two layers of each block of the principal system satisfy these conditions. Together with the Jacobi identity, this lemma gives the vanishing of the structure functions under consideration with a central output. Repeating the same argument for the system whose right-hand side is the cubic part of the last column, which becomes relevant after the quadratic part has been shown to vanish, we obtain the vanishing of those with a horizontal output. The resulting inclusions for the Lie brackets imply that every partial sum of the distributions $D_i^2$ is involutive.

It remains to recover the metrics themselves, which is done in Section \ref{rotation_sec} by rotating the frames on the eigendistributions. The involutive distributions $D_i^2$ define a local product structure, the horizontal subdistributions descend to its factors, and the lifted frames corresponding to different factors commute. The first divisibility condition then determines the Gram matrices of the lifted frames up to one positive scalar for each factor. Setting $\lambda_i=\alpha_i^2$, we show that the functions
\[
B_i=\frac{(\prod_j\lambda_j)^{1/(k+1)}}{\lambda_i},\qquad
\gamma_i=\prod_{j\ne i}|B_j-B_i|
\]
give the metric normal form. Moreover, the functions $B_i$ descend to the functions $\mu_i$ of Definition \ref{Gen_Levi-Civita}, which are constant on their own leaves when $\dim D_i>1$. This argument works near every stable point: in Subsection \ref{tuple_continuation_subsec}, the bracket relations obtained on the regular locus, which are identities between smooth tensors, are extended by continuity to the whole stable neighborhood, and in Subsection \ref{rotation_argument_subsec} the rotation is performed in an arbitrary orthonormal frame adapted to the eigenspace decomposition. In the affine case, every eigenvalue of the transition operator is constant, and the same construction yields the corresponding statement about the product structure of the metrics at every point.

In Section \ref{weyl_comparison_sec}, we compare Corollary~\ref{Weyl_thm} with the sub-Riemannian Weyl theorems of \cite{JMZ21}, which were proved there for distributions of arbitrary step, but under additional assumptions on complex abnormal extremals: either the families of complex non-strictly normal geodesics are not too large, or all complex abnormal extremals are of minimal order. We decide which of these hypotheses hold for several families of step $2$ nilpotent approximations. The hypothesis on non-strictly normal geodesics concerns only the nonzero levels of the Hamiltonian. The octonionic H-type metrics with centers of dimension six and seven satisfy it, because all their complex normal geodesics of nonzero energy are strictly normal. However, the zero energy level of these metrics is foliated by non-strictly normal geodesics. Hence the hypothesis of the more technical \cite[Theorem~3.2]{JMZ21}, to which \cite[Theorem~1.11]{JMZ21} is reduced, fails, and for these metrics the proof of \cite[Theorem~1.11]{JMZ21} requires an additional argument, whereas their Weyl rigidity follows from Corollary~\ref{Weyl_thm}. On the other hand, certain products of these algebras with arbitrary step $2$ algebras, as well as an infinite family of indecomposable central amalgams of dimensions $17,19,21,\ldots$, satisfy none of the hypotheses of \cite{JMZ21}. For the central amalgams, Corollary~\ref{Weyl_thm} implies that every sub-Riemannian metric on a distribution with such a symbol is projectively rigid.

The six appendices collect material that would otherwise interrupt the exposition. Appendix \ref{opp:fas} gives the interpretation of the fundamental algebraic system in operator form. Appendices \ref{app:rj} and \ref{app:column} compute the correction terms that are specific to the projective case. Appendix \ref{app:eigen} gives an independent proof of Proposition \ref{eigenconst}, working directly with the raw shifted system, and Appendix \ref{app:vanish} deduces from this proposition that the correction terms mentioned above vanish. Finally, Appendix \ref{app:jacobi} explains the geometric meaning of the characteristic tuple in terms of the growth of Jacobi curves; this material is not needed for the proofs of the main results.
\subsection{Disclosure of the use of AI tools}
\label{ai_disclosure_subsec}
In the preparation of this paper, the authors used the AI assistants Claude Opus~5.5 and Claude Fable~5.1 (Anthropic) and ChatGPT Astra~6 (OpenAI) for editorial purposes and for verifying the mathematical completeness of the many-step proof of Theorems~\ref{mainthm} and~\ref{mainthm_eisen}. The original version of this proof required the constancy of the characteristic tuple (Definition~\ref{reg_assump}). During the verification, the AI tools suggested that this assumption can be removed, and the authors implemented this in Subsection~\ref{tuple_continuation_subsec}. The verification also revealed an error in the coefficient formula of an earlier version of Lemma~\ref{augmented_path_rigidity}; the AI tools proposed the corrected formula, which contains an additional term, and this correction determines the order of elimination in Proposition~\ref{c2}. In addition, the AI tools were very helpful in the construction of the examples of Subsection~\ref{wc:new_subsec} and of the example in Remark~\ref{Jacobi_strict_rem}. Apart from the contributions described above, the results of this paper and their proofs are due to the authors, who have checked all statements and proofs in the final text and take full responsibility for its content.

\section{Reduction to orbital equivalence and fundamental algebraic system}
\label{orb_sec}
In \cite{jmz2019}, following \cite{Z}, the problems of projective and affine equivalence of sub-Riemannian metrics were reduced to the study of the orbital equivalence of the corresponding sub-Riemannian Hamiltonian systems for normal Pontryagin extremals of the energy-minimizing problem \eqref{energy}, which in turn was reduced to the study of solvability of a special linear algebraic system with coefficients polynomial in the fibers, called the \emph{fundamental algebraic system} (\cite[Proposition 3.10]{jmz2019}). In this section, we summarize the information from \cite{jmz2019} on the fundamental algebraic system that will be used below.
\subsection{Reduction to orbital equivalence}
\label{orbital_reduction_subsec}
As before, fix a connected manifold $M$ and a bracket-generating equiregular distribution $D$ on $M$, and consider two sub-Riemannian metrics $g_1$ and $g_2$ on $(M, D)$. We denote by $h_1$ and $h_2$ the respective sub-Riemannian Hamiltonians of $g_1$ and $g_2$, as defined in \eqref{Hamiltonian}. Let the annihilator $D^\perp$ of $D$ in $T^*M$ be defined as follows:
\begin{equation}
\label{perp}
D^\perp=\{(p, q)\in T^*M: p|_{D(q)}=0\}
\end{equation}
It  coincides with the zero level set of the sub-Riemannian Hamiltonian $h$ from \eqref{Hamiltonian}.
\begin{Def}
\label{orb_equiv_def}
We say that $\vec h_1$ and  $\vec h_2$ are \emph{orbitally diffeomorphic}
on an open subset $V_1$ of $T^*M\backslash D^\perp$  if there exists an open subset $V_2$ of  $T^*M\backslash D^\perp$ and a diffeomorphism $\Phi: V_1 \rightarrow V_2$ such that $\Phi$ is fiber-preserving, i.e.,  $\pi (\Phi(\lambda)) = \pi (\lambda)$, $\Phi$ preserves the level sets of the Hamiltonians $h_1$ and $h_2$, i.e. $h_2\bigl(\Phi(\lambda)\bigr) =h_1(\lambda)$,  and $\Phi$ sends the integral curves of $\vec h_1$ to the reparameterized integral curves of $\vec h_2$, i.e., there exists a smooth function $s=s(\lambda,t)$ with $s(\lambda,0)=0$ such that $\Phi\bigl(e^{t\vec h_1}\lambda\bigr) =e^{s\vec h_2} \bigl(\Phi(\lambda)\bigr)$ for all $\lambda \in V_1$ and $t\in \mathbb R$ for which $e^{t\vec h_1}\lambda$ is well defined. Equivalently, there exists a smooth function $a(\lambda)$ such that
\begin{equation}
\label{orbeq}
d\Phi\, \vec{h}_1(\lambda) = a(\lambda) \vec{h}_2(\Phi(\lambda)).
\end{equation}
The map $\Phi$ is called an \emph{orbital diffeomorphism} between the extremal flows of $g_1$ and  $g_2$.
\end{Def}
The reduction of projective (respectively,  affine) equivalence of sub-Riemannian metrics  to the orbital ( respectively, a special form of orbital) equivalence of the corresponding sub-Riemannian Hamiltonian systems is given by the following:
\begin{Prop}\cite[a combination of Proposition 3.4  and Theorem 2.10 there]{jmz2019}
\label{th:orb.to.proj}
Assume that the sub-Riemannian metrics $g_1$ and $g_2$ are projectively equivalent in a neighborhood  $U\subset M$ and let $\pi: T^*M\rightarrow M$ be the canonical projection.  Then, over every $q\in U$, the ample covectors form an open dense subset of $T_q^*M\setminus D^\perp$, and near each such covector $\lambda_1$,  $\vec{h}_1$ and $\vec{h}_2$ are orbitally diffeomorphic on a neighborhood $V_1$ of $\lambda_1$ in $T^*M$.
Moreover, if $g_1$ and $g_2$ are affinely equivalent in a neighborhood  $U\subset M$, then the function $a(\lambda)$ in \eqref{orbeq} satisfies $\vec{h}_1(a)=0$. \end{Prop}
Let us make several comments on the last proposition. In the cited paper, this result is formulated on the energy level $h_1=1/2$. Using positive homogeneity along the fibers, one extends the orbital map from this level to an open subset of the full cotangent bundle. Moreover, composing it with the fiberwise antipodal map if necessary, we can choose the time orientation so that $a>0$. Consequently, since by the last proposition an orbital map exists near the ample covectors over every point of the base, the polynomial minors used below vanish on a nonempty open subset of \emph{each} base fiber, and hence, being polynomial, they vanish on that entire fiber. Here, by these polynomial minors, we mean the maximal minors of the augmented matrix of the fundamental algebraic system (see the end of this section).

\subsection{Fundamental algebraic system}
\label{fas_subsec}
From now on, assume that we work in a neighborhood of a point $q_0 \in M$ that is stable with respect to the pair $(g_1, g_2)$ (see Definition \ref{transition_def}). Let $D_i \subset D$ be the eigenspace of the transition operator $S$ corresponding to the eigenvalue $\alpha_i^2$, and let $k$ be the number of distinct eigenvalues of $S$. Let $V$ and $E$ be arbitrary finite index sets with cardinalities $\mathrm{rank}\,D$ and $\mathrm{corank}\,D$, respectively. Write $V$ as a disjoint union $V=\bigsqcup_{i=1}^k V_i$, where $|V_i|=\mathrm{rank}\,D_i$. At this stage, $V$, the subsets $V_i$, and $E$ are only index sets; they are not yet identified with the vertices or edges of a graph.
Consider a local frame $\{X_s\}_{s\in V\cup E}$ in a neighborhood $\mathcal{U}$ of $q_0$ such that $\{X_v\}_{v\in V}$ is an orthonormal frame on $D$ with respect to the metric $g_1$ and
\begin{equation}
\label{span_comp}
\mathrm{span}\{X_v\}_{v\in V_i} = D_i, \quad 1 \leq i \leq k.
\end{equation}
Such a frame will be called \emph{weakly adapted to the pair of sub-Riemannian metrics $g_1$ and $g_2$}. In Section \ref{quasinormal_sec}, we will assign to the pair $(g_1, g_2)$ a special graph, which is a disjoint union of $k$ graphs. On the open dense locus where the characteristic tuple is locally constant (see Definition \ref{reg_assump}), we will consider more specialized frames (called quasi-normal frames) indexed by the vertices and edges of this graph. This is the reason for our current notation $V$ and $E$ for the indices of a weakly adapted frame. Note that quasi-normal frames will be used only on this locus: the bracket relations obtained with them are extended across the exceptional set, where the characteristic tuple jumps, by a continuity argument (Proposition \ref{tuple_continuation_prop}), and the last part of the proof uses only weakly adapted frames.
The \emph{structure functions} of a weakly adapted frame $\{X_s\}_{s\in V\cup E}$ are the smooth functions $c^w_{sv}$, $s,v,w\in V\cup E$, defined by
\begin{equation}
\label{struct_function}
[X_s,X_v]=\sum_{w\in V\cup E}c^w_{sv}X_w.
\end{equation}
A weakly adapted frame also induces linear coordinates on each fiber of $T^*\mathcal U$. To show this, we first define the following functions called \emph{quasi-impulses}:
\begin{equation}
\label{ui}
u_s(q,p) =  p \left(X_s(q)\right), \quad s \in V \cup E,\,q\in \mathcal U,\, p\in T_q^*\mathcal U.
\end{equation}
Given $q\in \mathcal U$, the tuple $\{u_s(q, \cdot)_{s \in V\cup E}\}$ forms a coordinate system on the fiber $T^*_q \mathcal U$ with corresponding coordinate frame $\{\partial_{u_s}\}_{s\in  V \cup E}$ on the fibers of $T^*\mathcal U$. Finally,  lift $X_s$ to the vector field $Y_s$ on $T^*M$ so that $\pi_* Y_s=X_s$ (where $\pi$ denotes the canonical projection) and $du_w(Y_s)=0$ for any $w \in V \cup E$. The collection  $\{Y_s,\partial_{u_s}\}_{s\in  V \cup E}$ constitutes a frame on $T^*\mathcal U$.
The expression for the sub-Riemannian Hamiltonian then becomes
\begin{equation}
 h_1 = \frac{1}{2}\sum_{v\in V} u_v^2
 \end{equation}
 Therefore,
 \begin{equation}
 \label{eq:vech}
\vec{h}_1 =\sum_{v\in V} u_v\vec u_v = \sum_{v\in V} u_v Y_v + \sum_{v\in V}\sum_{w,s\in V\cup E} c^{s}_{vw} u_v u_s \partial_{u_w}.
\end{equation}
 The proof of \eqref{eq:vech} follows from the relations
\begin{equation}
\label{poison1}
\overrightarrow {H_{Z_1}} (H_{Z_2})= d\, H_{Z_2}\bigl(\overrightarrow {H_{Z_1}}\bigr) =H_{[Z_1, Z_2]}.
\end{equation}
Where $H_Z(p, q):=p(Z(q)), p\in T_q^*M$.
From this and \eqref{struct_function}, it follows immediately that
\begin{equation}
\label{Ham_lift}
  \begin{split}  \overrightarrow{u_v} = Y_v + \sum_{w\in V \cup E}\vec{u_v}(u_w)\partial_{u_w}
=  Y_v + \sum _{w,s \in V \cup E}c_{vw}^{s}u_s\partial_{u_w},
\end{split}
\end{equation}
which, in turn, implies \eqref{eq:vech}.
Define the function $\alpha:V\rightarrow \mathbb R$ by
\begin{equation}
\label{alpha_of}
\alpha(v):=\alpha_i, \qquad v\in V_i,\quad 1\leq i\leq k,
\end{equation}
where $\alpha_i>0$ is chosen so that $\alpha_i^2$ is the eigenvalue of the transition operator corresponding to $D_i$. No value of $\alpha$ on the index set $E$ is needed in the construction below.
If $\vec{h}_1$ and $\vec{h}_2$ are orbitally diffeomorphic via $\Phi$ near $\lambda_0\in \pi^{-1}(q_0)$, then we
denote by
\begin{equation}
\label{Phi_w}\Phi_w:=u_w \circ \Phi (\lambda), w \in V \cup E
\end{equation}
the coordinates of $\Phi$ on a given fiber. Then it is easy to see \cite[Lemma 1]{Z} that the function $a$ from \eqref{orbeq}
 satisfies
 \begin{equation}
 \label{eq:2a}
        a^2=\dfrac{\sum\limits_{v\in V}\alpha(v)^2 u_v^2 }{\sum\limits_{v\in V} u_v^2 }
    \end{equation}
and
    \begin{equation}
        \Phi_v = \frac{\alpha(v)^2 u_v}{a}, \quad \forall v \in V.
    \end{equation}
By \cite[Lemma 2]{Z}, as a consequence of the orbital equivalence equation \eqref{orbeq}, the remaining $\Phi_w$, for the edges $w \in E$,  satisfy the overdetermined system of $|V|+|E|$ equations such that $|V|$ of them are linear algebraic and the rest $|E|$ equations are of the form
\begin{equation}
\label{diff_affine}
\vec h_1(\Phi_w)=\text{expression affinely depending on } \Phi \text{'s}\end{equation}

By differentiating the algebraic equations along the vector field $\vec h_1$ and substituting the right-hand sides of \eqref{diff_affine} for $\vec h_1(\Phi_w)$ in the resulting equations, we obtain $\vert{}V\vert{}$ new algebraic equations. Iterating this procedure an arbitrary number of times yields an algebraic system \eqref{A.phi.B} consisting of infinitely many linear algebraic equations in the variables $\Phi_w$. This system is called \emph{the fundamental algebraic system} and is explicitly described in the following proposition:
\begin{Prop}(\cite[ a combination of Proposition 3.4, Proposition 3.10, Proposition 4.2, Lemma 4.4] {jmz2019})
\label{fda}
Assume that the sub-Riemannian metrics $g_1$ and $g_2$ on a step $2$ distribution $D$ are projectively equivalent in a neighborhood $\mathcal U\subset M$ of a stable point.
Let $\{X_s\}_{s\in V\cup E}$ be a frame on $\mathcal U$ weakly adapted to the pair $(g_1,g_2)$ and $\Phi$ be the corresponding orbital diffeomorphism between the normal extremal flows of $g_1$ and $g_2$ with coordinates $\Phi_w$, $w \in V \cup E$, as in \eqref{Phi_w}. Set $$\widetilde{\Phi}_w=a\Phi_{w}, \quad \forall\, w \in E, $$
\begin{equation}
\label{qjk}
q_{ws} = \sum_{v \in V} c_{vw}^s u_v, \quad w,s\in V\cup  E,
\end{equation}
and, for $w\in V$,
\begin{equation}
\label{rj_unreduced}
\begin{split}
R_w={}&
\vec h_1\bigl(\alpha(w)^2\bigr)u_w
+\alpha(w)^2\vec h_1(u_w)
-\frac12\alpha(w)^2u_w\frac{\vec h_1(a^2)}{a^2}\\
&-\frac12\sum_{v\in V}
X_w\bigl(\alpha(v)^2\bigr)u
_v^2
-\sum_{v,s\in V}
\alpha(s)^2c^s_{vw}u_vu_s .
\end{split}
\end{equation}
Then $\widetilde{\Phi}$ satisfies a linear system of equations,
	 \begin{equation}
	 \label{A.phi.B}
	A \widetilde{\Phi} = b,
	 \end{equation}
where $A$ is a matrix with $|E|=n-m$ columns and an infinite number of rows, and $b$ is a column vector with an infinite number of rows. These infinite matrices can be decomposed in layers of $|V|=m$ rows each as
\begin{equation}
\label{eq:A_and_b}
		A = \left(
		 \begin{array}{c}
		A^1 \\
		A^2 \\
		 \vdots \\
		A^{p} \\
		 \vdots \\
		 \end{array}
		 \right)
		 \qquad \hbox{and}
		 \qquad b= \left(
		 \begin{array}{c}
		b^1 \\
		b^2 \\
		 \vdots \\
		b^{p} \\
		 \vdots \\
		 \end{array}
		 \right),
\end{equation}
where the coefficients $a^p_{vw}$ of the $(m \times (n-m))$ matrix $A^p$, $p \in \mathbb{N}$,  are defined by induction as
		 \begin{equation}
		 \label{elem.A}
		 \left\{
		 \begin{aligned}
		& a^1_{v,w} = q_{vw}, \qquad  &v \in V, \ w \in E,\\
		& a^{p+1}_{v, w} = \vec{h}_1(a^p_{vw}) + \sum_{l\in E}a^p_{v,l} q_{l w} , \qquad  &v \in V, \ w \in E,
		 \end{aligned}
		 \right.
		 \end{equation}
and the entries $b^{p}_v$, $v \in V$, of the vector $b^p \in \mathbb{R}^m$ are defined by
\begin{equation}
		 \label{term.of.b.rec.form}
		 \left\{
\begin{aligned}
		& b^{1}_{v} = \ R_v, \\
		& b^{p+1}_{v} = \ \vec{h}_1(b^p_{v}) -\cfrac{\vec h_1(a)}{a} b_v^p-\sum_{l\in E} \sum_{w\in V} a^p_{v,l} u_w \left(
		\alpha(w)^2 q_{lw} + \dfrac{X_l\bigl(\alpha(w)^2\bigr)}{2}\cdot u_w\right )\\
		 \end{aligned}
		 \right. .
\end{equation}
Finally, by \cite[Proposition 4.2 and Lemma 4.4]{jmz2019},
$\cfrac{\vec{h}_1(a)}{a}$
is a linear form on the fibers $T^*M$ depending on $u_v$ with $v\in V$ only.\footnote{In more detail, in \cite[Proposition 4.2 and Lemma 4.4]{jmz2019} it was shown that $\cfrac{\vec{h}_1(a^2)}{a^2}$ has this property, called the \emph{first divisibility condition}, following \cite{Z}. However, the same is true for $\cfrac{\vec{h}_1(a)}{a}=\cfrac{\vec{h}_1(a^2)}{2a^2}$.}
\end{Prop}
\begin{rk}
    The column vector $b$ in \eqref{A.phi.B} here corresponds to $a\cdot b$ in the notations of the fundamental algebraic system in \cite[Proposition 3.10]{jmz2019}.  Similar replacement of $b$ with $a\cdot b$ was made in \cite[Proposition 3.7]{Zaifeng} in the affine case in which $\vec h_1(a)= 0$. In the general projective case, $\vec h_1(a)$ is not necessarily zero. An additional term appears here compared to \cite[Proposition 3.7]{Zaifeng} on the right-hand side of the second line of \eqref{term.of.b.rec.form}. The reason for redefining $b$ in this way compared to \cite{jmz2019} is that all components of the new  $b$ are polynomials on the fibers of  $T^*M$. Though $\cfrac{\vec{h}_1(a)}{a}$ appears to be rational, the first divisibility condition mentioned in the last footnote shows that it is actually a polynomial on the fibers of $T^*M$.
\end{rk}
\begin{Prop}
\label{eigenconst}

Assume that the sub-Riemannian metrics $g_1$ and $g_2$ on a step $2$ distribution
are projectively equivalent in a neighborhood of a point stable with respect to the pair.
Then, for every $i$ such that $|V_i|>1$, the corresponding eigenvalue $\alpha_i^2$ of the transition operator satisfies
\begin{equation}
\label{same_eigenspace_zero}
X_y(\alpha_i^2)=0,
\qquad
y\in V_i,
\end{equation}
i.e., $\alpha_i^2$ is constant along the eigenspace distribution $D_i$.\footnote{For affine equivalence, the stronger conclusion that all eigenvalues are constant holds for arbitrary bracket-generating distributions; see \cite[Proposition 4.7]{jmz2019}.}

\end{Prop}
\begin{proof}
    (See Appendix  \ref{app:eigen}).
\end{proof}

\begin{rk}
\label{Weyl_proof_rem}
Corollary \ref{Weyl_thm} also follows directly from Proposition \ref{eigenconst}, without the graph calculations. Indeed, if $g_2=\alpha^2g_1$, then the transition operator is $\alpha^2\,\mathrm{id}_D$, so every point is stable and the only eigendistribution is $D$, which has rank at least two. Proposition \ref{eigenconst} then shows that $\alpha^2$ is constant along $D$. Since $D$ is bracket generating, $\alpha$ is locally constant, and hence constant on the connected manifold $M$.
\end{rk}
Combining Lemma \ref{rjlem} with the vanishing of its correction term established in Lemma \ref{vanishlem}, we obtain the following simplified expression for $R_w$:
\begin{Coro}
\label{rj_reduction}
Under the assumptions of Proposition \ref{fda},
the function $R_w$ in the fundamental algebraic system can  be written in the simpler form
\begin{equation}
\label{RJ}
R_w=
\sum_{\substack{s,v\in V\\s\neq v}}
\bigl(\alpha(w)^2-\alpha(s)^2\bigr)c^s_{vw}u_vu_s
+\alpha(w)^2\sum_{s\in E}\sum_{v\in V}c^s_{vw}u_vu_s,
\qquad w\in V.
\end{equation}
\end{Coro}
\begin{proof}
Item (1) of Lemma \ref{vanishlem} shows that the residual term in Lemma \ref{rjlem} vanishes for every component, including one-dimensional components by direct cancellation. Substitution into \eqref{rj_final} gives \eqref{RJ}.
\end{proof}
The sub-matrix $A^p$ of the matrix $A$  as in \eqref{eq:A_and_b} will be called the \emph{$p$th layer of $A$} and the subcolumn $b^p$ of the column $b$ there will be called the \emph{$p$th layer of $b$}.
The rows and columns of $A^p$ are indexed by $V$ and $E$, respectively. In the sequel, for brevity, we call the row of $A^p$ indexed by $v\in V$ the \emph{$v$-row} of $A^p$.
Assume that $n=\dim M$ and $m=\rank D$. The matrix $A$ has $n-m$ columns and infinitely many rows and $b$ is the infinite-dimensional column vector. So, the fundamental algebraic system \eqref{A.phi.B} is an over-determined linear system on $(\Phi_{m+1},\dots,\Phi_n)$, and, by \eqref{elem.A}, \eqref{term.of.b.rec.form}, and the last sentence of Proposition \ref{fda}, all entries of $A$ and $b$ are polynomials in the $u$-variables, i.e., polynomials on the fibers of $T^*M$. Therefore, all $(n-m+1)\times (n-m+1)$ minors of the augmented matrix $[A|b]$ must be equal to zero. Since all of the minors are polynomials in $u_j$'s,  the coefficient of every monomial of these polynomials must be equal to zero. The result is a large collection of constraints on the structure coefficient $c^k_{ij}$. By discovering and analyzing the monomials with the ``simplest'' coefficients, and since they necessarily must vanish by the previous comments, we are able to prove our main Theorem \ref{mainthm}.
\section{Characteristic tuples of step 2 fundamental Lie algebras, normalizable graphs, and quasi-normal frames}
\label{quasinormal_sec}
In this section, we develop the main graph-theoretic tools for the proof of Theorem \ref{mainthm}: the characteristic tuple of a step $2$ fundamental Lie algebra, the normalizable graph, and the quasi-normal frames. We then introduce the condition of constancy of the characteristic tuple at a point (Definition \ref{reg_assump}), and we work under this condition until Section \ref{rotation_sec}. This condition plays only an auxiliary role. Under it, one can work with a fixed normalizable graph and introduce quasi-normal frames (Subsection \ref{quasinormal_frames_subsec}), which is convenient for the analysis of the fundamental algebraic system; however, the condition is not needed for our main Theorem \ref{mainthm}. In more detail, the graph calculations are carried out on the open dense set where this condition holds, and the bracket relations obtained from them are then extended across the complement of this set by continuity, in Subsection \ref{tuple_continuation_subsec}.

\subsection{Tanaka symbols, characteristic tuples, and regularity}
\label{char_tuple_subsec}
We begin by recalling the Tanaka symbol of a distribution. Its definition requires the distribution to be \emph{equiregular} near the point under consideration.
A distribution $D$ is called \emph{equiregular} at a point $q$ if there exists a neighborhood $U$ of $q$ in $M$ such that, for every $j>0$, the dimension of $D^j(y)$ is constant for all $y\in U$, where $D^j$ is defined in \eqref{Dpower1} (equivalently, in \eqref{Dpower2}).
Note that a bracket-generating distribution is equiregular at a generic point.
From now on we assume that $D$ is an equiregular bracket-generating distribution with the degree of nonholonomy $\mu$.
Set
\begin{equation}
\label{Tanaka_comp}
\mathfrak{m}_{-1}(q):=D(q), \quad  \mathfrak{m}_{-j}(q) : =D^{j}(q)/D^{j-1}(q), \,\, \forall j>1
\end{equation}
and  consider the graded space
\begin{equation}
\label{symbdef}
\mathfrak{m}(q)=
\bigoplus_{j=-\mu}^{-1}\mathfrak{m}_j(q),
\end{equation}
associated with the filtration \eqref{filtr}.
The space $\mathfrak{m}(q)$  is endowed with the natural structure of a graded Lie algebra, i.e., with the natural Lie product $[\cdot, \cdot]$ such that
\begin{equation}
 \label{gradcond}
[\mathfrak{m}_{i_1}(q), \mathfrak{m} _{i_2}(q)]\subset \mathfrak{m}_{i_1+i_2}
\end{equation}
defined as follows:
Let $\mathfrak{p}_j:D^j(q)\to \mathfrak m_{-j}(q)$ be the canonical projection to a factor space. Take $Y_1\in\mathfrak m_{-i_1}(q)$ and $Y_2\in \mathfrak m_{-i_2}(q)$.
To define the Lie bracket $[Y_1,Y_2]$ take a local section $\widetilde Y_1$ of the distribution $D^{i_1}$ and
a local section $\widetilde Y_2$ of the distribution $D^{i_2}$ such that
\begin{equation}
\label{projbracket}
\mathfrak p_{i_1}\bigl(\widetilde Y_1(q)\bigr) =Y_1,\quad \mathfrak p_{i_2}\bigl(\widetilde Y_2(q)\bigr)=Y_2.
\end{equation}
It is clear from definitions of the spaces $D^j$ that $[\widetilde Y_1,\widetilde Y_2]\in D^{i_1+i_2}$. Then set
\begin{equation}
\label{Liebrackets}
[Y_1,Y_2]:=\mathfrak p_{i_1+i_2}\bigl([\widetilde Y_1,\widetilde Y_2](q)\bigr).
\end{equation}
It can be  shown (\cite{tan1, zeltan}) that the right hand-side  of \eqref{Liebrackets} does not depend on the choice of sections $\widetilde Y_1$ and
$\widetilde Y_2$. By construction, it is also clear that \eqref{gradcond} holds.
\begin{Def}
The graded Lie algebra $\mathfrak{m}(q)$ from \eqref{symbdef} is called the \emph{Tanaka symbol} of the distribution $D$ at the point $q$.
\end{Def}
By construction, it is clear that the Lie algebra $\mathfrak{m}(q)$ is nilpotent. Further, since $D$ is bracket-generating, its Tanaka symbol $\mathfrak{m}(q)$ at any point is generated by the component $\mathfrak{m}_{-1}(q)$. In this case  one says that the graded Lie algebra $\mathfrak{m}(q)$ is a \emph{step $\mu$ fundamental Lie algebra}. In what follows, we will always work with fundamental negatively graded Lie algebras and omit the words ``fundamental negatively graded'' for brevity. For the step $2$ distributions considered here, the Tanaka symbols are of step $2$.

To a step $2$ Lie algebra $\mathfrak m$, we assign a tuple of integers called its characteristic tuple.
To that end, suppose $\mathfrak{m}$ is a step 2 Lie algebra and $X\in \mathfrak{m}_{-1}$. We view $\ad X:\mathfrak{m}_{-1}\to \mathfrak{m}_{-2}$ as a linear map from $\mathfrak{m}_{-1}$ to $\mathfrak{m}_{-2}$ and assume that
\begin{equation}
\label{max_r}
r=\mathrm{max}\{\rnk \ad \widetilde X: \widetilde X\in \mathfrak m_{-1}\}
\end{equation}
\begin{Def}
\label{ad+max}
We say that $X\in \mathfrak m_{-1}$ is \emph{ad-maximal}, if
$\rnk \ad X=r$, where $r$ is as in \eqref{max_r}.
\end{Def}
Note that the set of ad-maximal elements is (non-empty) Zariski open  in $\mathfrak{m}_{-1}$. For step 2 Lie algebras, the kernel (of the map $\mathfrak{m}_{-1} \to \mathfrak{m}_{-2}$) and image of the map $\ad X: \mathfrak{m}_{-1} \to \mathfrak{m}_{-2}$ for an ad-maximal $X$ are particularly well behaved, as shown in the following:
\begin{Lemma}
    \label{a1}
    Let $\mathfrak{m}=\mathfrak{m}_{-1} \oplus \mathfrak{m}_{-2}$ be a step 2 Lie algebra. If $X \in \mathfrak{m}_{-1}$ is ad-maximal, then $[\ker \ad X , \mathfrak{m}_{-1}]\subset \imad X$ or, equivalently \footnote{Here we use that $\mathfrak m$ is a step 2 Lie algebra.},  $\ker \ad X \oplus \imad X$ is an ideal of $\mathfrak{m}$.
\end{Lemma}
\begin{proof}
    Let $r=\rnk \ad X$. By way of contradiction, suppose that there exists
    \begin{equation}
    \label{Y, Z}
      Y \in \ker \ad X, \quad Z \in \mathfrak{m}_{-1}
    \end{equation}
    such that
    \begin{equation}
    \label{Y, Z, W}
    [Z, Y]=W \in \mathfrak{m}_{-2} \backslash \imad X.
    \end{equation}
    Since $r$ satisfies \eqref{max_r}, it follows that $\rnk \ad (X+tZ)$ is equal to $r$ for sufficiently small $t$. For such $t$, the subspaces $\imad (X+tZ)$ converge to $\imad X$ in $\textbf{Gr}_r(\mathfrak{m}_{-2})$. Thus, by \eqref{Y, Z, W}, we have that $W \notin \imad (X+tZ)$ for sufficiently small  $t$. On the other hand, by \eqref{Y, Z}, we have $[X+tZ,Y]=tW$, i.e. $W \in \imad (X+tZ)$ for $t\neq 0$ -- a contradiction.
\end{proof}
Before we proceed, we introduce new notation that will ease the task of indexing in our constructions. In lieu of developing a complete set of numerical indices for an adapted frame, we instead relate elements of a basis to vertices and edges of a particular graph. As will be seen shortly, virtually all of the relevant information for our purposes can be encoded by a graph derived from ad-maximal elements of $\mathfrak m$ and its quotients by ideals.
To that end, suppose $\mathfrak{m}=\mathfrak{m}_{-1}\oplus \mathfrak m_{-2}$ is a step 2  Lie algebra. We introduce the basic, discrete invariants of this algebra called the \emph{characteristic tuple}:
Begin by setting

\begin{equation}
\label{m0}\mathfrak{m}^{(0)}:=\mathfrak{m}.
\end{equation}
$$\mathfrak{m}^{(0)}_{-j}:=\mathfrak{m}_{-j}, \quad  j \in \{1,2\}$$
and let $m=\dim \mathfrak{m}_{-1}$. Choose a vector $Y_0 \in \mathfrak{m}^{(0)}_{-1}$ such that $\ad Y_0:\mathfrak{m}^{(0)}_{-1} \to \mathfrak{m}^{(0)}_{-2}$ is ad-maximal. It follows from Lemma \ref{a1} and the fact that $\mathfrak m$ is step 2 that $\ker \ad Y_0 \oplus \imad Y_0$ is an ideal of $\mathfrak{m}^{(0)}$. Set
\begin{equation}
\label{M^1}
\mathfrak{m}^{(1)}(Y_0):=\mathfrak{m}^{(0)}/(\ker \ad Y_0 \oplus \imad Y_0)
\end{equation}
and proceed inductively to define the sequence of algebras
\begin{equation}
\label{mj}\mathfrak{m}^{(j+1)}(Y_0,\ldots , Y_j):=\mathfrak{m}^{(j)}(Y_0, \ldots Y_{j-1})/(\ker \ad Y_j \oplus \imad Y_j),\end{equation}
whereby each $\mathfrak{m}^{(j+1)}(Y_0,\ldots , Y_j)$ inherits a grading from its predecessor, namely,
\begin{equation}
\begin{split}
~&\mathfrak{m}^{(j+1)}(Y_0,\ldots , Y_j)=(\mathfrak{m}^{(j)}_{-1}(Y_0,\ldots , Y_{j-1})/\ker\ad Y_j)\oplus(\mathfrak{m}^{(j)}_{-2}(Y_0,\ldots , Y_{j-1})/\imad Y_j)\\
~&=:\mathfrak{m}^{(j+1)}_{-1}(Y_0,\ldots , Y_j)\oplus \mathfrak{m}^{(j+1)}_{-2}(Y_0,\ldots , Y_j),
\end{split}
\end{equation}
and each $Y_j \in \mathfrak{m}^{(j)}_{-1}$ is chosen to be ad-maximal for $\mathfrak m^{(j)}(Y_0,\ldots , Y_{j-1})$.
Let $r_j(Y_0,\ldots, Y_{j-1})$ be the rank of the linear map $\ad Y_j:\mathfrak m^{(j)}_{-1}\rightarrow \mathfrak m^{(j)}_{-2}$:
\begin{equation}
\label{r_j}
r_j(Y_0,\ldots, Y_{j-1})=\mathrm{rank}\,\ad Y_j
\end{equation}
and set $d=\dim \mathfrak{m}_{-2}$. Note that the horizontal dimension of the next quotient is equal to $r_j$: since $\mathfrak{m}^{(j+1)}_{-1}=\mathfrak{m}^{(j)}_{-1}/\ker\ad Y_j$, we have $\dim \mathfrak{m}^{(j+1)}_{-1}(Y_0,\ldots, Y_{j})=r_j(Y_0,\ldots, Y_{j-1})$. If the process has not terminated, i.e., if $\mathfrak{m}^{(j+1)}(Y_0,\ldots, Y_{j})$ is not abelian, then the ad-maximal element $Y_{j+1}$ is nonzero, and $\ad Y_{j+1}$ vanishes on the line spanned by $Y_{j+1}$; hence $r_{j+1}(Y_0,\ldots, Y_{j})\leq r_j(Y_0,\ldots, Y_{j-1})-1$. Otherwise, $r_{j+1}(Y_0,\ldots, Y_{j})=0$. Thus, as long as $r_j(Y_0,\ldots, Y_{j-1})>0$, we have the strict inequality
$$r_{j}(Y_0,\ldots, Y_{j-1})> r_{j+1}(Y_0,\ldots, Y_{j})\geq 0.$$
Moreover, since $\mathfrak m$ is fundamental, there exist $\omega$ (depending on the choice of $Y_0,\ldots, Y_\omega$) such that the sequence of quotient algebras
\begin{equation}
\label{resolution}
\mathfrak{m}=\mathfrak{m}^{(0)} \to \mathfrak{m}^{(1)} (Y_0)\ldots\to \mathfrak{m}^{(\omega)}(Y_0, \ldots, Y_{\omega-1} )\to \mathfrak{m}^{(\omega+1)}(Y_0,\ldots, Y_\omega)
\end{equation}
terminates with the non-zero abelian algebra $\mathfrak{m}^{(\omega +1)}(Y_0,\ldots, Y_\omega)$, or, equivalently, when $$r_{\omega+1}(Y_0,\ldots, Y_\omega)=0.$$This happens  precisely when
\begin{equation}
\label{sum_r}
\sum\limits_{j=0}^{\omega}r_j(Y_0,\ldots, Y_{j-1})=d, \quad  r_j(Y_0, \ldots , Y_{j-1})>0 , \,\,\forall\, 0\leq j\leq \omega.
\end{equation}
In particular, $\omega\leq d-1$ independently of the choice of $Y's$.
Although the choice of tuple or ranks $(r_0,r_1,\ldots,r_{\omega})$ (and even its length $\omega$) depends on the choice of $Y$'s in the above construction, it is constant on a nonempty Zariski open set in the spaces of $Y$'s , and on this set it is equal  to the maximal possible tuple  in the lexicographic order \footnote{To compare two tuples of different lengths lexicographically, we pad the shorter tuple with zeros to match the length of the longer one before performing the comparison.} among all such tuples arising from different choices of ad-maximal $Y$'s. 

Let us justify the last statement in more detail, working on the charts where certain maximal minors do not vanish. On the open set where a chosen maximal minor of $\ad Y_0$ is nonzero, the kernel and the image of this operator have bases depending rationally on the structure constants and on $Y_0$. Consequently, the bracket of the quotient algebra has rational coefficients on that chart. We repeat this construction for each successive maximal minor. Note that at each stage there are only finitely many possible ranks and minors, and that there are at most $d$ stages with nonzero rank. Moreover, the maximal rank attainable at the next stage is attained on a nonempty Zariski-open subset of the preceding parameter chart. Therefore, by induction, we obtain a generic tuple, which is equal to the lexicographically maximal attainable tuple. Equivalently, one may lift all the choices made in the quotients to the horizontal layer of the given algebra and exclude the zeros of the finitely many denominators and witness minors involved. Since nonempty Zariski-open conditions in this parameter space always intersect, the generic tuple is unique.

The same argument works for a smooth family of brackets. Indeed, a nonzero minor witnessing a given rank remains nonzero locally, and the successive ranks are lower semicontinuous \emph{on a stratum where the preceding ranks are fixed}. Thus, the characteristic tuple is lower semicontinuous with respect to the lexicographic order, and it is locally constant on an open dense set. Note that this lower semicontinuity is understood in the lexicographic sense only, and not separately for each entry of the tuple. To see the density explicitly, observe that, for fixed dimensions of the horizontal and central layers, the tuple can take only finitely many values. In any nonempty open set, choose a point at which the tuple is the lexicographically largest among the values attained in this set. Then lower semicontinuity gives a smaller nonempty open set on which this tuple is constant. Therefore, the union of all open sets on which the tuple is constant is open and dense. Finally, for an abelian algebra we use the empty tuple, which is padded by zeros when compared with other tuples.
\begin{Def}
\label{char_tuple_def}
The tuple of ranks on a nonempty  Zariski open set in the space of $Y$'s
(arising from different choices of ad-maximal $Y$'s) is called the \emph{characteristic tuple} of the step 2 Lie algebra $\mathfrak m$.
\end{Def}
Appendix \ref{app:jacobi} gives a geometric interpretation of the characteristic tuple via the Jacobi curves \cite{jac1} of generic normal extremals. Namely, for the flat metric model of the symbol, the characteristic tuple is the growth of the osculating flags of these curves, i.e., the sequence of the lengths of the columns of their Young diagrams \cite{li2} starting from the second column; for an arbitrary sub-Riemannian metric with the same symbol, this growth can be faster, but not slower (Lemma \ref{Jacobi_lem}). This interpretation is not used in the proofs of our main results.
\begin{rk}
\label{ad-surjective_rem}
A step $2$ fundamental graded Lie algebra $\mathfrak m$ is called \emph{$\mathrm{ad}$-surjective} if there exists $X\in \mathfrak m_{-1}$ such that $ \mathrm{ad} X:\mathfrak m_{-1}\to\mathfrak m_{-2}$ is onto. Obviously, if $\mathfrak m$ is $\mathrm{ad}$-surjective, then the characteristic tuple consists of $r_0$ only, i.e. $\omega=0$.
\end{rk}

We now apply the characteristic tuple to the Tanaka symbols $\mathfrak m(q)$ of a step $2$ distribution $D$. The following definition plays only an auxiliary role, but it is very convenient in the sequel.
\begin{Def}
\label{reg_assump}
We say that a step $2$ distribution $D$  satisfies the condition of \emph{constancy of characteristic tuple} at a point $q_0\in M$
if
the characteristic tuple of its Tanaka symbol $\mathfrak m(q)$
is constant for all $q$ in some neighborhood of $q_0$.
\end{Def}
The set of points satisfying Definition \ref{reg_assump} is open and dense, by the local constancy result above; we call it the \emph{regular locus of the characteristic tuple}. On each sufficiently small neighborhood in this locus, the graph constructions below use a fixed characteristic tuple. The graph calculations in the proof of the main theorems are carried out on this locus. Proposition \ref{tuple_continuation_prop} then extends the bracket relations obtained from them to the whole stable neighborhood, and the rest of the proof works near every stable point.

\subsection{Normalizable graphs of step 2 fundamental Lie algebras and compatible frames}
\label{frame_graph_subsec}
Now we return to constructions with a step 2 Lie algebra $\mathfrak m$.
In the sequel, we assume that the sequence of ad-maximal Y's is chosen such that the tuple of ranks defined by \eqref{r_j} forms the characteristic tuple of the algebra $\mathfrak m$, and, for brevity, we denote the corresponding algebras $\mathfrak m^{(j)} (Y_0,\ldots Y_{j-1})$  from \eqref{resolution} by $\mathfrak m^{(j)}$.
The sequence of quotient algebras $\{\mathfrak m^{(j)}\}_{j=0}^{\omega+1}$ naturally defines a filtration on $\mathfrak{m}$ by ideals. Indeed, if
\begin{equation}
\label{can_projection}
\varphi_j:\mathfrak{m}^{(j)}\to \mathfrak{m}^{(j+1)}
\end{equation}
denotes the canonical projection to a quotient space, then $\varphi^{-1}_{j}(\ker \ad Y_{j+1} \oplus \imad Y_{j+1})$ is an ideal of $\mathfrak{m}^{(j)}$.
Set
\begin{equation}
\label{h0}
\mathfrak h_0:=\ker \ad Y_0 \oplus \imad Y_0
\end{equation}
$\mathfrak{h}_{\omega+1}=\mathfrak m$, and let
\begin{equation}
    \label{hj}
    \mathfrak{h}_{j}=
(\varphi_{j-1}\circ\ldots\circ \varphi_0)^{-1}(\ker \ad Y_j \oplus \imad Y_j)
\subset \mathfrak{m}, \quad 1\leq j\leq \omega.
\end{equation}
Note that each $\mathfrak{h}_j$ is an ideal of $\mathfrak{m}$, and such a collection defines a filtration
\begin{equation}
\label{filt}
\mathfrak{h}_0\subset \mathfrak{h}_1 \ldots \subset \mathfrak{h}_{\omega}\subset\mathfrak{h}_{\omega+1}=\mathfrak{m}
\end{equation}
of the algebra $\mathfrak{m}$ by the ideals $\{\mathfrak{h}_j\}_{j=0}^{\omega +1}$. We say that this filtration is \emph{associated with the tuple $\{Y_j\}_{j=0}^\omega$}.
We wish to assign a class of bases to the algebra $\mathfrak m$ which are compatible with the filtration \eqref{filt}. As mentioned above, it is convenient to index the elements of such a basis by vertices and edges of a certain graph.
In more detail, let $\Gamma$ be a simple directed graph with $m$ vertices and $d$ edges which has no directed cycles. We denote by $V(\Gamma)$ and $E(\Gamma)$ the sets of all vertices and all edges of the graph $\Gamma$, respectively. The symbol $\overrightarrow{ab}$ represents the directed edge which connects the initial vertex $a$ to the terminal vertex $b$. For a given vertex $a \in V(\Gamma)$, we denote by $\mathrm{Out(a)}$ the set of all edges of the form $\overrightarrow{ab}$ such that $\overrightarrow{ab}\in E(\Gamma)$ and by $\mathrm{In(a)}$ the set of all edges of the form $\overrightarrow{ba}$ such that $\overrightarrow{ba}\in E(\Gamma)$. We call $|\mathrm{Out}(a)|$ and $|\mathrm{In}(a)|$ the \emph{outgoing index} and \emph{incoming index}, respectively, and their sum the \emph{total index} of the vertex $a$.
\begin{Def}
 If $\mathfrak{m}=\mathfrak m_{-1} \oplus \mathfrak m_{-2}$ is a step 2 Lie algebra, then we say $\mathfrak m$ and the graph $\Gamma$ are \emph{compatible} if there exists a basis
$\{X_v\}_{v\in V(\Gamma)}$ of $\mathfrak m_{-1}$  and a basis
$\{X_e\}_{e\in E(\Gamma)}$ of $\mathfrak m_{-2}$ such that, for any $e \in E(\Gamma)$ with $e=\overrightarrow{v_1v_2}$, we have
\begin{equation}
\label{edges_vertices}
X_{\overrightarrow{v_1v_2}}:=[X_{v_1}, X_{v_2}
].
\end{equation}
In this case, the basis
$\{X_s\}_{s\in \{V(\Gamma)\cup E(\Gamma)\}}$ of $\mathfrak m$ is also said to be \emph{compatible} with the graph $\Gamma$.
\end{Def}
\begin{Def}
\label{def_type_r}
Let $\Gamma$ be a graph with $m$ vertices. If there exists $\rho \in V(\Gamma)$ with outgoing index $r>0$ and $\Gamma$ has $m-r-1$ isolated vertices (i.e., vertices with total index $0$), then we say that the graph $\Gamma$ is \emph{normal of degree $r$} with \emph{root}  $\rho$.
\end{Def}
\begin{rk}
    Note that the root $\rho$ in Definition \ref{def_type_r} is unique and necessarily has incoming index $0$ since $\Gamma$ has no directed cycles.
\end{rk}
\begin{Def}
\label{reduce}
If $\Gamma$ is normal and $\rho \in V(\Gamma)$ a root, then the \emph{reduction  of $\Gamma$ by $\rho$} is the graph $R_{\rho}(\Gamma)$ obtained from $\Gamma$ by the deletion of the vertex $\rho$, the edges $\mathrm{Out}(\rho)$, and the isolated vertices of $\Gamma$.
\end{Def}
\begin{Def}
\label{creduce}
    The graph $\Gamma$ is called \emph{normalizable} if there exists vertices $\rho_0, \rho_1, \ldots, \rho_{\omega}$ of $\Gamma$ and positive integers $(r_0, r_1,\ldots r_\omega)$ such that the following hold:
\begin{enumerate}
    \item The graph $\Gamma$ is normal of degree $r_0$ with root $\rho_0$.
    \item For $1\leq j \leq\omega $, the graph $R_{\rho_{j-1}}\circ R_{\rho_{j-2}}\circ \cdots\circ R_{\rho_0}(\Gamma)$, obtained from $\Gamma$ by successive
    reductions $R_{\rho_0}$,\ldots , $R_{\rho_{j-1}}$, is normal of degree $r_j$  with root $\rho_j$.
    \item The set of edges of the graph
    $R_{\rho_{\omega}}\circ R_{\rho_{\omega-1}}\circ \cdots \circ  R_{\rho_0}(\Gamma)$
    is empty.
\end{enumerate}
\end{Def}
\begin{Def}
\label{root_def}
If a graph $\Gamma$ is normalizable, then we call the numbers $(r_0, r_1,\ldots r_\omega)$ the \emph{characteristic tuple} of $\Gamma$.
The vertex $\rho_j$ from item (2) of Definition \ref{creduce}  is called the \emph{$j$-th root} of $\Gamma$ and the tuple $\{\rho_j\}_{j=0}^\omega$ is called the \emph{tuple of roots} of $\Gamma$.
\end{Def}
In the sequel, we will assign compatible graphs to all summands of the decomposition of the Tanaka symbol obtained in the next subsection, and some of these summands may be abelian. For an abelian summand, the compatible graph has no edges and consists of isolated vertices only, their number being equal to the dimension $\dim H$ of the horizontal part of the summand. Its characteristic tuple and its tuple of roots are empty, and every vertex has depth $0$. Accordingly, we use the standard conventions for empty products and empty matrices: products over its edges or over its roots are equal to $1$, and its core matrix, which will appear in the analysis of the fundamental algebraic system, is the empty matrix of determinant $1$. Statements that mention a root are understood to concern only nonabelian summands. Finally, note that the out-star graphs in Remark \ref{constructuve_graph_rem} below form a union in which different out-stars may share vertices. Consequently, a normalizable graph may contain undirected cycles: among cycles, only the directed ones are excluded by its definition.

\begin{rk}
\label{constructuve_graph_rem}
  One can give a direct equivalent definition of a normalizable graph with characteristic tuple $(r_0, r_1, \ldots, r_\omega)$ without using the reduction operations:   It is a directed graph along with  the distinguished ordered tuples of vertices $(\rho_0,\ldots, \rho_{\omega})$, called roots, such that:
\begin{enumerate}
    \item The graph  $\Gamma(\mathfrak m)$ is the union of $\dim \mathfrak m_{-1}-r_0-1$ isolated vertices and $\omega +1$ out-star graphs such that the central vertex  of the $i$th out-star graph is $\rho_i$ and has outgoing degree $r_i$ for all $0\leq i\leq \omega$;
    \item $\overrightarrow{\rho_i \rho_{i+1}}$ is an edge of $\Gamma$ for all $0\leq i<\omega$.
    \item For any vertex $v$ if $\overrightarrow{\rho_i v}$ is an edge of $\Gamma(\mathfrak m)$,  then $\overrightarrow{\rho_j v}$ is also an edge of $\Gamma(\mathfrak m)$ for every $0\leq j<i$.
\end{enumerate}
\end{rk}
Let $I_0(\Gamma) \subset V(\Gamma)$ denote the set of vertices of $\Gamma$ which are deleted from $\Gamma$ by $R_{\rho_0}$. Similarly, for $1\leq j\leq \omega$, let  $I_j(\Gamma) \subset V(\Gamma)$ denote the set of vertices which are deleted from $R_{\rho_{j-1}}\circ R_{\rho_{j-2}}\circ \cdots\circ R_{\rho_0}(\Gamma)$ to obtain $R_{\rho_{j}}\circ R_{\rho_{j-1}}\circ \cdots\circ R_{\rho_0}(\Gamma)$, and, finally, let $I_{\omega+1}$ be the set of vertices of $R_{\rho_{\omega}}\circ R_{\rho_{\omega-1}}\cdots \circ  R_{\rho_0}(\Gamma)$. Note that
\begin{equation}
\label{union}
 \bigcup \limits_{j=0}^{\omega+1}I_j(\Gamma)=V(\Gamma).
\end{equation}
\begin{Def}
\label{vertex_depth} We say that a vertex of a normalizable graph $\Gamma$ belonging to $I_j(\Gamma)$ is a vertex of \emph{depth $j$}. We denote by $\mathrm{depth}(v)$ the depth of a vertex $v$ in $\Gamma$.
\end{Def}
\begin{rk}
\label{r_depth_rem}
By construction, the number $r_p$ from the characteristic tuple is equal to the number of vertices of depth greater than $p$.
\end{rk}
Using the constructions above, we may now state the main result of this section.
\begin{theorem}
\label{ccreduce}
    If $\mathfrak m$ is a nonabelian step 2 Lie algebra with characteristic tuple $(r_0,\ldots,r_\omega)$, then $\mathfrak m$ is compatible with a normalizable graph $\Gamma$ with characteristic tuple $(r_0,\ldots, r_\omega)$. Moreover, a basis $\{X_s\}_{s\in \{V(\Gamma)\cup E(\Gamma)\}}$, compatible with $\Gamma$, can be chosen such that
   \begin{equation}
    \mathrm{span}\left \{X_s:s\in \bigcup \limits_{k=0}^{j}\bigl(I_k(\Gamma)\cup \mathrm{Out}(\rho_k)\bigr)\right\}=\mathfrak h_j, \quad 0\leq j\leq \omega,
   \end{equation}
    where the filtration $\{\mathfrak h_j\}_{j=0}^{\omega+1}$ of $\mathfrak m$ is associated with the tuple  $\left\{(\varphi_{j-1}\circ\ldots\circ \varphi_0)(X_{\rho_j})\right\}_{j=0}^\omega$.
    Furthermore, if $\mathfrak m_{-1}$ is endowed with an inner product, then the basis $\{X_s\}_{s\in \{V(\Gamma)\cup E(\Gamma)\}}$ can be chosen such that $\{X_v\}_{v \in V(\Gamma)}$ is an orthonormal basis of $\mathfrak m_{-1}$.
\end{theorem}
\begin{proof} Suppose that $\mathfrak m_{-1}$ is endowed with an inner product and consider the ad-maximal tuple $\{Y_j\}_{j=0}^\omega$ as above.
First, we normalize $Y_0$ to have unit length (note that this leaves its kernel and image unchanged), set $X_{\rho_0}=Y_0$, and complete $\{X_{\rho_0}\}$ to an orthonormal basis $\{X_{v}\}_{v \in I_0(\Gamma)}$ of $\ker \ad X_{\rho_0}$.
There is an isomorphism $$\psi_{0}:(\ker \ad X_{\rho_0})^{\perp} \to \mathfrak{m}^{(1)}_{-1}$$
such that $\varphi_0\circ\psi_0^{-1}=\mathrm{id}|_{\mathfrak{m}^{(1)}_{-1}}$. The isomorphism $\psi _0$ induces an inner product on $\mathfrak{m}_{-1}^{(1)}$
(i.e., the inner product which makes $\psi_0$ an isometry). Thus, we have a splitting
$$\mathfrak{m}^{(1)}_{-1}=\langle Y_1\rangle\oplus \bigl(\langle Y_1 \rangle ^{\perp} \cap \ker\ad Y_{1}\bigr)\oplus( \ker \ad Y_1)^{\perp},$$
where $\perp$ denotes the orthogonal complement with respect to the induced inner product on $\mathfrak{m}^{1}_{-1}$.
Next, we normalize $Y_1$ with respect to this induced inner product, complete $\{Y_1\}$ to an orthonormal basis of $\ker \ad Y_1$, and let $\{X_{v}\}_{v \in I_1(\Gamma)}$ denote the preimage of this basis under $\psi_0$, indexed in such a way that $X_{\rho_1}=\psi^{-1}_0(Y_1)$. In the same way, at each subsequent step we normalize the chosen ad-maximal vector with respect to the inner product induced on the corresponding quotient. Proceeding inductively, we construct an orthonormal basis of $\mathfrak{m}_{-1}$ indexed by $V(\Gamma)$ such that the following hold:
\begin{enumerate}
    \item The vectors $\{X_{\rho_j}\}_{j=0}^{\omega}$ are the preimages in $\mathfrak{m}_{-1}$ of the ad-maximal elements $\{{Y_j}\}_{j=0}^{\omega}$ under the respective isomorphisms $\{\psi_k\}$ and are indexed by the roots $\rho_0, \rho_1, \ldots, \rho_{\omega} \in V(\Gamma)$(we also refer to such $X_{\rho_j}$ as ad-maximal elements of $\mathfrak{m}^{(j)}$).
    
    \item For $1 \leq j \leq \omega$, the vectors $\{X_{v}\}_{v \in I_j(\Gamma)}$ are the preimages in $\mathfrak{m}_{-1}$ of the chosen orthonormal basis of $\ker\ad Y_j$ under the respective isomorphisms $\{\psi_k\}$.
    \item The vectors $\{X_{v}\}_{v \in I_{\omega+1}(\Gamma)}$ are the preimages of the chosen orthonormal basis of $\mathfrak{m}^{\omega+1}_{-1}$ under the respective isomorphisms $\{\psi_k\}$
\end{enumerate}
For each root $\rho_i$ of $\Gamma$,
if $\overrightarrow{\rho_i,v}\in \mathrm{Out}(\rho_i)$, then we associate the vector
\begin{equation}
\label{out}
   X_{\overrightarrow{\rho_i,v}}=[X_{\rho_i},X_{v}]
\end{equation}
to the edge $\overrightarrow{\rho_i,v}$. Note that the vectors $X_e$ with $e \in\bigcup\limits_{j=0}^\omega\mathrm{Out}(\rho_j)$ form a basis of $\mathfrak{m}_{-2}$ by construction. In this way, the construction of the required basis $\{X_s\}_{s\in \{V(\Gamma)\cup E(\Gamma)\}}$ is completed.
\end{proof}
The proof of Theorem \ref{ccreduce} shows that any step 2 Lie algebra $\mathfrak{m}$ is compatible with the normalizable graph $\Gamma(\mathfrak m)$ which is uniquely determined, up to isomorphism, by its characteristic tuple together with $\dim\mathfrak m_{-1}$.
\begin{Def}
\label{graph_m}
The normalizable graph $\Gamma(\mathfrak m)$ is called the normalizable graph of the step 2 algebra $\mathfrak m$.
\end{Def}
When there are multiple graphs indexed by a certain indexing set $\mathfrak{I}$, say $\{\Gamma_i\}_{i \in \mathfrak{I}}$, we will label the roots of a given indexed graph $\Gamma_i$ by $\rt 0, \ldots, \rt {\omega_i}$ to distinguish its roots from those of other graphs in the collection.
Notice that the above construction implies that
\begin{equation}
\label{decrease}
\ad X_{\rho_j}(\mathfrak h_j) \subset \mathfrak{h}_{j-1}.
\end{equation}
for any root $\rho_j$ (here we set $\mathfrak h_{-1}=0$).
\begin{rk}
\label{ad_rem1}
Let $\{X_s\}_{s\in \{V(\Gamma)\cup E(\Gamma)\}}$ be a basis of $\mathfrak m $ as in Theorem \ref{ccreduce}. If $v$ is a vertex of depth $p$ in $\Gamma$, then, from the filtration \eqref{filt} by ideals and the definition of $\mathfrak h_j$ as in \eqref{hj}, it follows that
\begin{equation}
\label{ad_X_depth_p}
\imad X_v\subset \sum_{j=0}^{\min(p,\omega)} \imad X_{\rho_j}.
\end{equation}
\end{rk}
\subsection{Decomposition of Tanaka symbol through the eigenspace decomposition of the transition operator}
\label{decomposition_Tanaka_sec}
The important initial step toward the proof of Theorem \ref{mainthm} is the following
\begin{theorem}[Consequence of {\cite[Theorem 7.1]{jmz2019}}]
\label{Tanaka_decomp_prop}
Let $D$ be a distribution of step $2$, and let $g_1,g_2$ be sub-Riemannian metrics on it that are projectively equivalent in a neighborhood of a stable point. Denote by $W_i(q)=[D_i(q),D_i(q)]$ the image of the Levi bracket in $T_qM/D(q)$. Then, at every point of this neighborhood, the Tanaka symbol decomposes as
\begin{equation}\label{eigen_decomp}
\mathfrak m(q)=\bigoplus_{i=1}^k\mathfrak m^i(q),\qquad
\mathfrak m^i(q)=D_i(q)\oplus W_i(q).
\end{equation}
The summands of this decomposition are pairwise commuting graded ideals, some of which may be abelian, and the ranks of the $W_i$ are locally constant. In particular,
\begin{equation}\label{m-1i}\mathfrak m^i_{-1}=D_i,\qquad
W_i\cong D_i^2/(D_i^2\cap D).
\end{equation}
\end{theorem}
\begin{proof}
If $k=1$, the assertion is immediate, since in this case there is only one eigenspace. Now assume that $k>1$. Then the cited theorem gives the product decomposition of the Tanaka symbol on an open dense subset of the neighborhood under consideration. Grouping together the factors of this decomposition that correspond to the same eigenvalue of the transition operator, we obtain a decomposition with horizontal factors $D_i$. Consequently, $[D_i,D_j]=0$ in the symbol for $i\ne j$ at every point of this dense set, and hence everywhere, since the Levi bracket depends smoothly on the point. Therefore, the spaces $W_i$ span $TM/D$ at every point. Note that their dimensions $d_i$ are lower semicontinuous functions, because they are the ranks of the smooth bundle maps $\Lambda^2D_i\to TM/D$ induced by the Levi bracket.

Let us show that the sum of these spaces is direct at every point. Fix a point $q$ of the neighborhood. By lower semicontinuity, we can choose points $q'$ of the dense set, arbitrarily close to this point, such that all the inequalities $d_i(q')\ge d_i(q)$ hold. Then
\[
n-m=\sum_i d_i(q')\ge\sum_i d_i(q)\ge n-m.
\]
Here the first equality holds because the decomposition is direct on the dense set, and the last inequality holds because the spaces span the quotient at every point. Hence equality holds throughout, which proves that the sum is direct at $q$. Consequently, $\sum_i d_i=n-m$ everywhere. Since each term of this sum is an integer-valued lower semicontinuous function and the sum is constant, every $d_i$ is locally constant. This proves both that the decomposition extends from the dense set to the whole neighborhood and that the central subbundles are smooth, since they are the images of smooth bundle maps of constant rank.
\end{proof}

It is important to note that $\mathfrak{m}^i(q)$ is not \emph{a priori} equivalent to the Tanaka symbol of $D_i$ at $q$ since they differ in their definition of $-2$ components (though, it will be shown later that they necessarily coincide).
\begin{lemma}\label{char_summands}
If the characteristic tuple of $D$ is constant near $q_0$, then the characteristic tuple of each summand in the decomposition \eqref{eigen_decomp} is also constant near $q_0$.
\end{lemma}
\begin{proof}
First observe that, for a direct sum of Lie algebras, the adjoint maps, as well as their kernels and their images, split according to the summands. Applying this observation inductively to the successive quotient algebras, we obtain the following componentwise identity between the entries of the characteristic tuple of the Tanaka symbol and those of its summands (here we complete the characteristic tuple of a summand by zeros if necessary):
\begin{equation}\label{char_sum}r_j=\sum_i r_j^i.
\end{equation}
Let us now prove the local constancy of these entries by induction on the stage of the construction of the characteristic tuple. At the first stage, each $r_0^i$ is an integer-valued lower semicontinuous function, while the sum $r_0$ is constant by assumption; hence each $r_0^i$ is locally constant. Now suppose that the ranks at all the preceding stages have already been fixed, i.e., they are known to be locally constant. Choose, at $q_0$, tuples of successive ad-maximal elements attaining the generic ranks in all summands. The corresponding nonzero witness minors remain nonzero near $q_0$, and therefore the kernels and the quotients vary smoothly there. Consequently, on this stratum, where the preceding ranks are fixed, the next ranks $r_j^i$ are lower semicontinuous. Since their sum in \eqref{char_sum} is constant, they are again locally constant. Note that, in this argument, the lower semicontinuity of the ranks at a given stage is used only after the ranks at all the preceding stages have been fixed, that is, only on a stratum where the preceding ranks are constant; this is precisely the setting in which the lower semicontinuity of successive ranks was established when the characteristic tuple was introduced. Finally, since the number of stages is finite, the induction terminates, which completes the proof.
\end{proof}
\subsection{Quasi-normal frames}
\label{quasinormal_frames_subsec}
We begin by specifying the setting in which quasi-normal frames will be used. Let $\Omega$ be the open dense subset of a stable neighborhood on which the characteristic tuple of $D$ is locally constant. In the graph calculations below, we assume that the reference point belongs to $\Omega$, and we shrink the neighborhood under consideration so that it is contained in $\Omega$. Then, by Lemma \ref{char_summands}, the characteristic tuples of all summands of the Tanaka symbol are locally constant on that neighborhood. Note that the compatible graphs used on different neighborhoods may be different. In Subsection \ref{tuple_continuation_subsec}, the bracket relations obtained in this way are extended to the whole stable neighborhood; this extension does not require extending these graphs or the corresponding frames. On the other hand, recall that the splitting of the Tanaka symbol given by Theorem \ref{Tanaka_decomp_prop} holds on the entire stable neighborhood, and not only on this open dense subset.

Now, let $\Gamma_i$ be the normalizable graph of $\mathfrak m^i(q)$, i.e. $\Gamma_i:=\Gamma\bigl(\mathfrak m^i(q)\bigr)$ and let
\begin{equation}
\label{graph_decomp}
\Gamma=\bigsqcup_{i=1}^k \Gamma_i.
\end{equation}
\begin{Def}
\label{component_def}
We say that the graph $\Gamma_i$ is the \emph{$i$th component of the graph $\Gamma$} with respect to decomposition \eqref{eigen_decomp}. \footnote{ Note that the notion of component here is completely different from a connected component because the graphs $\Gamma_i$ are not connected in general: they may have isolated vertices, all of which are vertices of depth $0$.}.
\end{Def}
Let us now construct frames adapted to this graph near the reference point. By Lemma \ref{char_summands}, the characteristic tuples of all summands are constant near this point. Choose, at $q_0$, maximizing root tuples for all summands. Passing to a smaller neighborhood, we may assume that their nonzero rank minors remain nonzero on it. Then, by the constant-rank theorem, the successive kernels, their orthogonal complements, and the quotient metrics depend smoothly on the point, and a smooth orthonormal completion at each step gives the horizontal vectors of Theorem \ref{ccreduce} as smooth vector fields. Further, we define each edge field to be the actual Lie bracket of the two vertex fields corresponding to the endpoints of the edge. The classes of these edge fields modulo the distribution form the required bases of the central parts of the summands. Thus, all these vector fields together form a smooth frame $\{X_s\}_{s\in V(\Gamma)\cup E(\Gamma)}$ in $\mathcal U$ such that if
\begin{equation} \label{pr_def} \mathrm {pr}_q:T_qM\rightarrow T_qM/D(q)
\end{equation} is the canonical projection, then
for every $1\leq i\leq k$ and $q\in \mathcal U$ the tuple
 $$ \left\{\{X_v(q)\}_{v\in V(\Gamma_i)},\{\mathrm{pr}_q\bigl(X_e(q)\bigr)\}_{e\in E(\Gamma_i)}\right\}$$
is a compatible basis of $\mathfrak m^i(q)$ with respect to the graph $\Gamma_i$ and
\begin{equation}
\label{quasinorm_basis}
[X_v,X_w]=X_{\overrightarrow{vw}}, \quad \text { for every }  v, w \in V(\Gamma) \text { such that } \overrightarrow{vw} \in E(\Gamma).
\end{equation}
We call such frame $\{X_s\}_{s\in V(\Gamma)\cup E(\Gamma)}$ a \emph{quasi-normal frame adapted to the pair of sub-Riemannian metrics $g_1$ and $g_2$,} or, shortly a \emph{quasi-normal frame}.
In what follows, we use the shorthand
\begin{equation}
\label{graph_translation}
V=V(\Gamma), \qquad V_i=V(\Gamma_i), \qquad E=E(\Gamma), \qquad E_i=E(\Gamma_i).
\end{equation}
Thus the arbitrary index sets used in Subsection \ref{fas_subsec} are, from this point onward, identified with the vertex and edge sets of the normalizable graphs. We also extend the function $\alpha$ from \eqref{alpha_of} to $E$ by setting
\begin{equation}
\label{alpha_edge_extension}
\alpha(e):=\alpha_i, \qquad e\in E_i,\quad 1\leq i\leq k.
\end{equation}
In the sequel, all graph calculations on $\Omega$ will be carried out in quasi-normal frames.
The structure functions of this frame are those defined in \eqref{struct_function}, with the identification \eqref{graph_translation}. In particular, from \eqref{quasinorm_basis} it follows that
\begin{equation}
\label{struct_quasinorm}
\text {if } \overrightarrow {vw}\in E, \text{ then  }
c_{v,w}^e=\begin{cases}
1& \text {if }e=\overrightarrow{vw},\\
0& \text {if }e\in V\cup \bigl(E\backslash\{ \overrightarrow{vw}\}\Bigr)
\end{cases}
\end{equation}
\begin{rk}
\label{ad_rem2}
Let $\{X_s\}_{s\in \{V\cup E\}}$ be a quasi-normal frame. By Remark \ref{ad_rem1} (especially relation \eqref{ad_X_depth_p}) and the decomposition \eqref{eigen_decomp} of the Tanaka symbol $\mathfrak m$, if $v$ is a vertex of depth $p$ in $\Gamma_i$, it follows that
\begin{equation}
\label{ad_X_depth_p_i}
\imad X_v\subset \sum_{j=0}^{\min(p,\omega_i)} \imad X_{\rho_j^i}|_{D_i} \quad \mathrm{mod} D.
\end{equation}
\end{rk}
\begin{rk}
\label{block_q}
From the decomposition \eqref{eigen_decomp}, the quasi-normality of the frame, and the formula \eqref{qjk}, it follows that if $w\in V_i$ and $l\in E_j$ for $i\neq j$, then $q_{wl}=0$. In fact, later we will prove a more general statement regarding higher layers; see Corollary \ref{block_diag}, Item (1).
\end{rk}
\section{Sketch of the remaining proof of the main Theorem \ref{mainthm}}
\label{proof_outline_sec}
We begin by fixing the terminology for the structure functions of quasi-normal frames (see Subsection \ref{quasinormal_frames_subsec}), which will be used in the rest of the paper.
\begin{Def}
\label{cross_component_def}
Let $\{X_s\}_{s\in V\cup E}$ be a quasi-normal frame, and let $c^z_{xy}$ be one of its structure functions, i.e., the coefficient of $X_z$ in the Lie bracket $[X_x,X_y]$ (see \eqref{struct_function}).
\begin{enumerate}
\item The lower indices $x$ and $y$ are called the \emph{inputs} of $c^z_{xy}$, and the upper index $z$ is called its \emph{output}.
\item An input is called \emph{horizontal} if it is a vertex, and an \emph{edge input} if it is an edge. In the latter case, the corresponding vector field is the Lie bracket of the vector fields corresponding to the endpoints of this edge (see \eqref{quasinorm_basis}).
\item The output is called \emph{horizontal} if it is a vertex, and \emph{central} if it is an edge. In the latter case, $c^z_{xy}$ is the coefficient of the Lie bracket along a vector field whose class modulo $D$ belongs to the degree $-2$ (central) part of the Tanaka symbol.
\item Every vertex and every edge of $\Gamma$ belongs to exactly one component $\Gamma_i$ (see Definition \ref{component_def}), and we say that the corresponding index belongs to this component. The structure function $c^z_{xy}$ is called \emph{cross-component} if its three indices do not all belong to the same component. It is called \emph{horizontal cross-component} if, in addition, its inputs and its output are horizontal, i.e., all its indices are vertices.
\end{enumerate}
\end{Def}
We now give a synopsis of the remaining argument.  Only the forward implication of
Theorem \ref{mainthm} requires proof: the converse, namely that every generalized
Levi-Civita pair is projectively equivalent, was established in
\cite[Proposition 5.5]{jmz2019}. By Proposition \ref{th:orb.to.proj}, projective equivalence gives an orbital diffeomorphism between the corresponding Hamiltonian flows, and Proposition \ref{fda} turns the existence of this diffeomorphism into solvability of the fundamental algebraic system \eqref{A.phi.B}.  Since this system is overdetermined and polynomial in the fiber variables, all maximal minors of its augmented matrix vanish.
The goal is to extract from these determinant identities the following sparse form of the cross-component structure functions of a quasi-normal frame.
\begin{Prop}[Vanishing of cross-component structure functions]
\label{outline_transverse_vanishing}

Under the assumptions of Theorem \ref{mainthm}, suppose that $g_1$ and $g_2$ are
projectively equivalent, and let
$\{X_s\}_{s\in V\cup E}$ be a quasi-normal frame.  Then the
following structure functions vanish:
\begin{enumerate}
    \item if $i\neq t$, $v\in V_i$, $w\in V_t$, and $e\in E$, then
    \begin{equation}
    \label{outline_mixed_vertical_zero}
    c^e_{vw}=0;
    \end{equation}
    \item if $i\neq t$, $v,w\in V_i$, and $l\in V_t$, then
    \begin{equation}
    \label{outline_same_component_horizontal_zero}
    c^l_{vw}=0;
    \end{equation}
    \item if $i,t,s$ are pairwise distinct, $v\in V_i$, $w\in V_t$, and
    $l\in V_s$, then
    \begin{equation}
    \label{outline_three_component_zero}
    c^l_{vw}=0;
    \end{equation}
    \item if $i\neq t$, $y\in V_i$, $e\in E_i$, and
    $z\in V_t\cup E_t$, then
    \begin{equation}
    \label{outline_vertex_edge_zero}
    c^z_{ye}=0.
    \end{equation}
\end{enumerate}
\end{Prop}
The first item is an immediate consequence of the direct-sum decomposition
\eqref{eigen_decomp} of the Tanaka symbol: brackets between the degree $-1$
elements of different summands have no degree $-2$ component.  The other items are
the results of the analysis of certain special minors in the subsequent sections. In the terminology of Definition \ref{cross_component_def}, item (1) concerns the cross-component structure functions with horizontal inputs from two different components and a central output; items (2) and (3) concern the horizontal cross-component structure functions whose output belongs to a component different from the components of both inputs; and item (4) concerns the cross-component structure functions with a horizontal input and an edge input from the same component and with an output, central or horizontal, from another component. Note that a horizontal cross-component structure function whose inputs belong to two different components and whose output belongs to one of these components does not vanish in general (see Remark \ref{bracket_frames_rk}). The places where these vanishing families are proved, together with their
geometric roles, are summarized in the following table, in which item (4) occupies the last two rows.

\begin{center}\small
\begin{tabularx}{\textwidth}{@{}>{\raggedright\arraybackslash}X>{\raggedright\arraybackslash}X>{\raggedright\arraybackslash}X@{}}
\toprule
Vanishing family & Proved in & Geometric role\\
\midrule
Horizontal inputs from two different components, central output & Theorem \ref{Tanaka_decomp_prop} & The mixed brackets satisfy $[D_i,D_t]\subset D$.\\
Both horizontal inputs in $V_i$, horizontal output in $V_t$, $t\ne i$ & Propositions \ref{cofterm2}, \ref{c1bis}, and \ref{c2} & This gives $D_i^2\cap D=D_i$.\\
Horizontal inputs and horizontal output in three distinct components & Proposition \ref{c3} & Together with the first row, this gives $[D_i,D_t]\subset D_i\oplus D_t$.\\
Horizontal input in $V_i$, edge input in $E_i$, central output in $E_t$, $t\ne i$ & Lemma \ref{bianchi} and the subsequent Jacobi identity argument for mixed horizontal indices & The bracket $[D_i,D_i^2]$ has no central component belonging to another summand of the symbol.\\
Horizontal input in $V_i$, edge input in $E_i$, horizontal output in $V_t$, $t\ne i$ & Proposition \ref{mixed_edge_horizontal_vanishing} & Together with the preceding row, this gives $[D_i,D_i^2]\subset D_i^2$.\\
\bottomrule
\end{tabularx}\end{center}

The geometric content of Proposition \ref{outline_transverse_vanishing} is the
existence of the foliations that occur in Definition \ref{Gen_Levi-Civita}.
\begin{Coro}[The foliation net]
\label{outline_foliation_net}
Under the assumptions of Proposition \ref{outline_transverse_vanishing},
\begin{equation}
\label{outline_tangent_decomposition}
TM=\bigoplus_{i=1}^k D_i^2,
\end{equation}
and, for every nonempty $J\subset\{1,\ldots,k\}$, the distribution
$\bigoplus_{i\in J}D_i^2$ is involutive.  Consequently, the integral foliations
$\mathcal F_i$ of the distributions $D_i^2$ form a local net of foliations.  In
particular,
\begin{equation}
\label{outline_leaf_tangent}
T\mathcal F_i=D_i^2,
\end{equation}
and $D_i$ is bracket generating on every leaf of $\mathcal F_i$.
\end{Coro}
\begin{proof}
The vanishing statements above imply
\[
[D_t,D_i]\subset D_i\oplus D_t\quad (i\neq t),
\qquad
[D_i,D_i^2]\subset D_i^2.
\]
Proposition \ref{ady} gives the necessary invariance of the distributions under the adjoint action, and an application of the
Jacobi-identity gives the involutivity of $D_i^2$ and
$D_i^2\oplus D_t^2$, as stated in Corollary \ref{involutive}.  Expanding brackets
of sums then proves the involutivity of every partial sum.  Item (2) of
Proposition \ref{outline_transverse_vanishing} gives $D_i^2\cap D=D_i$; combined
with \eqref{eigen_decomp} and the step $2$ assumption, this yields the direct sum
\eqref{outline_tangent_decomposition}.  The Frobenius theorem now produces
mutually transverse local foliations and local product coordinates adapted to
them.  Finally, $D_i^2=D_i+[D_i,D_i]$, so $D_i$ is bracket generating within each
leaf.
\end{proof}
It remains to recover the metric normal form, which is the purpose of
Section \ref{rotation_sec}.  Since quasi-normal frames are available only on the regular
locus of the characteristic tuple, we first show in Proposition
\ref{tuple_continuation_prop} that the conclusions of Corollary
\ref{outline_foliation_net} hold on the whole stable neighborhood: each of the
underlying bracket relations expresses the vanishing of a smooth tensor, which
vanishes on the dense regular locus and hence everywhere.  The local product determined by the foliations
defines projections $\pi_i$ along the complementary leaves.  Lemma
\ref{rotation_section_lemma} shows that the $D_i^2$-lifts of a weakly adapted
frame under $\pi_i$ actually lie in $D_i$, thereby establishing the projection
condition in Definition \ref{Gen_Levi-Civita}; Lemma
\ref{rotation_commute_lemma} shows that lifts belonging to different components
commute.  The resulting differential equations for the change of frame matrices
(Lemma \ref{rotation_transition_pde}), together with the first-divisibility
identities in Proposition \ref{first divi}, show that the rotated frame is
conformal on each $D_i$.  The functions
$B_i=(\prod_r\lambda_r)^{1/(k+1)}/\lambda_i$ descend through $\pi_i$ to the
functions $\mu_i$ on the leaves, and the comparison of the conformal factor
with $\gamma_i=\prod_{j\neq i}|B_j-B_i|$ gives exactly the two metric formulas in
Definition \ref{Gen_Levi-Civita}.  Proposition \ref{eigenconst}, whose proof is
given in Appendix \ref{app:eigen}, supplies the required constancy of $\mu_i$
when $\dim D_i>1$.  In the affine case all eigenvalues are constant, so all
$\mu_i$ are constant and the same argument proves Theorem \ref{mainthm_eisen}.

We next describe how Proposition \ref{outline_transverse_vanishing} is proved.  In
Section \ref{principal_part_sec}, the column operation \eqref{eq:1} replaces $b$
by $\tilde b$ without changing the maximal minors.  Lemma \ref{lin_quad_lem}
shows that the vertex degrees of $A$ and $\tilde b$ are at least one and two,
respectively, and Corollary \ref{frak_FAL} therefore reduces the first relevant
homogeneous part of every maximal minor to the much simpler principal system
\eqref{A.phi.B.frak}.  The formula for its right-hand side is propagated by
Lemma \ref{recursive_tilde_b_frak}.  The additional terms that arise in the
projective, rather than affine, case are computed and shown not to affect these
formulas in Appendices \ref{app:rj}, \ref{app:column}, and \ref{app:vanish}.

Sections \ref{root_monomial_sec} and \ref{distinguished_minors_sec} then use the
normalizable graph to identify monomials whose coefficients cannot cancel.
Lemma \ref{minvert} expresses the coefficients of the linear vertex-homogeneous
entries as chains of structure functions, and Lemmas \ref{root_edge} and
\ref{root} show that a chain contributing to a root monomial is forced to follow
the successive roots of a graph component.  This produces the core sub-matrix
$T$.  Lemmas \ref{coeff=1} and \ref{pcoeff1} show that the distinguished monomial
$p(U)$ occurs in $\det T$ exactly once, with coefficient $\pm1$.  Thus $T$ is a
nonzero full-size minor of $\mathfrak A$, and augmenting it by one row and the
column $\tilde{\mathfrak b}$ provides maximal minors in which selected structure
functions can be isolated.

Section \ref{transverse_vanishing_sec} performs this isolation. It begins with the
augmentations by rows indexed by vertices of depth zero: for them, the
coefficients of suitable monomials in the cofactor expansions of the augmented
minors can be computed directly, and these calculations serve as preliminary
cases. For vertices of arbitrary depth, we augment the core sub-matrix by a row of the
first root of the component in a higher layer and isolate the coefficient of the
product of a depth path monomial with the distinguished monomial $p(U)$.
Lemma \ref{augmented_path_rigidity} shows, by a counting argument for the chains
of structure functions forming the Leibniz terms, that this coefficient is the
sum of the structure function to be eliminated and of a multiple of a structure
function having a root among its inputs; Proposition \ref{cofterm2}
transfers this computation to the whole augmented minor. The resulting
triangular elimination is carried out in Proposition \ref{c2} for a nonabelian
source component, while the case of an abelian source component is treated in
Proposition \ref{c1bis}.
Further, in Proposition \ref{c3} we contract the block equations of the first
layer and obtain, for three distinct components, a linear system with an
invertible three-by-three matrix. These results first force the entire first
layer $\tilde{\mathfrak b}^1$ to vanish (Corollary \ref{b10}).  The operator form
of the first two layers from Lemma \ref{A^_1 and 2_i_lemma} then implies, through
Lemma \ref{beta2}, that the second-layer obstruction vanishes as well.  Combining
this with the Jacobi identity gives Lemma \ref{bianchi}, i.e., the part of
\eqref{outline_vertex_edge_zero} with a central output.  Corollary \ref{btfrak0} then raises the minimal
vertex degree of $\tilde b$ to three; applying the same operator argument to the
cubic associated system proves Proposition \ref{mixed_edge_horizontal_vanishing}
and completes \eqref{outline_vertex_edge_zero}.  Proposition \ref{ady} and
Corollary \ref{involutive} translate these coefficient vanishings into the
foliation statement of Corollary \ref{outline_foliation_net}.
\section{Reduction to the principal part of the fundamental algebraic system}
\label{principal_part_sec}
In this section,  we will show that the coefficients of monomials of minimal degree with respect to $u$-variables indexed by vertices of $\Gamma$ in the $(n-m+1)\times (n-m+1)$ minors of the augmented matrix $[A|b]$ are equal to the coefficients of the same monomials in the  corresponding minors of a more simple augmented matrix, see Corollary\ref{frak_FAL} below. This is important because, as will be shown in the later stages of the proof, the coefficients of these monomials yield the vanishing of the horizontal cross-component structure functions in Section \ref{transverse_vanishing_sec}, while the monomials of the next vertex degree are used in Section \ref{other_sec} to treat the remaining structure functions with horizontal output.
To begin with, let us introduce some terminology and notation. Let $U$ be the tuple
\begin{equation}
U=\{u_s\}_{s\in {V\cup E}}
\end{equation}
\begin{Def}
    \label{U-poly}
A \emph{$U$-polynomial} is a polynomial in the variables of the tuple $U$ with coefficient which are smooth functions on an open set $\mathcal U$ of the manifold $M$, where the projective equivalence of our sub-Riemannian metrics $g_1$ and $g_2$ holds and the local quasi-normal frame $\{X_s\}_{s\in V\cup E}$ is defined.
\end{Def}
The degree of a monomial of a $U$-polynomial and the notion of homogeneity is defined in an obvious way.
\begin{rk}
\label{homogen_total}
Since $U$-polynomials $q_{vs}$ and $R_w$ are linear and quadratic, respectively, from recursive formulas \eqref{elem.A} and \eqref{term.of.b.rec.form} it follows by induction that the entries of $A^p$ and $b^p$ are homogeneous $U$-polynomials of degree $p$ and degree $p+1$, respectively.
    \end{rk}
\begin{Def}
\label{var_def}
The variable $u_s$ is  called a \emph{vertex variable} if $s\in V$ and an \emph{edge
variable} if $s\in E$. A vertex (respectively, an edge) variable $u_s$ is called \emph{$\Gamma_i$-vertex variable} if $s\in V_i$ (respectively, an \emph{$\Gamma_i$-edge variable} if $s\in E_i$).
\end{Def}
\begin{Def}
Given a $U$-monomial define its \emph{vertex degree} (respectively, \emph{edge degree}) as the number of vertex (respectively edge) variables factors in it counting repetitions. In the same way one can define $\Gamma_i$ -vertex (respectively, $\Gamma_i$-edge) degree of a monomial.
A $U$-polynomial $f$ is called vertex (respectively, edge) homogeneous of degree $l$ if all its monomials have  $l$ vertex degree $l$ (respectively, edge degree $l$)  variables. In the same way, we define $\Gamma_i$-vertex (respectively, $\Gamma_i$-edge) homogeneity.
\end{Def}
Any $U$- polynomial $f$ can be uniquely decomposed into vertex homogeneous polynomials:
\begin{equation}
f=\sum_{j=0}^\mu f_{(j)},
\end{equation}
for some nonnegative integer $\mu$,
where $f_{(j)}$ a vertex homogeneous of degree $j$. We call $f_{(j)}$ the \emph{$j$th vertex-homogeneous component of $f$}.
We say that $j\geq 0$ is the \emph{minimal vertex degree of $f$} if $j$ is the minimal nonnegative integer such that $f_{(j)}(U)\neq 0$\footnote{When $f$ is identically equal to zero, the minimal vertex degree is set to be equal to $\infty$.}. In this case, $f_{(j)}$ is called the \emph{initial vertex-homogeneous component} of $f$.
Given a matrix $L$ with entries being $U$-polynomials  denote by $L_{(j)}$ the matrix with entries equal to $j$ vertex homogeneous components of the corresponding entries of $L$. The matrix $L_{(j)}$ will be called the \emph{$j$ vertex homogeneous part of the matrix $L$}. The notion of $j$ vertex homogeneous and minimal vertex homogeneity for the matrix $L$ with polynomial in $U$ entries is defined in an obvious way (analogously to the same notion for a single polynomial in the previous paragraph). We also call $L_{(1)}$ (respectively $L_{(2)}$) the \emph{linear}  (respectively, \emph{quadratic})  vertex homogeneous part of $L$.  Finally, $1$ vertex homogeneous (respectively, $2$ vertex homogeneous) polynomials or  matrices with polynomial entries in $U$ will be called linearly (respectively, quadratically) vertex homogeneous.
Further, define a map $\mathcal C:V\cup E\mapsto [1:k]$ as follows:
\begin{equation}
\label{component_map}
\mathcal C(s)=i, \quad \text{ if } s\in V_i\cup E_i,
\end{equation}
where $\Gamma_i$ is as in \eqref{graph_decomp}.
\begin{Lemma}
\label{h_1_degree}
 If $j$ is the minimal vertex degree of a $U$-polynomial $f$ (respectively, a matrix $L$ with $U$-polynomial entries), then the minimal vertex degree of $\vec h_1(f)$ (respectively, $\vec h_1(L)$) is at least $j$. Moreover, $(\vec h_1(f))_{(j)}$ is equal to the sum of all possible terms obtained from the summands of $f_{(j)}$ by replacing one of their vertex variables factor by the $1$-vertex homogeneous component of the action of $\vec h_1$ on this variable.
\end{Lemma}
\begin{proof}
First note that from   \eqref{eq:vech}, \eqref{poison1}, and \eqref{struct_function}, it easily follows that
    \begin{equation}
    \label{Hamiltonian_vf_action}
    \vec{h}_1(u_v)=\sum\limits_{s \in V\cup E} q_{vs}u_s=\sum\limits_{j \in V}\sum\limits_{s \in V\cup E} c^s_{jv}u_ju_s,
    \end{equation}
    where $q_{ws}$ is as in \eqref{qjk}. By the product rule, the action of $\vec h_1$ on any summand  $S$ of a $U$-polynomial $f$ is equal to the sum of terms in which  one of the factors of $S$ is replaced by the action of $\vec h_1$ on this factor. There are three type of factors of $S$:
\begin{enumerate}
  \item a vertex  variable $u_v$: its vertex degree is $1$ and the vertex degree of $\vec h_1(u_v)$ is at least $1$ by \eqref{Hamiltonian_vf_action};
\item  and edge variable $u_e$: its vertex degree is $0$ and the vertex degree of $\vec h_1(u_e)$ is at least $1$ by \eqref{Hamiltonian_vf_action};
\item a function $K$ on an open set $\mathcal U$ of $M$: The vertex degree of $K$ is zero, but the vertex degree of $\vec h_1(K)$ is at least 1 by \eqref{eq:vech}.
\end{enumerate}
So, the first statement of the lemma follows from the fact that in all three cases above the minimal vertex degree is not decreased by the action of $\vec h_1$, and the second statement of the lemma follows from the fact that only in the case (1)  the minimal vertex degree can be preserved by this action.
\end{proof}
Given $l\in E$  let $A_l$ be the column of $A$ corresponding to the edge $l$. Since we are concerned with the minors of the augmented matrix  $[A|b]$ of the fundamental algebraic system \eqref{A.phi.B}, we can replace the column $b$ of this matrix with a column obtained from it by a column operation, i.e.  subtracting from $b$ any linear combination of columns of $A$. In order to use the vertex homogeneity the following column operation is convenient
\begin{equation}
 \tilde{b}=b-\sum\limits_{l\in E}\alpha^2(l)u_l A_l \label{eq:1},  \end{equation}
where $\alpha:V\cup E\rightarrow \mathbb R$ is defined by \eqref{alpha_of} and \eqref{alpha_edge_extension}. Let $\tilde b^p$ be the sub-column of $\tilde b$ corresponding to the $p$th layer.
The following lemma and its corollary justify the introduced terminology
\begin{Lemma}\label{lin_quad_lem}
Every entry of $A$ has vertex degree at least $1$, and every entry of $\tilde b$ has vertex degree at least $2$.
\end{Lemma}
Note that the lemma gives only lower bounds for the vertex degrees: it does not exclude that some individual entries, or even the entire last column, vanish identically.
\begin{proof}
For brevity, throughout the proof we write $\lambda_s=\alpha(s)^2$ and $c=\vec h_1(a)/a$. Recall that, by the first divisibility condition, the latter function is a linear form in the vertex variables, namely,
\begin{equation}\label{first_div_linear}
c=\frac12\sum_{v\in V}\frac{X_v\lambda_v}{\lambda_v}u_v.
\end{equation}
Further, combining Corollary \ref{rj_reduction} with the column operation \eqref{eq:1}, we obtain the following expression for the entries of the first layer:
\begin{equation}\label{tb1}
\tilde b^1_w=
\sum_{\substack{l,v\in V\\\mathcal C(l)\ne\mathcal C(w),\ v\ne l}}
(\lambda_w-\lambda_l)c^l_{vw}u_vu_l.
\end{equation}
Here and below, if several restrictions are written under a summation sign, then all of them are assumed to hold simultaneously. In particular, since the two summation indices in \eqref{tb1} are required to be distinct, this formula contains no square terms $u_l^2$. Moreover, for $e\in E_i$, the calculation of the correction term of the column operation, carried out in Appendix \ref{app:column}, gives
\[
\vec h_1\lambda_e-c\lambda_e
=\frac12\sum_{y\in V_i}X_y\lambda_e\,u_y=0.
\]
The last equality follows from Proposition \ref{eigenconst}, because $E_i\ne\varnothing$ implies $\dim D_i\ge2$. Note that Proposition \ref{eigenconst} is proved in Appendix \ref{app:eigen} by an independent argument.

Let us now derive an exact recursive identity for the layers of the modified column, that is, of the column obtained by the column operation. This identity is valid before any homogeneous components are discarded. Differentiating the column operation and using both recursive formulas of Proposition \ref{fda}, we obtain
\begin{equation}\label{exact_column_rec}
\begin{split}
\tilde b^{p+1}_w={}&\vec h_1\tilde b^p_w-c\tilde b^p_w
+\sum_{e\in E}(\vec h_1\lambda_e-c\lambda_e)u_e a^p_{we}\\
&+\sum_{\substack{e\in E\\s\in V\cup E}}(\lambda_e-\lambda_s)a^p_{we}q_{es}u_s
-\frac12\sum_{\substack{e\in E\\s\in V}}a^p_{we}X_e\lambda_s\,u_s^2.
\end{split}
\end{equation}
Indeed, by the definition of the column operation, $b^p_w=\tilde b^p_w+\sum_e\lambda_eu_ea^p_{we}$. Substituting this expression into \eqref{term.of.b.rec.form}, using the relation $\vec h_1u_e=\sum_s q_{es}u_s$ together with the recursive formula for the layers of the coefficient matrix, and cancelling the terms $\sum_e\lambda_eu_ea^{p+1}_{we}$, we arrive at the required identity. Note that the projective term enters this identity in the form $-c\tilde b^p_w$.

Now we are ready to prove the lemma by induction on the layers. The entries of the first layer of $A$ are linear in the vertex variables. Using the recursive formula for the layers of this matrix together with Lemma \ref{h_1_degree}, we see by induction that all entries of the subsequent layers have vertex degree at least $1$ as well. Further, by \eqref{tb1}, the first layer of $\tilde b$ is quadratic in the vertex variables. For the induction step in the case of the last column, consider the right-hand side of \eqref{exact_column_rec}: its first term does not decrease the vertex degree, its third term vanishes, as was shown above, and its remaining terms have vertex degree at least $2$. Hence, by induction, the assertion of the lemma holds for all layers.

Finally, denote by $\mathfrak V^j$ the space of polynomials of vertex degree at least $j$. Then the same identity yields the following recursive formula:
\begin{equation}\label{tbs}
\tilde b^{p+1}_w=\vec h_1\tilde b^p_w+
\sum_{\substack{e,f\in E\\\mathcal C(e)\ne\mathcal C(f)}}
(\lambda_e-\lambda_f)a^p_{we}q_{ef}u_f
\pmod{\mathfrak V^3}.
\end{equation}
This recursive formula will be needed in the sequel. Note that it holds in the general projective case and, in particular, in the affine one.
\end{proof}
\begin{rk}
The exact identity \eqref{exact_column_rec} obtained in the proof gives the complete recursion for the layers of the modified column in the general projective case. The terms containing derivatives of the eigenvalues appear in it explicitly, and Proposition \ref{eigenconst} is applied only afterwards, in order to show that some of these terms vanish. This does not lead to a circular argument, since in Appendix \ref{app:eigen} that proposition is derived directly from the unreduced fundamental algebraic system, without using the results of the present section.
\end{rk}

 By  Lemma \ref{lin_quad_lem} the minimal homogeneous degree of any $(n-m+1)\times (n-m+1)$ minor of the augmented matrix $(A|\tilde b)$ is at least $n-m+2$. This immediately implies the following
\begin{Coro}
\label{frak_FAL}
The $(n-m+2)$ vertex homogeneous component of any $(n-m+1)\times (n-m+1)$ minor of the augmented matrix $(A|b)$ of the fundamental algebraic system is equal to the corresponding minor of $(A_{(1)}|\tilde b_{(2)})$.
\end{Coro}
Corollary \ref{frak_FAL} shows that the analysis of coefficients of the $(n-m+2)$ vertex homogeneous component in the minors of the augmented matrix $(A|b)$ of the fundamental algebraic system can be done through the analysis of the same monomials in the minors of the much simpler matrix $(A_{(1)}|\tilde  b_{(2)})$.
Because of the importance and frequency of use  of matrices $A_{(1)}$ and $\tilde b_{(2)}$, and for simplicity of notation, we denote them by $\mathfrak{A}$ and $\tilde {\mathfrak{b}}$, respectively.
\begin{Def}
We call the system
\begin{equation}
	 \label{A.phi.B.frak}
	\mathfrak{A} \Psi = \tilde{\mathfrak{b}},
	 \end{equation}
the \emph{principal part} of the fundamental algebraic system \eqref{A.phi.B}.
\end{Def}
 As before, we write $\mathfrak{A}^p$ (respectively, $\tilde{\mathfrak{b}}^p$) for the linear vertex homogeneous component of the $p$th layer of the matrix $A$ (respectively, for the quadratic vertex homogeneous component of the $p$th layer of the column $\tilde b$). The entries of $\mathfrak{A}^p$ are labeled by $\mathfrak{A}^p_{ws}$, where $w\in V$ and $s\in E$ (respectively, the entries of $\tilde{\mathfrak{b}}^p$ are labeled by $\tilde{\mathfrak{b}}^p_w$), and we call $\mathfrak{A}^p$ (respectively, $\tilde{\mathfrak{b}}^p$) the $p$th layer of $\mathfrak{A}$ (respectively, of $\tilde{\mathfrak{b}}$).
\begin{rk}
Similarly to the proof of Lemma \ref{lin_quad_lem}, one can show by induction on the layer $p$ that the minimal vertex degrees of the last terms in the second lines of \eqref{elem.A} and \eqref{term.of.b.rec.form} are at least two and three, respectively. Therefore, the matrices $\mathfrak{A}^p$
satisfy the following recursive relations:
 \begin{equation}
\label{elem.A.frak}
		 \left\{
		 \begin{aligned}
        & \mathfrak{A}^1_{v,w} = a^1_{v,w}=q_{vw}, \qquad  &v \in V, \ w \in E,\\
		& \mathfrak{A}^{p+1}_{v, w} = \bigl(\vec{h}_1(\mathfrak{A}^p_{v,w})\bigr)_{(1)} , \qquad  &v \in V, \ w \in E,
		 \end{aligned}
		 \right. .
		 \end{equation}
        The recursive formula for $\tilde{\mathfrak{b}}^p$ is derived in Lemma \ref{recursive_tilde_b_frak} below, formula \eqref{tbs_frak} there.
\end{rk}

In the sequel, it will be convenient to have an explicit description of the part of the Hamiltonian vector field that preserves the vertex degree. Namely, at a fixed base point, the degree-preserving part of $\vec h_1$ is the derivation
\begin{equation}\label{principal_delta}
\delta u_v=\sum_{e\in E}q_{ve}u_e\quad(v\in V),\qquad
\delta u_e=0\quad(e\in E),\qquad \delta f(q)=0.
\end{equation}
Here the last relation means that this derivation annihilates the functions on the base. In other words, on homogeneous polynomials of vertex degree $j$ we have $(\vec h_1 f)_{(j)}=\delta f$; in particular, the recursive relations of the previous remark take the form $\mathfrak A^{p+1}=\delta\mathfrak A^p$. This notation will also be used later, when we take homogeneous parts of other systems.
\section{Root monomials}
\label{root_monomial_sec}
In this section, we study the entries of the principal part of the fundamental algebraic system in more detail. First, we express the coefficients of the entries of its coefficient matrix in terms of chains of structure functions, derive a recursive formula for the layers of its right-hand side, and show that every monomial of the right-hand side involves variables of at least two components of the graph. Then we introduce path and root monomials and show that the coefficients of root monomials reduce to single structure functions.

\begin{Lemma}\label{minvert}
The coefficient of a monomial
\begin{equation}\label{lin_vertex}u_v\prod_{e\in E}u_e^{m_e},\qquad \sum_e m_e=p-1,
\end{equation}
in $\mathfrak A^p_{ws}$ equals
\begin{equation}\label{cprod}
\sum_{(e_1,\ldots,e_{p-1})}\ \sum_{v_1,\ldots,v_{p-1}\in V}
c^s_{v_{p-1}w}\prod_{j=1}^{p-1}c^{e_j}_{v_{j-1}v_j},\qquad v_0=v.
\end{equation}
Here the first sum is taken over all distinct ordered words composed of edges in which each edge $e$ occurs exactly $m_e$ times; repeated occurrences of the same edge are not labelled separately.
\end{Lemma}
\begin{proof}
The proof is by induction on the layer. For the first layer we have $\mathfrak A^1_{ws}=\sum_v c^s_{vw}u_v$, which is the required formula in this case. The subsequent layers are obtained by applying the derivation $\delta$ repeatedly, see \eqref{principal_delta}. By the definition of this derivation, its application to a monomial replaces the vertex variable of this monomial by $\sum_{a,e}c^e_{av}u_au_e$ and leaves all edge variables and all coefficients unchanged. Therefore, by induction, the entry under consideration is the sum, over all ordered words of edges and all choices of the intermediate vertices, of the products of structure functions as in \eqref{cprod} multiplied by the corresponding monomials. Collecting the terms that give equal monomials (recall that the variables commute), we obtain exactly the sum in the statement of the lemma. Note that every ordered word is counted exactly once, so that no additional combinatorial factors appear when some edges occur several times.
\end{proof}

\begin{Coro}
\label{block_diag}
The following two statements hold:
\begin{enumerate}
    \item  If $w\in V$ and $s\in E$ such that $\mathcal C(w)\neq \mathcal C(s)$ , then $\mathfrak{A}_{ws}^p=0$.
    \item  If $w\in V_i$ and $s\in E_i$, then $\mathfrak{A}_{ws}^p$
    depends  on $\Gamma_i$-vertex variables   and on $\Gamma_i$-edges variables only. Moreover, it is linear  in $\Gamma_i$-vertex variables and has $\Gamma_i$-edge degree equal to $p-1$.
\end{enumerate}
\end{Coro}
\begin{proof}
The corollary immediately follows from Lemma \ref{minvert} and the fact that if the product \eqref{cprod} does not vanish, then vertices $w, v_1,\ldots, v_{p-1} , v$ and edges $s, e_1,\ldots, e_{p-1}
$ belong to the same component of the graph $\Gamma$. The later fact follows from the decomposition \eqref{eigen_decomp} of the Tanaka symbol $\mathfrak m$.
\end{proof}
\begin{rk}
\label{block_diag_rem}
Recalling that the rows (columns) of $\mathfrak{A}^p$ are indexed by vertices (respectively edges) of $\Gamma$, $\mathfrak{A}^p$ has a block diagonal form according to the decomposition \eqref{graph_decomp}.
\end{rk}
\begin{Lemma}\label{recursive_tilde_b_frak}
For every $p\ge1$ and $v\in V$, the following recursive formula holds:
\begin{equation}\label{tbs_frak}
\tilde{\mathfrak b}^{p+1}_v=\delta\tilde{\mathfrak b}^p_v+
\sum_{\substack{e,f\in E\\\mathcal C(e)=\mathcal C(v),\ \mathcal C(f)\ne\mathcal C(e)}}
(\alpha(e)^2-\alpha(f)^2)\mathfrak A^p_{ve}q_{ef}u_f.
\end{equation}
\end{Lemma}
\begin{proof}
By definition, the layers of the principal part of the last column are the quadratic vertex homogeneous components of the corresponding layers of the modified column. Therefore, we take the quadratic vertex homogeneous part of both sides of \eqref{tbs} and analyze separately the two terms on its right-hand side. By Lemma \ref{h_1_degree}, the quadratic part of the first term gives the first term of the required formula. As for the second term, only the linear part $\mathfrak A^p$ of $A^p$ contributes to it, because the product of the remaining factors of each summand is linear in the vertex variables. Moreover, by Corollary \ref{block_diag}, the nonzero columns of this linear part satisfy $\mathcal C(e)=\mathcal C(v)$, which gives the restriction in the summation. Finally, note that the derivation of \eqref{tbs} takes into account all additional terms arising in the projective case, so no assumption of affine equivalence is used here.
\end{proof}
\begin{rk}
Note that all formulas of this section are valid for arbitrary projectively equivalent pairs of sub-Riemannian metrics, and not only for affinely equivalent ones. In the affine case, the derivatives of the transition eigenvalues vanish, because these eigenvalues are constant by \cite[Proposition 4.7]{jmz2019}. The derivatives of the structure functions need not vanish, even in the affine case. However, they do not enter the formulas of this section: differentiation of a function on the base raises the vertex degree (see the proof of Lemma \ref{h_1_degree}), so it does not contribute to the vertex-degree-preserving part $\delta$ of $\vec h_1$.
\end{rk}

In the sequel, we need the following consequences of projective equivalence, written in the notation of the present paper. They follow from the so-called \emph{first divisibility condition}, stating that $\frac{\vec h_1(a^2)}{a^2}$ is a polynomial, where $a$ is as in \eqref{eq:2a}. This condition was introduced in \cite{Z}, where it was established for contact and even-contact distributions and where the identities of the following proposition were derived from it in the projective case \cite[Proposition 6]{Z}. For arbitrary distributions, the first divisibility condition was proved in \cite[Proposition 4.2 and Lemma 4.4]{jmz2019}. In the affine case, the transition eigenvalues are constant, and the corresponding specializations of these identities appear in \cite[Proposition 11]{Zaifeng}.

\begin{Prop}
\label{first divi}
Suppose $g_1$ and $g_2$ are non-conformal, projectively equivalent sub-Riemannian metrics on a distribution $D$. If $\{X_s\}_{s \in V\cup E}$ is a weakly adapted frame with respect to this pair and $\mathcal C$ as in \eqref{component_map}, then the following hold:
    \begin{enumerate}
        \item $X_v \left ( \frac{\alpha(w)^2}{\alpha(v)^2}\right )=2c^w_{vw}\left(1-\frac{\alpha(w)^2}{\alpha(v)^2}\right)$ for $v,w \in V$
        \item $X_v \left ( \frac{\alpha(w)^2}{\alpha(v)}\right )=0$ whenever $\mathcal{C}(w)\neq \mathcal{C}(v)$
        \item $c^l_{vw}=-c^w_{vl}$ whenever $v, w,l \in V$  such that $\mathcal C(w)=\mathcal C(l)\neq \mathcal C(v)$ and $w\neq l$.
        \item $(\alpha(v)^2-\alpha(w)^2)c^l_{vw}+(\alpha(w)^2-\alpha(l)^2)c^v_{wl}+(\alpha(v)^2-\alpha(l)^2)c^w_{vl}=0$ whenever $v,l,w \in V$ are pairwise distinct vertices.
    \end{enumerate}
\end{Prop}
Note that item (3) above is a consequence of item (4), applied to the case when one pair of vertices in the triple $\{l,v,w\}$ (for example, $l$ and $w$) belongs to the same component of the graph $\Gamma$ (see Definition \ref{component_def} above). Items (1), (3), and (4) follow by comparing the coefficients of the corresponding monomials in the first divisibility condition. In \cite{Z}, item (2) is obtained by applying this condition to both ordered pairs $(g_1,g_2)$ and $(g_2,g_1)$. For completeness, we include a different proof of item (2), in the present notation, which uses the first layer of the fundamental algebraic system.
\begin{proof}[Proof of item (2)]
Fix $i$ and a vertex $w\in V_i$. By Theorem \ref{Tanaka_decomp_prop}, $[D_i,D_t]\subset D$ for $t\neq i$, so that $c^e_{xv}=0$ whenever $e\in E$ and $\mathcal C(x)\neq\mathcal C(v)$. Together with the skew-symmetry of the structure functions in the lower indices, this gives
\[
\sum_{v\in V_i}u_vq_{ve}=\sum_{v\in V_i}\sum_{x\in V}c^e_{xv}u_xu_v=\sum_{x,v\in V_i}c^e_{xv}u_xu_v=0,\qquad e\in E.
\]
Let us multiply the equation of the first layer of the fundamental algebraic system \eqref{A.phi.B} corresponding to a vertex $v\in V_i$ by $u_v$ and add up the results over all $v\in V_i$. By \eqref{elem.A} and \eqref{term.of.b.rec.form}, we get
\[
\sum_{v\in V_i}u_vR_v=\sum_{e\in E}\widetilde{\Phi}_e\sum_{v\in V_i}u_vq_{ve}=0
\]
on the domain of the orbital diffeomorphism. This domain meets every fiber over a neighborhood of $q_0$ in a nonempty open set, and the left-hand side of the last identity is a polynomial on the fibers; hence, it vanishes identically. Now let $y\in V_t$ with $t\neq i$, and write $\lambda_i=\alpha_i^2$ and $\lambda_t=\alpha_t^2$. The monomial $u_wu_y^2$ appears in the last identity only in the summand $u_wR_w$. By \eqref{rj_unreduced} and \eqref{Hamiltonian_vf_action}, the coefficient of $u_y^2$ in $R_w$ is equal to
\[
(\lambda_i-\lambda_t)c^y_{yw}-\frac12X_w(\lambda_t),
\]
because the first and the third terms of \eqref{rj_unreduced} are divisible by $u_w$. On the other hand, by item (1) applied to the pair $(w,y)$,
\[
(\lambda_i-\lambda_t)c^y_{yw}=-\frac12\lambda_iX_w\Bigl(\frac{\lambda_t}{\lambda_i}\Bigr)=-\frac12X_w(\lambda_t)+\frac{\lambda_t}{2\lambda_i}X_w(\lambda_i).
\]
Therefore, the vanishing of the coefficient of $u_wu_y^2$ gives $X_w(\lambda_t)=\dfrac{\lambda_t}{2\lambda_i}X_w(\lambda_i)$, i.e. $X_w\Bigl(\dfrac{\lambda_t}{\sqrt{\lambda_i}}\Bigr)=0$. Since $\sqrt{\lambda_i}=\alpha(w)$ and $\lambda_t=\alpha(y)^2$, this is exactly item (2).
\end{proof}
In the sequel, the relations given in items (1)--(3) of Proposition \ref{first divi} will be used repeatedly.
The following lemma is the generalization of Lemma 3 from \cite{Zaifeng}, where it is proved for components of the first layer $\tilde{\mathfrak{b}}^1$ of $\tilde {\mathfrak{b}}$ only.
\begin{Lemma}\label{bu2}
Every monomial appearing with a nonzero coefficient in an entry of $\tilde{\mathfrak b}$ contains variables indexed by vertices or edges of at least two distinct components of $\Gamma$.
\end{Lemma}
\begin{proof}
We prove the lemma by induction on the layer $p$. For the first layer, recall that, by \eqref{tb1}, its entries contain no squares of vertex variables. Therefore, it is enough to consider the products of two distinct vertex variables indexed by vertices of the same component. If $v\ne l$ belong to one component, then the coefficient of $u_vu_l$ is
\begin{equation}\label{b1tilde}
(\alpha(w)^2-\alpha(l)^2)(c^l_{vw}+c^v_{lw}).
\end{equation}
This coefficient vanishes. Indeed, if the vertex indexing the entry belongs to the same component as these two vertices, then the first factor vanishes because of the equality of the corresponding eigenvalues. Otherwise, the second factor vanishes by item (3) of Proposition \ref{first divi}, applied to the distinct vertices $v,l$, combined with the skew-symmetry of the structure functions in the lower indices. This proves the statement of the lemma for the first layer.

Now assume that the statement holds for some layer and consider the next one. We analyze separately the two terms on the right-hand side of the recursive formula \eqref{tbs_frak}. Regarding the first term, by the definition of the derivation $\delta$ and the decomposition of the Tanaka symbol, this derivation replaces a vertex variable only by variables indexed by vertices and edges of the same component, and it leaves the edge variables unchanged. Therefore, it preserves the property in the statement of the lemma. Regarding the second term, each of its summands contains variables from the component of $e$ and the edge variable $u_f$ of a different component, so it has this property as well. This completes the induction and proves the lemma.
\end{proof}

We now introduce a few definitions that will simplify the statement and proof of the following lemmas. In the sequel, we
denote by $\rt j$ the $j$th-root of the graph $\Gamma_i$, see Definition \ref{root_def}.
\begin{Def}
\label{root_monom_def}
   We call a $U$-monomial a \emph{path monomial} of the graph $\Gamma_i$ if the monomial contains the variable $u_{\rho_0^i}$ and all of the remaining variables are edge-variables indexed by edges from a simple path in $\Gamma_i$ starting at $\rho_0^i$.
 In addition, if a path monomial corresponds to a path in $\Gamma_i$ which traverses only roots
 in consecutive order (i.e. the $j$th  edges  connects the $(j-1)$'st root of $\Gamma_i$ with the $j$th root of $\Gamma_i$),
then we call such a monomial a \emph{root monomial} and the corresponding path a \emph{root path}.\footnote {There are $\omega_i+1$ different root monomials of $\Gamma_i$, one for each degree from 1 to $\omega_i+1$, namely $u_{\rho_0^i}, u_{\rho_0^i}u_{\overrightarrow{\rho_0^i\rho_1^i}},\ldots,u_{\rho_0^i}u_{\overrightarrow{\rho_0^i\rho_1^i}}\ldots u_{\overrightarrow{\rho_{\omega_i-1}^i,\rho_{\omega_i}^i}}$. }
\end{Def}
Since the $U$-polynomials appearing as entries of $\mathfrak{A}^p$ are homogeneous of degree $p$, it follows that root monomials in the entries of the $p$th layer $\mathfrak{A}^p$ of the fundamental algebraic system are of the form $Cu_{\rho_i}u_{\overrightarrow{\rho_0^i\rho_1^i}}\ldots u_{\overrightarrow{\rho_{p-2}^i,\rho_{p-1}^i}}$ for some $i\in [1:k]$,  $p\in [1:\omega_i+1]$.

Note that path and root monomials satisfy Lemma \ref{minvert}. In particular, the coefficient $C$ of a root monomial is a sum of terms of the form
\begin{equation}
\label{prelim_prod}
c^s_{v_{p-1}w}\prod_{j=1}^{p-1} c^{\overrightarrow{\rho^i_{j-1} \rho^i_{j}}}_{v_{j-1}v_j}.
\end{equation}
Here, as in Lemma \ref{minvert}, the sum is also taken over the ordered words formed by the edges of the root path; it will follow from the properties of root edges established below that, for the root monomials under consideration, only the natural order of these edges gives a nonzero contribution. Moreover, it turns out that the remaining indices which are not fixed by Lemma \ref{minvert} in the coefficient of a root monomial (i.e. $v_1,\ldots,v_{p-1}$) are completely determined for suitably chosen rows of the fundamental algebraic system, and the sum of terms in this coefficient collapses to a single term.

\begin{Lemma}
\label{root_edge}
Let $0\leq k\leq \omega_i$ and let $z\in V_i$ satisfy
$\mathrm{depth}(z)>k$. Set
\[
e_{k,z}:=\overrightarrow{\rho_k^i z}.
\]
Then $e_{k,z}\in E_i$ and, for $x,y\in V_i$, the following hold:
\begin{enumerate}
    \item If $x=\rho_k^i$, then
    \[
    c^{e_{k,z}}_{\rho_k^i y}=0
    \qquad\text{unless }y=z,
    \]
    while
    \[
    c^{e_{k,z}}_{\rho_k^i z}=1.
    \]
    \item If either $x$ or $y$ has depth strictly smaller than $k$, then
    \[
    c^{e_{k,z}}_{xy}=0.
    \]
\end{enumerate}
\end{Lemma}
\begin{proof}
First, since $z$ has depth greater than $k$, it survives the first $k$ reductions of the normalizable graph. At the $k$th reduction step, it is a non-isolated vertex of the reduced graph whose root is $\rho_k^i$. Hence
$e_{k,z}=\overrightarrow{\rho_k^i z}$ belongs to $\mathrm{Out}(\rho_k^i)$.
By quasi-normality,
\[
[X_{\rho_k^i},X_z]=X_{e_{k,z}},
\]
so that $c^{e_{k,z}}_{\rho_k^i z}=1$.

Now let $\{\mathfrak h_l^i\}_{l=0}^{\omega_i+1}$ be the filtration associated with the chosen tuple of roots. By construction,
\[
\mathrm{pr}(X_{e_{k,z}})\in \mathfrak h_k^i,
\qquad
\mathrm{pr}(X_{e_{k,z}})\notin \mathfrak h_{k-1}^i,
\]
where $\mathfrak h_{-1}^i=0$. If, for example,
$\mathrm{depth}(x)<k$, then $X_x$ projects into
$\mathfrak h_{k-1}^i$. Since $\mathfrak h_{k-1}^i$ is an ideal,
$[X_x,X_y]$ projects into $\mathfrak h_{k-1}^i$ for every $y$.
Consequently, this bracket has zero coefficient along $X_{e_{k,z}}$. The same argument with $x$ and $y$ interchanged proves item (2).

It remains to prove the vanishing assertion in item (1). If
$\mathrm{depth}(y)<k$, then this assertion follows from item (2). If
$\mathrm{depth}(y)=k$, then $X_y$ projects into $\mathfrak h_k^i$ and
\eqref{decrease}, together with the fact that $\mathfrak h_k^i$ is generated at this stage by the depth-$k$ part and the preceding ideal, gives
\[
\mathrm{pr}[X_{\rho_k^i},X_y]\in \mathfrak h_{k-1}^i,
\]
so the coefficient of $X_{e_{k,z}}$ again vanishes. Finally, if
$\mathrm{depth}(y)>k$, then
$\overrightarrow{\rho_k^i y}\in E_i$ and quasi-normality gives
\[
[X_{\rho_k^i},X_y]=X_{\overrightarrow{\rho_k^i y}},
\]
which has a nonzero $X_{e_{k,z}}$-coefficient only when $y=z$.
This completes the proof of the lemma; the corresponding statements with the lower indices interchanged follow from the skew-symmetry of the structure functions.
\end{proof}

\begin{Lemma}
\label{root}
    Assume that  $f(U)=K u_{e_1} u_{e_2}\ldots u_{e_{p-1}} u_{\rho_0^i}$ is a monomial which appears in the $\mathfrak{A}^p_{ws}$ entry of $\mathfrak{A}^p$ such that each $e_j$ is an edge belonging to a root path in  $\Gamma_i$ (see Definition \ref{root_monom_def}). Then the following two statements hold \begin{enumerate}
        \item
   For $p\geq 2$, $u_{\overrightarrow{\rho_{0}^i\rho_{1}^i}}$ divides $f(U)$;
   \item For $p\geq 3$ if for some $0\leq j\leq p-3$, $u_{\overrightarrow{\rho_k^i\rho_{k+1}^i}}$ appears linearly in  $f(U)$ for all $0\leq k\leq j$, then  $u_{\overrightarrow{\rho_{j+1}^i\rho_{j+2}^i}}$ appears in $f(U)$.
   \end{enumerate}
   Moreover, if $f(U)$ divides a root monomial of $\Gamma_i$, then
   \begin{equation}
   \label{f_div}
       f(U)=c^s_{\rho_{p-1}^i,w}u_{\rt 0}u_{\overrightarrow{\rho_0^i\rho_1^i}}\ldots u_{\overrightarrow{\rho_{p-2}^i,\rho_{p-1}^i}},
   \end{equation}
and, in particular, it is a root monomial.
\end{Lemma}
\begin{proof}
By Lemma \ref{minvert}, the coefficient $K$ is a sum of terms of the form
\begin{equation}
\label{root_coeff_chain}
c^s_{v_{p-1}w}\prod_{r=1}^{p-1}c^{e_r}_{v_{r-1}v_r},
\qquad v_0=\rho_0^i.
\end{equation}
Let us choose a nonzero summand in this expression and label the edge factors in the order in which they occur in that summand. For brevity, we also set
\[
\varepsilon_r:=\overrightarrow{\rho_r^i\rho_{r+1}^i}.
\]

Assume first that $p\geq2$. Since $e_1$ belongs to a root path, $e_1=\varepsilon_k$ for some $k\geq0$. If $k>0$, then the first lower index $v_0=\rho_0^i$ has depth smaller than $k$, so item (2) of Lemma \ref{root_edge} gives $c^{\varepsilon_k}_{\rho_0^iv_1}=0$, contrary to the choice of a nonzero summand. Hence $e_1=\varepsilon_0$, and item (1) of Lemma \ref{root_edge} forces $v_1=\rho_1^i$. This proves statement (1).

To prove statement (2), suppose that the variables $u_{\varepsilon_0},\ldots,u_{\varepsilon_j}$ occur linearly in $f(U)$. The preceding argument gives the base case $e_1=\varepsilon_0$ and $v_1=\rho_1^i$. Assume by induction that
\[
e_r=\varepsilon_{r-1},\qquad v_r=\rho_r^i,
\qquad 1\leq r\leq j+1.
\]
Write $e_{j+2}=\varepsilon_k$. If $k>j+1$, then $\rho_{j+1}^i$ has depth smaller than $k$, and item (2) of Lemma \ref{root_edge} makes the factor $c^{\varepsilon_k}_{\rho_{j+1}^iv_{j+2}}$ vanish. If $k\leq j$, then $\varepsilon_k$ has already occurred among $e_1,\ldots,e_{j+1}$, contradicting the assumed linear occurrence of $u_{\varepsilon_k}$ in $f(U)$. Therefore $k=j+1$, and item (1) of Lemma \ref{root_edge} gives $v_{j+2}=\rho_{j+2}^i$. Thus $u_{\varepsilon_{j+1}}$ divides $f(U)$, proving statement (2).

Finally, suppose that $f(U)$ divides a root monomial. Then all its edge variables occur linearly, and the repeated application of statements (1) and (2) gives
\[
e_r=\varepsilon_{r-1},\qquad v_r=\rho_r^i,
\qquad 1\leq r\leq p-1.
\]
The same argument applies to every nonzero summand of \eqref{root_coeff_chain}, so the intermediate indices are uniquely determined. Since
\[
c^{\varepsilon_r}_{\rho_r^i\rho_{r+1}^i}=1,
\]
the coefficient reduces to $c^s_{\rho_{p-1}^i,w}$, and \eqref{f_div} follows.
\end{proof}
\section{Distinguished minors of the principal part of the fundamental algebraic system}
\label{distinguished_minors_sec}
We begin the task of analyzing maximal minors of the fundamental algebraic system. We will fix a particular $(n-m)\times (n-m)$ sub-matrix $T$ of $\mathfrak{A}$ by choosing rows of $\mathfrak{A}$ such that the coefficients of root monomials in their entries are identically $1$.
\begin{Def}
\label{core_submatrix}
The sub-matrix $T$ is called the \emph{core sub-matrix} of the fundamental algebraic system and its determinant is called the \emph{core minor}.
\end{Def}
Every $i\in [1:k]$ contributes  $\dim \mathfrak m_{-2}^i$ rows of $A$ to $T$ as follows:   for each $p \in [1:\omega_i+1]$, choose the rows of $\mathfrak{A}^p$ which correspond to the vertices of $\Gamma_i$ of depth greater than $p-1$ (see Definition \ref{vertex_depth}). Note that by Remark \ref{r_depth_rem}  there are $r_{p-1}^i$ such vertices.
\begin{rk}
\label{degree_rk}
    By item (2) of Corollary \ref{block_diag}, the core minor $\det T$ is a $\Gamma_i$-vertex homogeneous (of degree $d_i$) and $\Gamma_i$-edge homogeneous (of degree $\displaystyle{\sum_{j=1}^{\omega_i}j r_j^i}$
    ).
\end{rk}
\begin{Lemma}
\label{coeff=1}
The root monomials (of some $\Gamma_i$) appear exactly once in each row and each column of the core sub-matrix $T$ and its coefficient is equal to $1$. In more detail, if $w$ is a vertex  of depth greater than $p-1$ in $\Gamma_i$, then in the $w$-row of the core sub-matrix $T$ in $\mathfrak{A}^p$ the root monomial appears in the $\overrightarrow{\rho_{p-1}^i,w}$-column.
\end{Lemma}
\begin{proof}
By Lemma \ref{root} the coefficient of a root monomial in the entry $\mathfrak{A}^p_{ws}$ is given by $c^s_{\rho_{p-1}^i,w}$. Since the columns and rows of $\mathfrak{A}^p$ are naturally parameterized by $E$ and $V$, respectively, the decomposition \eqref{eigen_decomp} of the symbol $\mathfrak{m}$ implies that $c^s_{\rho_{p-1}^i,w}$ vanishes whenever $s$ and $w$ belong to different graphs. On the other hand, for a fixed $\Gamma_i$, if $s\in E_i$ and $w \in V_i$, then \eqref{struct_quasinorm} and our choice of rows in $T$ imply $c^s_{\rho_{p-1}^i,w}=1$ when $s=\overrightarrow{\rho_{p-1}^i,w}$
and vanishes otherwise, which completes the proof.
\end{proof}
The previous lemma implies that the product of all of the root monomials in the entries of the core sub-matrix $T$ will appear (with coefficient $\pm 1$)  in the Leibniz expansion \footnote{From now on, the classical formula for determinants in terms of permutations of the matrix elements is referred  as the \emph{Leibniz expansion} for determinants. } for $\det T$.
For a given $\Gamma_i$, we denote the product of all of the root monomials of $\Gamma_i$ which appear in the entries of $T$ by $p_i(U)$. By Lemma \ref{coeff=1}, the root monomial in the entry $\mathfrak{A}^p_{w,\overrightarrow{\rho_{p-1}^i,w}}$ is equal to $u_{\rho_0^i} u_{\overrightarrow{\rho_0^i\rho_1^i}}\ldots u_{\overrightarrow{\rho_{p-2}^i,\rho_{p-1}^i}}$. Therefore, if $E_i\ne\varnothing$, then, setting $d_i=\dim \mathfrak m_{-2}^i$, we get
\begin{equation}
\label{p_i}
p_i(U)=u_{\rho ^i_0}^{d_i}\displaystyle{\prod_{j=1}^{\omega_i} \left(u_{\overrightarrow{\rho^i_{j-1} \rho^{i}_{j}}}\right)^{d_i-r^i_0-\ldots-r^i_{j-1}}.}
\end{equation}
If $E_i=\varnothing$, then the corresponding component contributes no rows to the core sub-matrix, and in this case we set $p_i(U)=1$.

If \begin{equation}
\label{p}
p(U)=\displaystyle{\prod_{i=1}^k} p_i(U)
\end{equation}
is the product of all root monomials appearing in the entries of the core sub-matrix $T$, then the coefficient of $p$ in the core minor $\det T$ is particularly simple.
 \begin{Lemma}
 \label{pcoeff1}
 The monomial  $p(U)$ appears in the Leibniz expansion of the core minor $\det T$ only in the term
 \begin{equation}
 \label{diag_term_det_T}
\prod_{i=1}^k\prod_{p=1}^{\omega_i+1}\prod_{{\begin{array}{c}v\in V_i,\\ \mathrm{depth}(v)>p-1\end{array}}} \mathfrak{A}^p_{v\overrightarrow{\rho_{p-1}^i,v}}.
 \end{equation}
 Moreover, the coefficient of $p(U)$ in the core minor  $\det T$ is $\pm 1$.
 \end{Lemma}
\begin{proof}
Recall that $d_i=\dim\mathfrak m^i_{-2}$. By Corollary \ref{block_diag}, after permuting rows and columns, the core sub-matrix $T$ is block diagonal, with a $d_i\times d_i$ block $T_i$ corresponding to each component $\Gamma_i$. Since these permutations affect only the overall sign, it is enough to prove the assertion for the coefficient of $p_i(U)$ in $\det T_i$.

Consider a term in the Leibniz expansion of $\det T_i$ whose monomial part is $p_i(U)$, and write the selected monomial from each of the $d_i$ rows as $f_1(U),\ldots,f_{d_i}(U)$, so that
\begin{equation}
\label{prod}
\prod_{k=1}^{d_i}f_k(U)=C p_i(U)
\end{equation}
for some coefficient $C$. Every entry of $\mathfrak A$ is linearly vertex homogeneous. Hence every $f_k$ contains exactly one vertex variable. Since $p_i(U)$ contains no vertex variable other than $u_{\rho_0^i}$ and has $u_{\rho_0^i}$-degree $d_i$, the unique vertex factor of every $f_k$ is $u_{\rho_0^i}$.

For brevity, set
\[
\varepsilon_r:=\overrightarrow{\rho_r^i\rho_{r+1}^i},
\qquad 0\leq r\leq\omega_i-1.
\]
There are exactly $d_i-r_0^i$ selected factors which come from layers $\mathfrak A^p$ with $p\geq2$. By statement (1) of Lemma \ref{root}, each of these factors is divisible by $u_{\varepsilon_0}$. The exponent of $u_{\varepsilon_0}$ in $p_i(U)$ is also $d_i-r_0^i$ by \eqref{p_i}; therefore $u_{\varepsilon_0}$ occurs exactly once in every selected factor from a layer $p\geq2$.

We proceed by induction. Fix $r\geq1$ and assume that, for every $q<r$, each selected factor coming from a layer $\mathfrak A^p$ with $p\geq q+2$ contains $u_{\varepsilon_q}$ exactly once. A factor coming from a layer $p\geq r+2$ consequently contains $u_{\varepsilon_0},\ldots,u_{\varepsilon_{r-1}}$ linearly, so statement (2) of Lemma \ref{root}, with $j=r-1$, shows that it is divisible by $u_{\varepsilon_r}$. The number of such factors is
\[
d_i-r_0^i-\cdots-r_r^i,
\]
which is exactly the exponent of $u_{\varepsilon_r}$ in $p_i(U)$. Hence $u_{\varepsilon_r}$ occurs exactly once in every selected factor from a layer $p\geq r+2$.

It follows that a selected factor from the layer $\mathfrak A^p$ contains precisely
\[
u_{\rho_0^i}u_{\varepsilon_0}\cdots u_{\varepsilon_{p-2}}.
\]
No other edge variable can occur, because the product in \eqref{prod} is $p_i(U)$. Thus every $f_k(U)$ is the root monomial corresponding to its row. By Lemma \ref{coeff=1}, that root monomial occurs in exactly one column of the row and has coefficient $1$. Consequently, the selected columns are uniquely determined, and the only possible Leibniz term is the diagonal-type term \eqref{diag_term_det_T}. Its coefficient in the determinant is the sign of the corresponding permutation, hence $\pm1$. Multiplying the results obtained for the individual blocks $T_i$, we complete the proof of the lemma.
\end{proof}

\noindent\textbf{Full rank and rational consistency.}
We now derive two consequences of Lemma \ref{pcoeff1} which will be used repeatedly in the sequel. First, by this lemma the core minor is a nonzero polynomial in the fiber variables. Consequently, $\mathfrak A$ has full column rank over the field of rational functions in the fiber variables. The same conclusion holds for each diagonal block of this matrix separately. Note that an abelian block, i.e., a block corresponding to a component without edges, contributes neither columns nor rows of the core sub-matrix; by convention, its determinant and its factor $p_i$ are set equal to $1$.

Second, let $d$ be an arbitrary polynomial last column. If all maximal minors of the corresponding augmented matrix vanish, then the system $\mathfrak A\Psi=d$ is rationally consistent, i.e., it has a solution whose entries are rational functions in the fiber variables. Indeed, choose the core rows and solve the corresponding square system by Cramer's rule, which is possible because its determinant is nonzero. Then append to the core rows any other row. The corresponding augmented determinant is equal, up to a sign, to $\det T$ multiplied by the residual of the appended row, that is, by the difference between its right-hand side and the value of its left-hand side at the Cramer solution. Since all these augmented determinants vanish and $\det T\ne0$, every residual vanishes, i.e., the Cramer solution satisfies all equations of the system. Note that each row is tested by a single finite minor, so that the argument applies also to the infinite system under consideration.

The maximal minors of the fundamental algebraic system are determinants of $(n-m+1)\times(n-m+1)$ sub-matrices. Since the matrix $T$, constructed above, is a particular $(n-m)\times(n-m)$ sub-matrix of $[\mathfrak{A}|\tilde{\mathfrak{b}}]$, we must augment $T$ by an additional row and column of the fundamental algebraic system if its determinant is to be a maximal minor of the fundamental algebraic system. There is no freedom in choosing a column as the only remaining column is $\tilde{b}$ itself. However, the choice of row in $\mathfrak{A}$ used to augment $T$ is crucial for our analysis.
\begin{Def}
Given an integer $p\geq 1$ and a vertex $w$ whose row in layer $p$ is not one of the rows of the core sub-matrix, the sub-matrix of $[\mathfrak{A}|\tilde{\mathfrak{b}}]$ obtained by augmenting $T$ by the column $\tilde{\mathfrak{b}}$ and the $w$--row of $\mathfrak{A}^p$ is called the \emph{$(p, w)$--augmentation of the core sub-matrix $T$} and will be denoted by $W^p_{w}$ .
\end{Def}
\section{Vanishing of horizontal cross-component structure functions}
\label{transverse_vanishing_sec}
In this section, we prove items (2) and (3) of Proposition \ref{outline_transverse_vanishing}, i.e., the vanishing of the horizontal cross-component structure functions (see Definition \ref{cross_component_def}) whose output belongs to a component different from the components of both inputs.
\begin{Prop}
\label{cofterm}
    Let $v,l \in V$, and let $w \in V_i$ be a vertex of depth $0$. Then the monomial
    \begin{equation}
    \label{vl1_monom}
    u_vu_lp(U),
    \end{equation}
    where $p(U)$ is as in \eqref{p}, appears only in the term $\tilde{\mathfrak{b}}^1_w\det T$ of the cofactor expansion of $\det W^1_w$.
\end{Prop}
\begin{proof}
First note that if $E_i=\varnothing$, then all entries of the augmenting row, except possibly the last one, are equal to zero. Indeed, by the block-diagonal structure of the coefficient matrix, the nonzero entries of this row could only appear in the columns indexed by the edges of the component containing the vertex of this row, and there are no such edges. Consequently, the cofactor expansion along the augmenting row consists of the single term $\pm\tilde{\mathfrak b}^1_w\det T$, and there is nothing to prove. Therefore, from now on we assume that this component, which we call the source component, is nonabelian.

We will analyze $\det W^1_w$ by two consecutive cofactor expansions: first along the column corresponding to $\tilde{\mathfrak{b}}$ and then along the $w$-row for each cofactor of the first step, except the one obtained from removing this column and the $w$-row. In more detail, given $p\geq 1$, a vertex $j$ of depth greater than $p-1$, and an edge $y$, denote by $T_{\hat j\hat y}^p$ the matrix obtained from the matrix $T$ by removing the $j$-row of $\mathfrak{A}^p$ and the $y$-column of $\mathfrak{A}$.
Then,
up to a sign, the general term in the resulting expansion is given by
\begin{equation}
\label{cofactor}
    \tilde{\mathfrak{b}}^p_j\mathfrak{A}^1_{wy}\det T_{\hat{j}\hat{y}}^p, \quad  j\neq w,\,\mathrm{depth}(j)>p-1,
\end{equation}
or
\begin{equation}
\label{cofactor1}
\tilde{\mathfrak{b}}^1_w \det T.
\end{equation}
We proceed by excluding terms of the form \eqref{cofactor}. By item (1) of Corollary \ref{block_diag}, it follows that $\mathfrak{A}^1_{wy} = 0$ unless $y$ belongs to the same component $\Gamma_i$ as $w$. Using also Remark \ref{block_diag_rem}, one gets that $\det T_{\hat{j}\hat{y}}^p = 0$ unless $y$ and $j$ belong to the same component $\Gamma_i$. The latter conclusion follows from the fact that a cofactor of a block-diagonal matrix corresponding to an entry off the block-diagonal block is equal to zero. In summary, the terms \eqref{cofactor} vanish unless the vertex $j$ and the edge $y$ belong to the same component $\Gamma_i$ (as the vertex $w$).

Moreover, since by the first line of \eqref{elem.A.frak} and \eqref{qjk} $\mathfrak{A}^1_{wy}=\sum\limits_{x \in V}c^y_{xw}u_x$, from \eqref{ad_X_depth_p_i} and the assumption that $w$ has depth $0$ it follows that \eqref{cofactor} also vanishes unless $y$ is of the form
\begin{equation}
\label{y_bar z}
y=\overrightarrow{\rt 0 \bar z}, \quad \text{ for some } \bar z \in V_i \text { with } \mathrm{depth} (\bar z)>0.
\end{equation}

Since $j$ and $y$ belong to the same component $\Gamma_i$, by second item of Corollary \ref{block_diag}  $\Gamma_i$-vertex degree of $\det  T_{\hat{j}\hat{y}}^p$ is equal to $d_i-1$.
By Lemma \ref{coeff=1} the monomial $u_{\rt 0}$ appears exactly in one entry of each row of the matrix $T$ belonging to $\mathfrak{A}^1$. Moreover, in the $\bar z$-row of $\mathfrak{A}^1$ it appears precisely in the entry $\mathfrak{A}^1_{\bar z \overrightarrow{\rt 0 \bar z}}$. Consider the following two cases separately:

{\bf Case 1.} \emph{In \eqref{cofactor} either $p>1$, or $p=1$ and $j\neq \bar z$, where $\bar z$ is as in \eqref{y_bar z}.} In this case $\bar z$-row of $\mathfrak{A}^1$, with the entry $\mathfrak{A}^1_{\bar z \overrightarrow{\rt 0 \bar z}}$ excluded, is a row of  the matrix $T_{\hat{j}\hat{y}}^p$. Consequently,
 this row does not depend on $u_{\rt 0}$. Therefore, monomials in the Leibniz expansion of  $\det T_{\hat{j}\hat{y}}^p$ will contain at most $d_i-2$ factors of $u_{\rt 0}$. On the other hand,  by \eqref{p_i} and \eqref{p} the monomial \eqref{vl1_monom} contains at least $d_i$ such factors. Hence, the monomials of the factor $\tilde{\mathfrak{b}}^p_j\mathfrak{A}^1_{wy}$ of \eqref{cofactor}, contributing to the monomial \eqref{vl1_monom}, must be  divisible by $u_{\rt 0}^2$. Since $w$ is of depth $0$, using  the first line of \eqref{elem.A.frak} and \eqref{qjk}, $\mathfrak{A}^1_{wy}$ is independent of $u_{\rt 0}$.
 Therefore, the monomials of the factor of $\tilde{\mathfrak{b}}^p_j$ of \eqref{cofactor} contributing to the monomial \eqref{vl1_monom} must be divisible by  $u_{\rt 0}^2$.

 Now consider two sub-cases:

 {\bf Case 1.1.} \emph{$p=1$:}
 By Lemma \ref{bu2}, $\tilde{\mathfrak{b}}^1_j$ does not contain the monomial $u_{\rt 0}^2$, so this case does not contribute to monomial \eqref{vl1_monom}.

{\bf Case 1.2.} \emph{$p>1$:}
 By Lemma \ref{bu2}, if $u_{\rt 0}^2$ divides a monomial of  $\tilde{\mathfrak{b}}^p_j$, then that monomial must depend on an edge variable indexed by a different component $\Gamma_{t'}$ of $\Gamma$ with $t' \neq i$. In other words, by Remark \ref{homogen_total} and Lemma \ref{lin_quad_lem}, the $\Gamma_i$-edge degree of this monomial of  $\tilde{\mathfrak{b}}^p_j$ is strictly smaller than $p-1$.

 Let us show that in this case the term \eqref{cofactor} cannot contribute to the monomial \eqref{vl1_monom}. Indeed, by \eqref{p_i} and \eqref{p} the monomial \eqref{vl1_monom} has $\Gamma_i$-edge degree equal to $\sum_{k=1}^{\omega_i} kr_k^i$. On the other hand,  by Remark \ref{degree_rk}, $\det T_{\hat{j}\hat{y}}^p$ has $\Gamma_i$-edge degree equal to $\sum_{k=1}^{\omega_i} kr_k^i-(p-1)$, because, by the second item of Corollary \ref{block_diag}, the $\Gamma_i$-edge  degree of non-zero entries of the $j$-row of $\mathfrak{A}^p$ is equal to $p-1$. Besides, by Remark \ref{homogen_total} and Lemma \ref{lin_quad_lem} again, the (total) edge degree of the entry $\tilde{\mathfrak{b}}_j^p$ is equal to $p-1$, and $\mathfrak{A}^1_{wy}$ does not depend on edge variables. Therefore, if the monomial of the factor $\tilde{\mathfrak{b}}_j^p$ of \eqref{cofactor} contributes to the monomial \eqref{vl1_monom}, then the $\Gamma_i$-edge degree of this monomial (of $\tilde{\mathfrak{b}}_j^p$) must be exactly equal to $p-1$, which proves the statement in the first sentence of the paragraph.

{\bf Case 2.} \emph{In \eqref{cofactor} $p=1$ and $j=\bar z$.}
In this case, the monomials in the Leibniz expansion of  $\det T_{\hat{j}\hat{y}}^1$ will contain at most $d_i-1$ factors of $u_{\rt 0}$, and by the same arguments as in the first paragraph of Case 1 above, the monomial from the factor $\tilde{\mathfrak{b}}^1_j$ of \eqref{cofactor} contributing to \eqref{vl1_monom} must contain $u_{\rt 0}$. Recall from \eqref{tb1} that the entry $\tilde{\mathfrak{b}}^1_{\bar{z}}$ is given by

$$\sum \limits_{{\begin{array}{c} s,x \in V\\ \mathcal C(s)\neq \mathcal C(\bar z)\end{array}}} (\alpha(\bar{z})^2-\alpha(s)^2)c^s_{x\bar{z}}u_su_x.$$

First this implies that  $s\neq \rt 0$.
On the other hand, if $x=\rt 0$, then $c^s_{x\bar z}=0$ by \eqref{struct_quasinorm}. Thus $\tilde{\mathfrak{b}}^1_{\bar z}$ is independent of $u_{\rt 0}$ and $u_vu_lp(U)$ doesn't appear.

We conclude that no terms of the form \eqref{cofactor} contribute to the coefficient of the monomial \eqref{vl1_monom} in $\det W_w^1$.
\end{proof}

In the sequel, given a collection $\mathcal F$ of smooth functions on an open set $\mathcal U$ of the manifold $M$, where the projective equivalence of our sub-Riemannian metrics $g_1$ and $g_2$ holds and the local quasi-normal frame $\{X_s\}_{s\in V\cup E}$ is defined, denote by $\mathrm{Ideal}(\mathcal F)$ the ideal in $C^\infty(\mathcal U)$ generated by the collection $\mathcal F$.

\begin{Prop}
 \label{c0}
 If $v\in V_i$ is a vertex of depth $0$ and $l \in V_t$ with $t\neq i$, then the coefficient of the monomial \eqref{vl1_monom} in the determinant of the $(1,\rt 0)$-augmentation $W_{\rt 0}^1$ of the matrix $T$ is equal, up to a sign, to
 \begin{equation}
    \label{c0v}
       \begin{cases}\bigl(\alpha(\rt 0)^2-\alpha(l)^2\bigr)c^l_{\rt 0v},& \text{ if } l=\rho_0^t,\\
       \bigl(\alpha(\rt 0)^2-\alpha(l)^2\bigr)c^l_{\rt 0v} \,\, \mathrm{mod}\,\mathrm{Ideal}(c^{\rho_0^t}_{\rt 0v}), & \text{ if } l\in V_t, l\neq \rho_0^t
       \end{cases}
    \end{equation}
Therefore, if $\det W_{\rt 0}^1=0$, then
\begin{equation}
\label{crt_depth0}
c^l_{\rt 0v}=0, \qquad v\in V_i,\ \mathrm{depth}(v)=0,\ l\in V_t,\ t\neq i.
\end{equation}
\end{Prop}
Note that relation \eqref{crt_depth0} holds also for every $v\in V_i$ of positive depth, by \eqref{struct_quasinorm}, because $\overrightarrow{\rt 0 v}\in E$ in this case. Note also that in Proposition \ref{c0} the source component $\Gamma_i$ is nonabelian, since its initial root appears in the statement. If the target component, i.e. the component containing the upper index of the structure functions under consideration, is abelian, then $p_t=1$. In this case there is no alternative involving the root of the target component, and the same formula holds exactly, that is, without the ideal term.
\begin{proof}
First, let us consider the case when the target component is abelian. Then its block contributes neither vertex nor edge variables to the core minor $\det T$. Therefore, comparing the vertex degrees with respect to the target component, we see that the vertex variable $u_l$ must come from the factor taken from the last column. Consequently, in the calculation given below only case (A) can occur, and it gives the exact formula. Hence, in the rest of the proof we may assume that the target component is nonabelian.

 By Proposition \ref{cofterm}, we only need to consider the coefficient of the monomial \eqref{vl1_monom} in the term
\begin{equation}
\label{only_term1}
\tilde{\mathfrak{b}}^1_{\rt 0} \det T.
\end{equation}
of the cofactor expansion of $\det W^1_{\rt 0}$. By \eqref{p_i} and \eqref{p}, the following vertex variables appear
\begin{equation}
\label{vl_vertex}
\{u_{\rho_0^j}\}_{j=1}^k, u_v, \text {and } u_l
\end{equation}
 in the monomial \eqref{vl1_monom}. Further, in this monomial the number of  $\Gamma_j$-vertex variable factors (see Definition \ref{var_def}) is equal to  $d_j+1$ if $j\in \{i, t\}$ and to $d_j$ otherwise, while Remark \ref{degree_rk}
implies that in the monomials appearing in the core minor $\det T$ the number of  $\Gamma_j$-vertex variable factors is equal to $d_j$ for all $j\in [1:k]$. Therefore, if a monomial of $\tilde{\mathfrak{b}}^1_{\rt 0}$ contributes to \eqref{vl1_monom} in \eqref{only_term1}, then one of its factor is $\Gamma_i$ -vertex variables, i.e., by \eqref{vl_vertex}, either $u_{\rt 0}$ or $u_v$  and the other is $\Gamma_t$ vertex variable, i.e.
either $u_{\rho_0^t}$ or $u_l$.

By \eqref{tb1}
\begin{equation}
\label{tb1_rho}
\tilde{\mathfrak{b}}^1_{\rt 0}=\sum_{\substack{s,x\in V\\\mathcal C(s)\ne i,\ x\ne s}} (\alpha (\rt 0)^2-\alpha (s)^2)c^s_{x\rt 0}u_xu_s.
\end{equation}
Since $\mathcal C(s)\neq i$, by the previous paragraph we have:
\begin{enumerate}
\item
$s\in V_t$, therefore $s\in  \{\rho_0^t, l\}$;
\item $x\in V_i$,  therefore $x\in \{\rho_0^i, v\}$, but since $c_{\rt 0\rt 0}^s=0$ , we conclude that $x=v$.
\end{enumerate}
According to item (1) above we have the following two cases:

{\bf (A)} $s=l$. In this case $\tilde b^1_{\rt 0}$ contributes  to monomial \eqref{vl1_monom} in the product \eqref{only_term1} the coefficient of $u_v u_l$, which by \eqref{tb1_rho} is equal, up to a sign,
to
\begin{equation}
\label{b^1_rt_contrib}
\bigl(\alpha(\rt 0)^2-\alpha(l)^2\bigr)c^l_{\rt 0v}.
\end{equation}
Therefore, the contribution of the core minor $\det T$ to the coefficient of monomial \eqref{vl1_monom} in the same product comes from the coefficient of $p(U)$  in the core minor $\det T$. By Lemma \ref{pcoeff1}, the latter coefficient is equal to $\pm1$. Hence, the contribution of the considered case to the coefficient of monomial  \eqref{vl1_monom} in the product \eqref{only_term1} is equal, up to a sign, to the \eqref{b^1_rt_contrib}.

{\bf (B)} $s=\rho_0^t$. In this case $\tilde b^1_{\rt 0}$  contributes  to monomial \eqref{vl1_monom} in the product \eqref{only_term1} the coefficient of $u_v u_{\rho_0^t}$ and this coefficient contains $c^{\rho_0^t}_{\rt 0v}$ as a factor. Therefore, the contribution of the considered case to the coefficient of monomial  \eqref{vl1_monom} in the product \eqref{only_term1} belongs to  $\mathrm{Ideal}(c^{\rho_0^t}_{\rt 0v})$.

Taking into account that cases (A) and (B) coincide when $l=\rho_0^t$, we obtain \eqref{c0v}. If $\det W_{\rt 0}^1=0$, the first line of \eqref{c0v} first gives $c^{\rho_0^t}_{\rt 0v}=0$. The ideal term in the second line therefore vanishes, and that line gives $c^l_{\rt 0v}=0$ for every remaining $l\in V_t$. This proves \eqref{crt_depth0}.
\end{proof}

 \begin{Coro}
 \label{brt0}
 If relation \eqref{crt_depth0} holds, then the entries $\tilde{\mathfrak{b}}_{j}^1$ for $j \in V_i$ do not depend on $u_{\rt 0}$.
 \end{Coro}
\begin{proof}
    This follows immediately from the assumption \eqref{crt_depth0} and
    equations \eqref{tb1} and \eqref{struct_quasinorm}.
\end{proof}

\begin{Prop}\label{c1}
Let $\Gamma_i$ be nonabelian, let $v,w\in V_i$ with $w$ of depth $0$, let $l\in V_t$ with $t\ne i$, and assume that relation \eqref{crt_depth0} holds. Then, up to the overall sign of the determinant, the coefficient of the monomial $u_vu_lp(U)$ in $\det W_w^1$ is equal to
\begin{equation}\label{c1v}
(\alpha_i^2-\alpha_t^2)c^l_{wv}
\pmod{\operatorname{Ideal}(c^{\rho_0^t}_{wv})}
\end{equation}
provided that $\Gamma_t$ is nonabelian and $l\ne\rho_0^t$. For $l=\rho_0^t$, as well as for every $l$ in an abelian target component, the same formula holds exactly, that is, without the ideal term. Consequently, if the augmented minors under consideration vanish, then
\begin{equation}\label{cw_depth0}
c^l_{wv}=0\qquad(v,w\in V_i,\ \mathrm{depth}(w)=0,\ l\in V_t,\ t\ne i).
\end{equation}
\end{Prop}
\begin{proof}
By Proposition \ref{cofterm}, in the analysis of the determinant we only need to consider the term $\tilde{\mathfrak b}^1_w\det T$ of its cofactor expansion. By Corollary \ref{brt0}, the factor of this term taken from the last column does not depend on the internal root variable, i.e. on the vertex variable of the initial root of the source component. Therefore, comparing the vertex degrees with respect to each component in the same way as in the proof of the previous proposition, we see that the monomials of this factor contributing to the monomial under consideration contain $u_v$ and, in addition, either $u_l$ or the external root variable $u_{\rho_0^t}$, i.e. the vertex variable of the initial root of the target component.

In the first case, the remaining factor $p(U)$ must be supplied by the core determinant, where its coefficient is given by Lemma \ref{pcoeff1}; hence, this choice contributes the main term stated in the proposition. In the second case, the contribution is a multiple of $c^{\rho_0^t}_{wv}$. This contribution does not vanish a priori, and it is responsible for the ideal term in the formula. If the target component is abelian, then $p_t=1$, so that the second choice does not occur at all, and the formula is exact. The formula is exact also when $l=\rho_0^t$, since in this case the two choices coincide.

It remains to deduce the last statement of the proposition. Assume that the minors under consideration vanish. Applying the exact formula to the root of the target component, we first obtain that the structure function generating the ideal is zero (recall that the eigenvalues corresponding to distinct components are distinct). Consequently, the ideal term vanishes, and the formula for the other values of $l$ gives the required relations. This completes the proof of the proposition.
\end{proof}

\begin{Def}
\label{depth_path_def}
Let $w\in V_i$ have depth $d>0$. For $0\leq r\leq d-2$ set
\[
\varepsilon_r^i:=\overrightarrow{\rho_r^i\rho_{r+1}^i},
\]
and set
\[
\varepsilon_{d-1}^i(w):=\overrightarrow{\rho_{d-1}^i w}.
\]
The monomial
\begin{equation}
\label{Pw_def}
P_w(U):=
\left(\prod_{r=0}^{d-2}u_{\varepsilon_r^i}\right)u_{\varepsilon_{d-1}^i(w)}
\end{equation}
is called the \emph{depth path monomial from $\rho_0^i$ to $w$}. When $d=1$, the empty product is understood to be $1$, so that
$P_w(U)=u_{\overrightarrow{\rho_0^i w}}$.
\end{Def}

\begin{Lemma}
\label{augment_path}
Let $w\in V_i$ have depth $d>0$. In the $\rho_0^i$-row of
$\mathfrak A^{d+1}$, the monomial
\begin{equation}
\label{augment_path_monom}
u_{\rho_0^i}P_w(U)
\end{equation}
occurs only in the entry
\[
\mathfrak A^{d+1}_{\rho_0^i,\overrightarrow{\rho_0^i w}},
\]
and its coefficient is equal to $-1$.
\end{Lemma}
\begin{proof}
Put
\[
z_0=\rho_0^i,\qquad z_r=\rho_r^i\ (1\leq r\leq d-1),
\qquad z_d=w,
\]
and $e_r=\overrightarrow{z_rz_{r+1}}$ for $0\leq r\leq d-1$.
By Lemma \ref{minvert}, a nonzero summand in the coefficient of
$u_{\rho_0^i}u_{e_0}\cdots u_{e_{d-1}}$ in an entry
$\mathfrak A^{d+1}_{\rho_0^i,y}$ has the form
\begin{equation}
\label{augment_chain}
c^y_{v_d,\rho_0^i}
\prod_{r=0}^{d-1}c^{e_{\pi(r)}}_{v_rv_{r+1}},
\qquad v_0=\rho_0^i,
\end{equation}
for a permutation $\pi$ of $\{0,\ldots,d-1\}$ describing the order in which the edges are used along the chain. We will show that $\pi$ must be the identity and that all intermediate vertices are uniquely determined.

More precisely, we claim that every nonzero summand forces $v_r=z_r$ for
$0\leq r\leq d$. At the first step, an edge $e_s$ with $s>0$ cannot occur, because $v_0=\rho_0^i$ has depth $0<s$ and item (2) of Lemma \ref{root_edge} would make the corresponding factor vanish. Thus the first edge is $e_0$, and item (1) of Lemma \ref{root_edge} gives $v_1=z_1$. Suppose inductively that $v_r=z_r$. The edges $e_0,\ldots,e_{r-1}$ have already been used. An unused edge $e_s$ with $s>r$ cannot occur next, because $z_r$ has depth $r<s$; item (2) of Lemma \ref{root_edge} again gives zero. Hence the next edge is $e_r$, and item (1) gives $v_{r+1}=z_{r+1}$. This proves the claim.

Consequently $v_d=w$, and the remaining factor in
\eqref{augment_chain} is $c^y_{w,\rho_0^i}$. Since
\[
[X_{\rho_0^i},X_w]=X_{\overrightarrow{\rho_0^i w}},
\]
quasi-normality and skew-symmetry imply that this coefficient vanishes unless
$y=\overrightarrow{\rho_0^i w}$, in which case it equals $-1$.
Thus \eqref{augment_path_monom} occurs in the asserted entry with coefficient $-1$ and nowhere else in the $\rho_0^i$-row.
\end{proof}

\begin{Lemma}
\label{mixed_b_chain}
Let $j,x\in V_i$ and $s\in V_t$ with $t\neq i$.
Suppose that
\begin{equation}
\label{mixed_b_monom}
K\,u_xu_su_{e_1}\cdots u_{e_{p-1}}
\end{equation}
is a monomial of $\tilde{\mathfrak b}^{p}_j$, where
$e_1,\ldots,e_{p-1}\in E_i$. Then $K$ is a sum of terms of the form
\begin{equation}
\label{mixed_b_coeff}
(\alpha_i^2-\alpha_t^2)c^s_{x_0j}
\prod_{r=1}^{p-1}c^{e_r}_{x_rx_{r-1}},
\qquad x_{p-1}=x,
\end{equation}
where the intermediate vertices $x_0,\ldots,x_{p-2}$ lie in
$V_i$, and the sum is taken also over all distinct ordered words formed by the edges with the prescribed multiplicities, the edges in each word being listed in the order in which they are introduced by the recursion. As in the expansion of the entries of the coefficient matrix, repeated copies of the same edge are not labelled separately.
\end{Lemma}
\begin{proof}
We argue by induction on the layer. For $p=1$, the assertion is exactly \eqref{tb1}. Assume that it holds in layer $p$, and consider the recursion \eqref{tbs_frak}. The second summand in \eqref{tbs_frak} contains a factor $u_q$ with
$q\in E_{t'}$ for some $t'\neq i$, and hence cannot contribute to a monomial of the form \eqref{mixed_b_monom}, whose edge factors all lie in $E_i$. Thus only
$(\vec h_1(\tilde{\mathfrak b}^{p}_j))_{(2)}$ can contribute.

By Lemma \ref{h_1_degree}, preservation of vertex degree $2$ forces
$\vec h_1$ to act on one of the two vertex-variable factors. It cannot act on the external factor $u_s$: by the decomposition \eqref{eigen_decomp}, the linear-vertex part of $\vec h_1(u_s)$ has its edge variable in $E_t$, whereas every edge in \eqref{mixed_b_monom} lies in $E_i$. Therefore $\vec h_1$ acts on the $\Gamma_i$-vertex factor. Using \eqref{principal_delta}, a factor $u_{x_{p-1}}$ is replaced by
\[
\sum_{x_p\in V_i}\sum_{e_p\in E_i}
c^{e_p}_{x_px_{p-1}}u_{x_p}u_{e_p}.
\]
Multiplying this factor by the induction hypothesis gives precisely
\eqref{mixed_b_coeff} in layer $p+1$. This completes the induction.
\end{proof}

For the next coefficient calculation, fix $d>0$, a nonabelian component
$\Gamma_i$, and $s\in V_t$, $t\ne i$. Let $N_s$ be the matrix obtained from the
$\Gamma_i$-block of $T$ by adjoining the $\rho_0^i$-row of
$\mathfrak A^{d+1}$ and the corresponding entries of the last column,
retaining in that column only monomial terms of the form
$u_su_x\prod_{e\in E_i}u_e^{m_e}$ with $x\in V_i$.
Thus $N_s$ has $d_i+1$ rows and columns. This notation is used only in
the following lemma and its application to the full augmented minor.

\begin{Lemma}[Coefficients of depth path monomials]
\label{augmented_path_rigidity}
Let $w,v\in V_i$ satisfy $\mathrm{depth}(w)=d>0$ and
$0<\mathrm{depth}(v)\le d$. Assume that
\begin{equation}\label{crt_depth0_s}
c^s_{\rho_0^i x}=0
\qquad\text{for every }x\in V_i\text{ of depth }0.
\end{equation}
Then the coefficient of
\begin{equation}\label{vs_path_monom}
u_vu_sP_w(U)p_i(U)
\end{equation}
in $\det N_s$ is, up to an overall sign,
\begin{equation}\label{path_coeff_formula}
(\alpha_i^2-\alpha_t^2)
\begin{cases}
c^s_{vw}+Q_w c^s_{v\rho_d^i},
  &d\le\omega_i\text{ and }w\ne\rho_d^i,\\
c^s_{vw},&\text{otherwise}.
\end{cases}
\end{equation}
In the first case, $Q_w$ is a polynomial with integer coefficients in
$c^e_{xy}$, $e\in E_i$, $x,y\in V_i$, independent of $v$ and $s$.
\end{Lemma}

\begin{proof}
Write $\rho_r=\rho_r^i$,
$\varepsilon_r=\overrightarrow{\rho_r\rho_{r+1}}$ for
$0\le r<\omega_i$, and
$f=\overrightarrow{\rho_{d-1}w}$.
We refer to a row of $T$ by $(p,z)$, where $p$ is its layer and
$\mathrm{depth}(z)\ge p$; the additional row is $(d+1,\rho_0)$.

We use the coefficient formulas of Lemmas \ref{minvert} and
\ref{mixed_b_chain}. Read from the vertex-variable end, a summand in
either formula is a product
\begin{equation}\label{path_product}
c^a_{x_{p-1}z}\prod_{r=1}^{p-1}c^{e_r}_{x_{r-1}x_r},
\qquad x_0=x,
\end{equation}
corresponding to the vertex factor $u_x$ and an ordering of the edge
factors. For an entry of $\mathfrak A$, $a$ is its column index;
for the last column, $a=s$ and the product is multiplied by
$(\alpha_i^2-\alpha_t^2)u_s$. Repeated copies of an edge are not
labelled separately. Thus we may compute the required coefficient by
choosing a determinant term, monomials in its entries, and products
\eqref{path_product} for their coefficients.

Two previously established facts control these products. If
$e\in\mathrm{Out}(\rho_r)$ and $c^e_{xy}\ne0$, then
$\mathrm{depth}(x),\mathrm{depth}(y)\ge r$, by Lemma
\ref{root_edge}. If one lower index is $\rho_r$, quasi-normality
and \eqref{decrease} give
\begin{equation}\label{path_root_factors}
\begin{aligned}
c^a_{\rho_r z}&=
\begin{cases}1,&a=\overrightarrow{\rho_r z},\\0,&\text{otherwise},\end{cases}
&&\mathrm{depth}(z)>r,\quad a\in V\cup E,\\
c^e_{\rho_r z}&=0,
&&\mathrm{depth}(z)\le r,\quad
e\in\mathrm{Out}(\rho_k),\ k\ge r.
\end{aligned}
\end{equation}
In particular, the first line is zero for $a=s$.

\medskip\noindent\emph{1. The first layer and the additional row.}
In a term contributing to \eqref{vs_path_monom}, exactly one row
selects the last column. Since the vertex part of the monomial is
$u_{\rho_0}^{d_i}u_vu_s$, there are only two possibilities:
\begin{itemize}
\item[(A)] the last column supplies $u_vu_s$, and every other row
supplies $u_{\rho_0}$;
\item[(B)] the last column supplies $u_{\rho_0}u_s$, one other row
supplies $u_v$, and all remaining rows supply $u_{\rho_0}$.
\end{itemize}
The additional row cannot select the last column: its final factor
$c^s_{x\rho_0}$ vanishes by \eqref{crt_depth0_s} when $x$ has
depth zero, and by quasi-normality when $x$ has positive depth.
Its nonzero entries in $\mathfrak A$ lie only in the columns
$\overrightarrow{\rho_0x}$ with $\mathrm{depth}(x)>0$, since its
final factor there is $c^a_{x\rho_0}$.

There are $r_0^i$ such columns and $r_0^i$ first-layer rows of $T$.
By Lemma \ref{coeff=1}, a first-layer row $(1,z)$ supplying
$u_{\rho_0}$ selects the column $\overrightarrow{\rho_0z}$.
The additional row uses one of these columns, so exactly one
first-layer row, say $(1,j)$, must make a different choice. There
cannot be two such rows in either (A) or (B). All other first-layer
rows therefore select their usual columns, and the additional row
selects $\overrightarrow{\rho_0j}$; its product
\eqref{path_product} ends at $j$. In (A), $(1,j)$ selects the last
column and supplies $(\alpha_i^2-\alpha_t^2)c^s_{vj}u_vu_s$.
In (B), $(1,j)$ supplies $u_v$ and selects a column of
$\mathfrak A$ outside $\mathrm{Out}(\rho_0)$; the last column
cannot be selected in the first layer, because
$c^s_{\rho_0z}=0$ for every positive-depth $z$.

\medskip\noindent\emph{2. Counting the available edge factors.}
We say that an edge of $\Gamma_i$ has \emph{level} $r$ if it belongs to $\mathrm{Out}(\rho_r)$; by the first of the two facts recalled above, a factor $c^e_{xy}$ with $e$ of level $r$ vanishes unless $\mathrm{depth}(x),\mathrm{depth}(y)\ge r$. In a contributing term, every row contributes one product \eqref{path_product}, the product of the row $(1,j)$ having no edge factors, and the edge factors of all these products, counted with multiplicity, are exactly the edge factors of \eqref{vs_path_monom}, because the monomials selected in the rows multiply to \eqref{vs_path_monom}.
Every row other than $(1,j)$ starts its product
\eqref{path_product} at $\rho_0$. A row $(p,z)$ of $T$ uses
$p-1$ edge factors, and the additional row uses $d$. By
\eqref{p_i} and \eqref{Pw_def}, for $0\le r\le\omega_i$ the number of
edge factors in \eqref{vs_path_monom} belonging to
$\mathrm{Out}(\rho_r)$, counted with multiplicity, is
\begin{equation}\label{path_edge_count}
\sum_{k=r+1}^{\omega_i}r_k^i+\mathbf 1_{\{r<d\}}.
\end{equation}
The same number counts the rows whose products use at least $r+1$
edges. The available edges of this type are $\varepsilon_r$ and, only when
$r=d-1$, $f$; $\varepsilon_r$ is absent if $r=\omega_i$.

We apply this equality of counts simultaneously to all the products.
The induction statement is that, after the first $r$ steps, every
product with at least $r$ edge factors has used one edge from each of
$\mathrm{Out}(\rho_0),\ldots,\mathrm{Out}(\rho_{r-1})$, in that order,
and all available factors from these sets have been exhausted. For
$r<d$, every product continuing beyond step $r$ is at $\rho_r$.
The statement is immediate for $r=0$. Assuming it at $r<d$, a product
continuing to step $r+1$ cannot use an earlier edge, since all such
factors have been exhausted, or a later edge, by Lemma \ref{root_edge}.
It must therefore use an edge of $\mathrm{Out}(\rho_r)$.
There are exactly as many such factors as there are products continuing
to this step, by \eqref{path_edge_count}, so each uses one and none
remain. Formula \eqref{path_root_factors} forces the next vertex to be
$\rho_{r+1}$ when $r+1<d$. At step $d$, the edge is either $f$, leading
to $w$, or $\varepsilon_{d-1}$, leading to $\rho_d$; this latter option exists
only when $d\le\omega_i$. This completes the induction.
In particular, after step $d$ every product with at least $d$ edge factors is at $w$ or at $\rho_d$. Unless $d\le\omega_i$ and $w=\rho_d$, the edge $f$ occurs only once in \eqref{vs_path_monom}, so exactly one of these products is at $w$; thus, if the additional row ends at $w$, no other row can use $f$.
Applied to the additional row, the induction gives
\begin{equation}\label{j_alternatives}
j\in\{w,\rho_d\}\quad(d\le\omega_i),
\qquad j=w\quad(d=\omega_i+1).
\end{equation}

\medskip\noindent\emph{3. Excluding (B).}
Suppose $(1,j)$ selects the column
$y=\overrightarrow{\rho_kx}$, $k\ge1$, with coefficient
$c^y_{vj}\ne0$. Lemma \ref{root_edge} gives
$k\le\min(\mathrm{depth}(v),\mathrm{depth}(j))\le d$.
If $d\le\omega_i$, $j=\rho_d$, and $k=d$, \eqref{path_root_factors} makes this
coefficient zero, since $\mathrm{depth}(v)\le d$.
Otherwise consider the row $(k+1,x)$ of $T$, which exists because
$\mathrm{depth}(x)>k$. Its product ends at $\rho_k$ by step 2.
For $k=d$ and $w\ne\rho_d$, this uses the fact that $j=w$ and
the additional row has already used the only factor $f$.
By \eqref{path_root_factors}, the row $(k+1,x)$ must then select
$y$ or give zero; selecting the last column also gives zero.
But $y$ is already selected by $(1,j)$. Hence (B) contributes
nothing.

\medskip\noindent\emph{4. Computing the contributions in (A).}
If $j=w$, the additional row supplies
$u_{\rho_0}P_w(U)$ in column $\overrightarrow{\rho_0w}$ with
coefficient $-1$, as in Lemma \ref{augment_path}. For a core row $(p,z)$ with $p-1\le d$, step 2 forces the product to follow the successive roots: if the edge leading to $w$ is not a root edge, its sole occurrence has already been used by the additional row. For the rows with $p-1>d$, we continue the same induction. After $r$ steps, all edge factors of levels less than $r$ have been exhausted by the count in \eqref{path_edge_count}. A product continuing from $\rho_r$ must therefore use an edge of level at least $r$; Lemma \ref{root_edge} bounds its level by $r$, and \eqref{path_root_factors} forces this edge to be $\varepsilon_r$, leading to $\rho_{r+1}$. Thus every core row $(p,z)$ follows the root product up to $\rho_{p-1}$. Since $\mathrm{depth}(z)\ge p$, its final factor in \eqref{path_root_factors} is $1$ in column $\overrightarrow{\rho_{p-1}z}$ and zero in every other column, including the last column. These choices determine a unique term of the determinant and a unique product contributing to its coefficient. Multiplying their factors gives
$\pm(\alpha_i^2-\alpha_t^2)c^s_{vw}$.

The only other possibility is $d\le\omega_i$ and
$j=\rho_d\ne w$. The first-layer row $(1,\rho_d)$ then supplies
$(\alpha_i^2-\alpha_t^2)c^s_{v\rho_d}u_vu_s$.
Delete that row and the last column. The contribution of the
remaining square matrix is the coefficient of $P_w(U)p_i(U)$ in
its determinant. The remaining matrix consists of entries of the $\Gamma_i$-block of $\mathfrak A$, whose coefficients are given by Lemma \ref{minvert}. Hence this coefficient is a polynomial with integer coefficients in the internal structure functions $c^e_{xy}$, $x,y\in V_i$, $e\in E_i$, which are the structure constants of the Tanaka symbol in the basis given by the frame; in particular, it is independent of $v$ and $s$. Denote it, with the relative determinant
sign included, by $Q_w$. These two contributions give
\eqref{path_coeff_formula}.
\end{proof}

\begin{rk}\label{path_coeff_rem}
The second term in \eqref{path_coeff_formula} need not vanish.
For example, consider a component with characteristic tuple
$(5,3,1)$ and vertices $\rho_0$, $\rho_1$, $a$ of depth $1$,
$\rho_2$, $c$ of depth $2$, and $y$ of depth $3$.
For $w=c$ and $v=a$, the coefficient in the lemma is
\[
\pm(\alpha_i^2-\alpha_t^2)
\bigl(c^s_{ac}+c^{\overrightarrow{\rho_2y}}_{cy}
c^s_{a\rho_2}\bigr).
\]
Quasi-normality does not force
$c^{\overrightarrow{\rho_2y}}_{cy}$ to vanish, and this explains why Proposition \ref{c2} first eliminates the structure functions having a root as one of their inputs. For instance, let $\mathfrak m_{-1}$ have the orthonormal basis $\rho_0,\rho_1,a,\rho_2,c,y$, let $\mathfrak m_{-2}$ have the basis formed by $e_{0x}$, $x\in\{\rho_1,a,\rho_2,c,y\}$, by $e_{1z}$, $z\in\{\rho_2,c,y\}$, and by $e_{2y}$, and set
\[
[\rho_0,x]=e_{0x},\qquad [\rho_1,z]=e_{1z},\qquad [\rho_2,y]=e_{2y},\qquad [c,y]=\kappa\,e_{2y},
\]
where $\kappa\in\mathbb R$, all other brackets of basis vectors being zero. For every $\kappa$, the characteristic tuple of this algebra is $(5,3,1)$, the given basis is a basis as in Theorem \ref{ccreduce} with the tuple of roots $(\rho_0,\rho_1,\rho_2)$, and $c^{\overrightarrow{\rho_2y}}_{cy}=\kappa$.
\end{rk}

\begin{Prop}
\label{cofterm2}
Let $\Gamma_i$ be nonabelian, let $v,w\in V_i$ with
\[
0<\mathrm{depth}(v)\leq\mathrm{depth}(w)=d,
\]
and let $l\in V_t$ with $t\neq i$. Assume that \eqref{crt_depth0} holds. Then the coefficient of the monomial
\begin{equation}\label{vl2_monom}
u_vu_lP_w(U)p(U)
\end{equation}
in $\det W^{d+1}_{\rt 0}$ is equal, up to a sign, to
\begin{equation}\label{cofterm2_coeff}
\begin{aligned}
&(\alpha_i^2-\alpha_t^2)
\begin{cases}
c^l_{vw}+Q_w c^l_{v\rho_d^i},
  &d\le\omega_i\text{ and }w\ne\rho_d^i,\\
c^l_{vw},&\text{otherwise}
\end{cases}\\
&\qquad\pmod{\mathrm{Ideal}\bigl(c^{\rho_0^t}_{xy}:x,y\in V_i\bigr)},
\end{aligned}
\end{equation}
where $Q_w$ is the polynomial in the first case of Lemma \ref{augmented_path_rigidity}. The ideal term is absent, i.e., the formula is exact, if $l=\rho_0^t$ or if the component $\Gamma_t$ is abelian.
\end{Prop}
\begin{proof}
Expand $\det W^{d+1}_{\rt 0}$ along the last column. By Remark \ref{block_diag_rem}, the matrix $\mathfrak A$ is block diagonal, and the rows of $W^{d+1}_{\rt 0}$ belonging to the component $\Gamma_i$ (the core rows of this component and the augmenting row) outnumber by one the columns indexed by $E_i$. Hence the cofactor of an entry of the last column lying in a row of a component $\Gamma_j$ with $j\neq i$ vanishes, because in the corresponding minor the $d_i+1$ rows of the component $\Gamma_i$ have nonzero entries only in the $d_i$ columns indexed by $E_i$. Therefore only the rows of the component $\Gamma_i$ contribute, and the cofactor of the entry of the last column in such a row is, up to a sign, the product of the determinant of the block of $\Gamma_i$ with this row removed and of the core determinants $\det T_j$, $j\neq i$.

Let us compare degrees. Fix a row $(p,z)$ of the component $\Gamma_i$ (with $(p,z)=(d+1,\rt 0)$ for the augmenting row) and a monomial term of $\tilde{\mathfrak b}^p_z$ which, multiplied by a monomial of the cofactor, gives \eqref{vl2_monom}. This term is a product of two vertex variables and $p-1$ edge variables. For $j\neq i$, the $\Gamma_j$-edge degree of \eqref{vl2_monom} is equal to that of $p_j(U)$, which, by Remark \ref{degree_rk}, is the $\Gamma_j$-edge degree of $\det T_j$; hence the term contains no $\Gamma_j$-edge variables, so all its $p-1$ edge variables are $\Gamma_i$-edge variables. Comparing the $\Gamma_j$-vertex degrees, we see that for $j\neq i,t$ the factor $\det T_j$ supplies a monomial of $\Gamma_j$-vertex degree $d_j$, hence exactly the monomial $p_j(U)$, whose coefficient is $\pm1$ by Lemma \ref{pcoeff1}, and that the term of the last column contains exactly one $\Gamma_i$-vertex variable and one $\Gamma_t$-vertex variable $u_s$. If $\Gamma_t$ is nonabelian, then each row of $T_t$ supplies one $\Gamma_t$-vertex variable, so that $\det T_t$ supplies a monomial of $\Gamma_t$-vertex degree $d_t$ dividing the $\Gamma_t$-vertex part $u_lu_{\rho_0^t}^{d_t}$ of \eqref{vl2_monom}; hence $s\in\{l,\rho_0^t\}$. If $\Gamma_t$ is abelian, then $p_t(U)=1$, and $s=l$.

Consequently, grouping the terms of the expansion according to the external vertex variable $u_s$ of the selected term of the last column, we obtain that the coefficient of \eqref{vl2_monom} in $\det W^{d+1}_{\rt 0}$ is equal to
\begin{equation}\label{cofterm2_decomp}
\pm\sum_{s}C_s\,D_s,
\end{equation}
where $s$ ranges over $\{l,\rho_0^t\}$ if $\Gamma_t$ is nonabelian and over $\{l\}$ otherwise, $C_s$ is the coefficient of $u_vu_sP_w(U)p_i(U)$ in $\det N_s$, and $D_s$ is the coefficient of $u_lp_t(U)/u_s$ in $\det T_t$ (with $D_l=\pm1$ by Lemma \ref{pcoeff1}, also when $\Gamma_t$ is abelian). Indeed, for a fixed $s$, the sum over the rows $(p,z)$ of the component $\Gamma_i$ and over the terms of $\tilde{\mathfrak b}^p_z$ containing $u_s$ of the products of these terms with the corresponding cofactors of the block of $\Gamma_i$ is exactly the expansion of $\det N_s$ along its last column, and the signs of the permutations of the block structure factorize accordingly.

By Lemma \ref{augmented_path_rigidity}, applied with $s=l$ (its assumption \eqref{crt_depth0_s} is \eqref{crt_depth0}), $C_l$ is given, up to a sign, by \eqref{path_coeff_formula} with $s=l$, and $D_l=\pm1$. This proves the proposition if $\Gamma_t$ is abelian or $l=\rho_0^t$. Otherwise, the term with $s=\rho_0^t$ in \eqref{cofterm2_decomp} is a multiple of $C_{\rho_0^t}$, and every coefficient of the last column of $N_{\rho_0^t}$ belongs to the ideal generated by the structure functions $c^{\rho_0^t}_{xz}$, $x,z\in V_i$, by Lemma \ref{mixed_b_chain}. Hence $C_{\rho_0^t}$ belongs to this ideal, which completes the proof.
\end{proof}

\begin{Prop}[Abelian source components]
\label{c1bis}
If $E_i=\varnothing$, i.e. the source component is abelian, and the principal augmented minors, i.e. the maximal augmented minors of the principal part of the fundamental algebraic system, vanish, then
\begin{equation}\label{cw_depth0_c1bis}
c^l_{vw}=0\qquad(v,w\in V_i,\ l\in V_t,\ t\ne i).
\end{equation}
\end{Prop}
\begin{proof}
Since the source component has no edges, all rows of the $i$th block of the principal coefficient matrix are zero. Therefore, the rational consistency of the principal system, which follows from the vanishing of the principal augmented minors, implies that the corresponding entries of the last column vanish; in particular, $\tilde{\mathfrak b}^1_w=0$ for $w\in V_i$. On the other hand, the coefficient of $u_vu_l$ in \eqref{tb1} is equal to $(\alpha_i^2-\alpha_t^2)c^l_{vw}$, because the term with upper index $v\in V_i$ does not contribute, the corresponding difference of eigenvalues being equal to zero. Since the eigenvalues corresponding to distinct components are distinct, this proves the assertion.
\end{proof}

\begin{Prop}
\label{c2}
Assume that all maximal augmented minors of the fundamental algebraic system vanish. Then
\begin{equation}\label{cw_depth0_c2}
c^l_{vw}=0,\qquad v,w\in V_i,\quad l\in V_t,\quad t\ne i.
\end{equation}
In particular, $D_i^2\cap D=D_i$ for every $i$.
\end{Prop}
\begin{proof}
By Corollary \ref{frak_FAL}, the vanishing of the maximal augmented minors of the fundamental algebraic system implies the vanishing of the principal augmented minors. If the source component $\Gamma_i$ is abelian, the required relations follow from Proposition \ref{c1bis}. So let $\Gamma_i$ be nonabelian. Since $\det W^1_{\rt 0}=0$, Proposition \ref{c0} gives relation \eqref{crt_depth0}, i.e. $c^l_{\rt 0v}=0$ for every $v\in V_i$ of depth $0$ and every $l\in V_t$, $t\neq i$ (for $v$ of positive depth this relation holds by \eqref{struct_quasinorm}). In particular, the hypotheses of Propositions \ref{c1} and \ref{cofterm2} are satisfied.

Now fix $t\neq i$. The relations \eqref{cw_depth0_c2} are obtained in two stages. In the first stage, which is needed only if $\Gamma_t$ is nonabelian, we prove them for $l=\rho_0^t$; in the second stage we prove them for an arbitrary $l\in V_t$, using the relations of the first stage, which annihilate all the ideal terms appearing in Propositions \ref{c1} and \ref{cofterm2}. In both stages we argue in the same way, as follows.

First, Proposition \ref{c1} gives $c^l_{wv}=0$ whenever $w\in V_i$ has depth $0$ and $v\in V_i$ is arbitrary. It remains to consider pairs of vertices of positive depth. Fix $d\in[1:\omega_i+1]$; by skew-symmetry, it suffices to prove that $c^l_{vw}=0$ whenever $0<\mathrm{depth}(v)\leq\mathrm{depth}(w)=d$. If $d\leq\omega_i$, we first take $w=\rt d$. By Proposition \ref{cofterm2}, the coefficient of the monomial $u_vu_lP_{\rt d}(U)p(U)$ in $\det W^{d+1}_{\rt 0}$ is equal to $\pm(\alpha_i^2-\alpha_t^2)c^l_{v\rt d}$. Since this minor vanishes and the eigenvalues corresponding to distinct components are distinct, we get
\begin{equation}\label{root_lower_eliminated}
c^l_{v\rt d}=0\qquad\text{whenever } 0<\mathrm{depth}(v)\leq d.
\end{equation}
Now let $w$ be an arbitrary vertex of depth $d$. By Proposition \ref{cofterm2} and by \eqref{root_lower_eliminated} (which is needed only when $d\le\omega_i$ and $w\ne\rt d$), the coefficient of the monomial $u_vu_lP_w(U)p(U)$ in $\det W^{d+1}_{\rt 0}$ is equal to $\pm(\alpha_i^2-\alpha_t^2)c^l_{vw}$. Since this minor vanishes, we obtain $c^l_{vw}=0$, which proves \eqref{cw_depth0_c2}.

To prove the last statement, recall that the classes of the edge fields of the corresponding component span the $i$th central summand of the symbol, and that the brackets of the vertex fields joined by edges are the chosen edge fields themselves. Hence, \eqref{cw_depth0_c2} gives $D_i^2=\operatorname{span}\{X_s:s\in V_i\cup E_i\}$, and the last conclusion follows.
\end{proof}

\begin{Prop}[Three distinct horizontal components]
\label{c3}
Let $i,t,t'$ be pairwise distinct, and let $w\in V_i$, $v\in V_t$, $l\in V_{t'}$. If the principal system has a rational solution, then
\[
c^l_{vw}=c^v_{lw}=c^w_{lv}=0.
\]
\end{Prop}
\begin{proof}
Let us contract the first-layer equations of the $i$th block with $Z_i$, i.e. multiply the equation corresponding to each vertex of this block by the corresponding vertex variable and add up the results. In the operator form of the first layer, the left side of the resulting equation is the value of the solution at the bracket of this horizontal vector with itself; hence, it is zero, because $[Z_i,Z_i]=0$. Thus
\begin{equation}\label{block_first_contraction}
\sum_{w\in V_i}u_w\tilde{\mathfrak b}^1_w=0
\end{equation}
is a polynomial identity (note that its left-hand side does not involve the solution). Extracting from this identity the coefficient of the monomial $u_wu_vu_l$, we get
\[
(\alpha_i^2-\alpha_t^2)c^v_{lw}
 +(\alpha_i^2-\alpha_{t'}^2)c^l_{vw}=0.
\]
In the same way, applying the contraction separately to the blocks corresponding to $t$ and $t'$, we obtain two more relations of the same type. In order to write these three relations in a compact form, put
\[
a=\alpha_i^2,\quad b=\alpha_t^2,\quad c=\alpha_{t'}^2,
\qquad x=c^w_{vl},\quad y=c^v_{lw},\quad z=c^l_{wv}.
\]
Then, using the skew-symmetry of the brackets, we can rewrite the resulting three equations as the following homogeneous linear system:
\[
\begin{pmatrix}
0&a-b&-(a-c)\\
-(b-a)&0&b-c\\
c-a&-(c-b)&0
\end{pmatrix}
\begin{pmatrix}x\\y\\z\end{pmatrix}=0.
\]
The determinant of its matrix is equal to $-2(a-b)(a-c)(b-c)$, which is nonzero, since the eigenvalues corresponding to distinct components are pairwise distinct. Consequently, $x=y=z=0$, which proves the proposition. Note that the argument uses only the first layer of the principal system; therefore, it is valid for vertices of arbitrary depth, as well as for abelian components.
\end{proof}

\begin{Coro}
\label{b10}
    If the maximal augmented minors of the fundamental algebraic system vanish, then the entries of $\tilde{\mathfrak{b}}^1$ are identically $0$.
\end{Coro}
\begin{proof}
We use formula \eqref{tb1} and consider the monomials of an arbitrary entry of the first layer; by the row component we mean the component containing the vertex which indexes this entry. Each such monomial is a product of two vertex variables. When one of the two vertex factors belongs to the row component, the corresponding coefficient vanishes by Proposition \ref{c2}. When the components of the two factors and the row component are pairwise distinct, Proposition \ref{c3} applies, and the coefficient vanishes again. Finally, if both factors belong to the same external component, i.e. to the same component different from the row component, then their combined coefficient vanishes by item (3) of Proposition \ref{first divi}. Since there are no square terms in the first layer, these cases exhaust all its monomials, which completes the proof.
\end{proof}

The following remark records an immediate consequence of Corollary \ref{b10}.

\begin{rk}\label{bindp}
The entries of the second-layer column in block $i$ do not depend on the \emph{internal} edge variables, that is, on the edge variables indexed by $E_i$. Indeed, by Corollary \ref{b10}, the first layer of the last column vanishes, so that the first term of the recursive formula \eqref{tbs_frak} does not contribute to the second layer, and we get, for every $v\in V_i$,
\[
\tilde{\mathfrak b}^{2}_v=
\sum_{\substack{e\in E_i\\f\in E_t,\ t\ne i}}
(\alpha_i^2-\alpha_t^2)q_{ve}q_{ef}u_f.
\]
Thus only external edge variables occur in these entries. It is precisely this independence of the internal edge variables that will be used in Lemma \ref{beta2} below.
\end{rk}

\section{Vanishing of other cross-component structure functions}
\label{other_sec}

In this section, we prove item (4) of Proposition \ref{outline_transverse_vanishing}, i.e., the vanishing of the cross-component structure functions with a horizontal input and an edge input from the same component and with an output from another component (see Definition \ref{cross_component_def}). For this purpose, we use the operator form of the fundamental algebraic system given in Appendix \ref{opp:fas}, more precisely, the operator form of its first two layers. We first note that Proposition \ref{c2} gives a smooth identification
\begin{equation}\label{mi_-2W}
W_i=\mathfrak m^i_{-2}\cong D_i^2/D_i,
\end{equation}
and that, under this identification, the Levi bracket $\Lambda^2D_i\to W_i$ is onto. We will assume that this identification is made in what follows. Moreover, it follows immediately from Corollary \ref{b10} that the first layer of the last column of the principal part is identically zero. Our next goal is to show that the same is true for the second layer. Note that the solvability of the linear systems under consideration is derived from the vanishing of their minors, which are polynomial in the fiber variables. Consequently, the solvability obtained in this way is generic: the systems admit rational solutions defined on a nonempty Zariski open set. This is exactly the setting of the following elementary lemma.
\begin{Lemma}\label{beta2}
Fix a base point and a component $i$, and let $\beta$ be a polynomial $D_i^*$-valued function of the fiber variables which is independent of the functional $\eta_i\in W_i^*$ from \eqref{eta_i}, i.e., of the internal edge variables $u_e$, $e\in E_i$. If, on a nonempty Zariski-open set, the equations
\begin{equation}\label{2layer_sys}
(\ad Z_i)^*\xi_i=0,\qquad
(\ad(J^i_{\eta_i}Z_i))^*\xi_i=\beta
\end{equation}
have a rational solution $\xi_i\in W_i^*$, then $\beta=0$ identically. The same conclusion holds when external variables are adjoined as parameters.
\end{Lemma}
\begin{proof}
If $W_i=0$, then both left-hand sides of the system are zero, so that the second equation reduces to the required vanishing. Otherwise, take a generic $Z_i$; recall that such an element is ad-maximal. Put $K=\ker\ad Z_i$, $I=\operatorname{im}\ad Z_i$. Since the kernel of the adjoint of a linear map is the annihilator of its image, the first equation says that $\xi_i|_I=0$. Further, Lemma \ref{a1} gives $[K,D_i]\subset I$. Consequently, the second equation yields $\beta|_K=0$ wherever the rational solution is defined.

Let us show that the same holds on the orthogonal complement of the kernel. Contracting the second equation with $J^i_{\eta_i}Z_i$, we get
\[
\beta(J^i_{\eta_i}Z_i)=\xi_i([J^i_{\eta_i}Z_i,J^i_{\eta_i}Z_i])=0.
\]
Since the left-hand side is polynomial and the identity holds on a nonempty Zariski open set, this is a polynomial identity; hence it holds for all $\eta_i$, including any values where the solution itself is undefined. As $\beta$ is independent of $\eta_i$, it therefore annihilates the whole image of the map $\eta_i\mapsto J^i_{\eta_i}Z_i$, which, by duality, is $K^{\perp_{g_1}}$. Combining this with the vanishing on the kernel obtained above, we conclude that $\beta=0$ for generic $Z_i$ and external parameters, and, by polynomiality, this holds everywhere. Note that in this argument the ad-maximality is required only for generic $Z_i$.
\end{proof}

Let us apply this lemma to the principal part of the fundamental algebraic system. By Corollary \ref{frak_FAL} and the finite-row argument following Lemma \ref{pcoeff1}, the principal system is rationally consistent, that is, it admits a rational solution. Since the first layer of its last column vanishes, the first two layers of each of its blocks, written in the operator form, constitute a system of the type considered in the lemma. Therefore, for every component, we can apply the lemma with $\beta=\beta_i^2$; the required independence of $\eta_i$, i.e., of the internal edge variables, is precisely the content of Remark \ref{bindp}. Hence $\tilde{\mathfrak b}^2=0$.

We are now in a position to prove the vanishing of the cross-component structure functions with a horizontal input and an edge input from the same component and a central output from another component.
\begin{Lemma}
    \label{bianchi}
    If $y \in V_i$, $l \in E_i$ and $w \in E_t$ with $i \neq t$, then $c^w_{yl}=0$.
\end{Lemma}
\begin{proof}
    First, given $x, z, v \in V_i$ and $w\in E_t$ with $t\neq i$, we have
   \begin{equation}
   \label{jacobi1}
    F_w(x,z,v):=\varepsilon^w\Bigl(\bigl[X_x,[ X_z, X_v]\bigr]\Bigr)=\sum_{l\in V\cup E}c_{xl}^wc_{zv}^l
   \end{equation}
   (note that \eqref{jacobi1} does not involve directional derivatives of the structure functions, because, by the decomposition of the Tanaka symbol, $c^w_{zv}=0$ identically for $z,v\in V_i$ and $w\in E_t$, $t\ne i$). Observe that:
   \begin{enumerate}
   \item if $l\in V_i$, then $c_{xl}^w=0$ by the decomposition of the Tanaka symbol \eqref{eigen_decomp};
   \item if $l\in E_s$ with $s\neq i$, then $c_{zv}^l=0$, again by the decomposition of the Tanaka symbol \eqref{eigen_decomp};
   \item if $l\in V_s$ with $s\neq i$, then $c_{zv}^l=0$ by Propositions \ref{c0}, \ref{c1}, and \ref{c2}.
   \end{enumerate}
   Therefore, $F_w(x,z,v)$, defined in \eqref{jacobi1}, satisfies
    \begin{equation}
    \label{jacobi2}
    F_w(x,z,v)=\sum\limits_{l \in E_i}c^w_{xl}c^l_{zv}.
    \end{equation}
    Using the definition of $F_w(x,z,v)$ given in \eqref{jacobi1} together with the Jacobi identity, we obtain
    \begin{equation}
        \label{fw}
        F_w(x,z,v)+F_w(z,v,x)+F_w(v,x,z)=0
    \end{equation}
    for every $x,z,v \in V_i$.
    Note that $F_w$ is skew-symmetric in its last two arguments, by the skew-symmetry of the Lie bracket. We will show that $F_w$ is in fact totally skew-symmetric.

    Indeed, as we have shown above with the help of Lemma \ref{beta2}, $\beta^2=0$, and hence $\tilde{\mathfrak{b}}^2_v=0$ for all $v \in V$. Since $\tilde{\mathfrak{b}}^1_v=0$ as well, by Corollary \ref{b10}, it follows from equation \eqref{tbs} that
\begin{align}
&\sum_{l\in E}
\sum_{\substack{w\in E\\
                \mathcal{C}(w)\neq\mathcal{C}(l)}}
\bigl(\alpha(l)^2-\alpha(w)^2\bigr)
q_{vl}q_{lw}u_w
\\
&=
\sum_{l\in E}
\sum_{\substack{w\in E\\
                \mathcal{C}(w)\neq\mathcal{C}(l)}}
\sum_{x,y\in V}
\bigl(\alpha(l)^2-\alpha(w)^2\bigr)
c^l_{xv}c^w_{yl}\,u_xu_yu_w =0.
\end{align}
Consequently, comparing the coefficients of $u_xu_yu_w$ yields
\begin{equation}
\label{jacobi_b2=0}
\sum_{\substack{l\in E\\
                \mathcal{C}(l)\neq\mathcal{C}(w)}}
\bigl(\alpha(l)^2-\alpha(w)^2\bigr)
\left(
c^l_{xv}c^w_{yl}
+
c^l_{yv}c^w_{xl}
\right)
=0, \quad \forall x,y\in V, \,w\in E.
\end{equation}
Using \eqref{jacobi1} and item (2) above, we obtain that if $x, y\in V_i$ and $w\in E_t$ with $t\neq i$, then
\begin{equation}
\label{jacobi3}
\bigl(\alpha_i^2-\alpha_t^2\bigr)
\Bigl(
F_w(y,x,v)+F_w(x,y,v)
\Bigr)=0.
\end{equation}
Here, by item (2) above, the summation in \eqref{jacobi_b2=0} can be restricted to $l \in E_i$, so that $\alpha(l) = \alpha_i$ and $\alpha(w) = \alpha_t$.
Thus $F_w$ is skew-symmetric in its first two arguments as well, and is therefore totally skew-symmetric. Consequently, the left-hand side of \eqref{fw} equals $3F_w(x,y,v)$, and \eqref{fw} implies that
\begin{equation}
\label{f0}
    F_w(x,y,v)=0, \quad \forall x,y,v \in V_i,\, w \in E_t \text{ with } t \neq i.
\end{equation}

Now, for fixed \(x\in V_i\) and \(w\in E_t\), where
\(t\neq i\), define
\begin{equation}
\label{etawx}
\eta_x^w\in W_i^*,
\qquad
\eta_x^w(X_\ell)=c^w_{x\ell},
\quad \ell\in E_i.
\end{equation}
Then, for \(y,v\in V_i\), we have
\begin{align*}
&g_1(J^i_{\eta^w_x}X_y,X_v)
\stackrel{\eqref{J_eta}}{=}
\eta_x^w\bigl([X_y,X_v]\bigr)=
\eta_x^w\left(
    \sum_{\ell\in E_i}
    c^\ell_{yv}X_\ell
\right)
=\\
&\sum_{\ell\in E_i}
c^\ell_{yv}\,\eta_x^w(X_\ell)\stackrel{\eqref{etawx}}{=}
\sum_{\ell\in E_i}
c^\ell_{yv}c^w_{x\ell}
=F_w(x,y,v).
\end{align*}
In other words, $F_w(x,y,v)=g_1(J^i_{\eta^w_x}X_y,X_v)$, which vanishes by \eqref{f0}. Since $y,v \in V_i$ were arbitrary, we conclude that $J^i_{\eta^w_x}=0$.
Moreover, since the Levi bracket map $\bigwedge^2 D_i \to D_i^2/D_i$, $(X, Y) \mapsto [X, Y] \bmod D_i$, is surjective onto $W_i$ under the identification \eqref{mi_-2W}, we conclude that $\eta^w_x = 0$ as well.
By \eqref{etawx}, this yields that $c^w_{yl}=0$ for all $y \in V_i$ and $l \in E_i$, as required.
\end{proof}

To show that the whole principal part of the last column vanishes, we also need to treat the mixed horizontal index in $q_{ef}$, that is, the terms of its defining sum whose horizontal index belongs to a component different from that of the first edge index. Let $e\in E_i$ and $x\in V_k$, $k\ne i$, and write $X_e=[X_a,X_b]$ with $a,b\in V_i$. The splitting of the Tanaka symbol into pairwise commuting ideals, together with Proposition \ref{c3}, gives $[D_k,D_i]\subset D_k\oplus D_i$. Then, by the Jacobi identity,
\[
[X_x,X_e]=[[X_x,X_a],X_b]+[X_a,[X_x,X_b]]\in D_k+D_i^2.
\]
Expanding this bracket in the frame, we see, in particular, that $c^f_{xe}=0$ for every edge $f\notin E_i$. For $x\in V_i$, the same conclusion is precisely the statement of Lemma \ref{bianchi}. Combining both cases, we get
\[
q_{ef}=\sum_{x\in V}c^f_{xe}u_x=0\qquad(e\in E_i,\ f\in E_t,\ t\ne i).
\]
Since, moreover, $\tilde{\mathfrak b}^1=0$, the recurrence \eqref{tbs_frak} implies, by induction on the layer, the following corollary.
\begin{Coro}
\label{btfrak0}
    The principal part $\tilde{\mathfrak{b}}$ of $\tilde{b}$ is identically $0$. Consequently, the minimal vertex degree of the entries of $\tilde{b}$ is at least $3$.
\end{Coro}

The vanishing of the principal part allows us to pass to the cubic vertex-homogeneous part of the last column and to prove the analogue of the previous lemma for horizontal outputs.
\begin{Prop}
\label{mixed_edge_horizontal_vanishing}
    If $y \in V_i$, $l \in E_i$, and $w \in V_t$ with $i \neq t$, then $c^w_{yl}=0$.
\end{Prop}
\begin{proof}
Fix \(i\neq t\) and \(w\in V_t\). Throughout the proof we set
\[
d:=|E|=n-m,
\qquad
\widehat{\mathfrak b}:=\tilde b_{(3)}.
\]
Here the second quantity is the vertex homogeneous component of degree three of the last column. We first show that this cubic part is the right-hand side of a consistent linear system. By Corollary \ref{btfrak0}, the quadratic vertex-homogeneous part of
\(\tilde b\) vanishes. Consequently, for every
\((d+1)\times(d+1)\) minor of the augmented matrix
\([A\mid\tilde b]\), its component of vertex degree \(d+3\) is
precisely the corresponding maximal minor of
\[
[\mathfrak A\mid\widehat{\mathfrak b}],
\qquad
\mathfrak A=A_{(1)}.
\]
Indeed, the entries of the coefficient columns have vertex degree at least one, while, by the same corollary, the entries of the last column have vertex degree at least three. Hence, in the Leibniz expansion of such a minor, the component of the indicated vertex degree comes only from the products in which the last column contributes its degree \(3\) component, while each of the other \(d\) columns contributes its
degree \(1\) component. Since the maximal minors of
\([A\mid\tilde b]\) vanish, all maximal minors of
\([\mathfrak A\mid\widehat{\mathfrak b}]\) vanish as well. On the
other hand, Lemma \ref{pcoeff1} gives \(\det T\not\equiv0\), so that
\(\mathfrak A\) has full column rank over the field of rational
functions in the fiber variables. Therefore, by the same finite-row argument as for the principal system, the cubic associated
system
\begin{equation}
\label{cubic_associated_system}
\mathfrak A\,\Xi=\widehat{\mathfrak b}
\end{equation}
has a rational solution.

Since, by Corollary \ref{block_diag}, the coefficient matrix of this system is block diagonal, the rational solution restricts to each block. We will apply Lemma \ref{beta2} to the block of the component under consideration, regarding all variables belonging to the other components as parameters. For this purpose, let us first compute the first two layers of the cubic right-hand side. Since the first layer of the last column is quadratic in the vertex variables, the first layer of the cubic right-hand side is zero:
\[
\widehat{\mathfrak b}^{\,1}=0.
\]
For the second layer, we take \(p=1\) in
\eqref{term.of.b.rec.form}, use the column operation \eqref{eq:1},
and extract the vertex-degree-\(3\) component. Since
\(\tilde b^{\,2}\) is homogeneous of total \(U\)-degree \(3\), every
term in
\[
\hat{\mathfrak b}^{\,2}
=
\bigl(\tilde b^{\,2}\bigr)_{(3)}
\]
has edge degree \(0\). In particular,
\(\hat{\mathfrak b}^{\,2}\) is independent of the edge variables. Taking into account that \(a^1_{v\ell}=q_{v\ell}\), we obtain
\begin{align}
\label{cubic_second_layer}
\hat{\mathfrak b}^{\,2}_v
={}&
\sum_{\ell\in E}
\sum_{s\in V}
\bigl(\alpha(\ell)^2-\alpha(s)^2\bigr)
q_{v\ell}q_{\ell s}u_s
\nonumber\\
&-\frac12
\sum_{\ell\in E}
\sum_{s\in V}
q_{v\ell}X_\ell\bigl(\alpha(s)^2\bigr)u_s^2.
\end{align}

Next, let \(\xi_i\) denote the \(i\)-th block of a rational solution of
\eqref{cubic_associated_system}. Then
\[
\xi_i\in W_i^*
\]
and
\[
\mathfrak A_i\xi_i=\widehat{\mathfrak b}_i
\]
on the nonempty Zariski-open subset of the fiber on which the chosen
rational solution is defined. By
Lemma \ref{A^_1 and 2_i_lemma}, the equations of the first two layers of this block read
\[
\bigl(\operatorname{ad}Z_i\bigr)^*\xi_i=0,
\qquad
\Bigl(\operatorname{ad}(J^i_{\eta_i}Z_i)\Bigr)^*\xi_i
=
\widehat\beta_i^2,
\]
where
\[
\widehat\beta_i^2
=
\sum_{v\in V_i}
\widehat{\mathfrak b}_v^2\,\theta^v.
\]
As we have already observed, since \(\tilde b^2\) is homogeneous of total \(U\)-degree \(3\) and
\[
\widehat{\mathfrak b}^{\,2}
=
\bigl(\tilde b^2\bigr)_{(3)}
\]
has vertex degree \(3\), it has edge degree \(0\). Therefore,
\(\widehat\beta_i^2\) is independent of \(\eta_i\). Applying Lemma \ref{beta2}, with the external variables as parameters, we get
\[
\widehat\beta_i^2=0
\]
on the set where the solution is defined. Since \(\widehat\beta_i^2\) is polynomial in the fiber
variables, it follows that
\[
\widehat\beta_i^2\equiv0.
\]

Now fix \(v\in V_i\) and, as above, \(w\in V_t\), where
\(t\neq i\). In the identity
\(\widehat\beta_i^2=0\), we collect the terms which are linear in \(u_w\),
quadratic in the \(\Gamma_i\)-vertex variables, and independent of
the remaining external variables. Let us find out which terms of the second layer contribute to this part.
First, by Corollary \ref{block_diag}, the factor \(q_{v\ell}\) vanishes
unless \(\ell\in E_i\). Therefore, in the first sum of
\eqref{cubic_second_layer}, the desired terms can arise
from \(s=w\). If instead \(s\in V_i\) and the factor \(u_w\)
comes from \(q_{\ell s}\), then
\[
\alpha(\ell)^2-\alpha(s)^2
=
\alpha_i^2-\alpha_i^2
=
0,
\]
so that such terms vanish. As for the second sum in
\eqref{cubic_second_layer}, it contains the square \(u_s^2\) and hence
has no term which is linear in \(u_w\).
Consequently, we obtain
\[
\bigl(\alpha_i^2-\alpha_t^2\bigr)
\sum_{\ell\in E_i}
\sum_{x,y\in V_i}
c^\ell_{xv}c^w_{y\ell}\,u_xu_y
=
0.
\]
Since \(\alpha_i^2\neq\alpha_t^2\), this gives the quadratic
polynomial identity
\[
\sum_{\ell\in E_i}
\sum_{x,y\in V_i}
c^\ell_{xv}c^w_{y\ell}\,u_xu_y
=
0.
\]
Polarizing this identity, we get
\begin{equation}
\label{cubic_skew}
\sum_{\ell\in E_i}
\left(
c^\ell_{xv}c^w_{y\ell}
+
c^\ell_{yv}c^w_{x\ell}
\right)
=0
\end{equation}
for all \(x,y,v\in V_i\).

As in the proof of the previous lemma, where the output was central, we define
\begin{equation}
\label{Gw_vertex}
F_w(x,y,v)
:=
\sum_{\ell\in E_i}
c^\ell_{yv}c^w_{x\ell}.
\end{equation}
Equation \eqref{cubic_skew} says precisely that \(F_w\) is skew-symmetric in
its first two arguments. Moreover, it is skew-symmetric in its last two
arguments, because
\[
c^\ell_{yv}=-c^\ell_{vy}.
\]
Hence \(F_w\) is an alternating three-form on \(D_i\).

It remains to use the Jacobi identity. Namely, for
\(x,y,v\in V_i\), we take the \(X_w\)-component of
\[
[X_x,[X_y,X_v]]
+
[X_y,[X_v,X_x]]
+
[X_v,[X_x,X_y]]
=0.
\]
The \(X_w\)-component of the first summand is equal to
\[
X_x(c^w_{yv})
+
\sum_{r\in V\cup E}
c^r_{yv}c^w_{xr}.
\]
By Propositions \ref{c0}, \ref{c1}, and \ref{c2},
\(c^w_{yv}=0\) whenever \(y,v\in V_i\) and
\(w\in V_t\), \(t\neq i\), so the derivative term vanishes.
The same propositions eliminate the terms with an intermediate
vertex index outside \(V_i\), as well as, through the factor
\(c^w_{xr}\), the terms with an intermediate vertex index in \(V_i\). Finally, the decomposition \eqref{eigen_decomp}
eliminates the terms with \(r\in E_s\), \(s\neq i\).
Thus only the terms with \(r=\ell\in E_i\) remain, and, treating the other two summands in the same way, we see that the Jacobi identity
becomes
\begin{equation}
\label{Gw_vertex_cyclic}
F_w(x,y,v)+F_w(y,v,x)+F_w(v,x,y)=0.
\end{equation}
Since \(F_w\) is alternating and cyclic permutations are even, the three terms in
\eqref{Gw_vertex_cyclic} are equal. Therefore,
\[
3F_w(x,y,v)=0,
\]
and hence
\[
F_w=0.
\]

Let us now deduce from this the vanishing of the required structure functions. For fixed \(x\in V_i\), define
\[
\eta_x^w\in W_i^*,
\qquad
\eta_x^w(X_\ell):=c^w_{x\ell},
\quad
\ell\in E_i.
\]
Further, denote by
\[
L_i:\bigwedge^2D_i\longrightarrow W_i
\]
the Levi bracket. Then \eqref{Gw_vertex} gives
\[
\eta_x^w\bigl(L_i(X_y,X_v)\bigr)
=
F_w(x,y,v)
=
0
\]
for all \(y,v\in V_i\). Since the Levi bracket is surjective onto
\(W_i\) under the identification \eqref{mi_-2W}, we conclude that
\(\eta_x^w=0\). In other words,
\[
c^w_{x\ell}=0
\qquad
\text{for all }
x\in V_i,\quad
\ell\in E_i,
\]
which proves the proposition.
\end{proof}

The vanishing results obtained so far can be summarized in terms of the brackets of the vertex fields with the squares of the components, as follows.
\begin{Prop}\label{ady}
For every $i,t\in[1:k]$, the following inclusions hold:
\begin{enumerate}
    \item $\operatorname{ad}_{X_y}(D_i^2)\subset D_i^2+D_t^2$ whenever
    $y\in V_t$ and $t\neq i$;
    \item $\operatorname{ad}_{X_y}(D_i^2)\subset D_i^2$ whenever
    $y\in V_i$.
\end{enumerate}
\end{Prop}

\begin{proof}
First note that
Propositions \ref{c1bis} and \ref{c2} imply
\[
D_i^2\cap D=D_i.
\]
Moreover, if \(e\in E_i\), then \(e=\overrightarrow{vw}\) for
some \(v,w\in V_i\), and the quasi-normality relation
\eqref{quasinorm_basis} gives
\[
X_e=[X_v,X_w].
\]
In particular, \(X_e\) is a section of \(D_i^2\). Since the classes
of the \(X_e\), \(e\in E_i\), form a basis of
\(\mathfrak m^i_{-2}=D_i^2/D_i\), it follows that
\[
\{X_v\}_{v\in V_i}
\cup
\{X_e\}_{e\in E_i}
\]
is a local frame of \(D_i^2\). Thus, in both parts of the
proposition, it is enough to compute the brackets of \(X_y\) with
these frame fields.

We first establish the following mixed-bracket inclusion. Let
\(i\neq t\), \(v\in V_i\), and \(y\in V_t\).
Since \(\mathfrak m^i\) and \(\mathfrak m^t\) are distinct summands
of the direct-sum decomposition \eqref{eigen_decomp}, their mutual
bracket in the Tanaka symbol is zero. Hence
\[
[X_y,X_v]\bmod D=0,
\]
so that \([X_y,X_v]\) is a section of \(D\). On the other hand,
Proposition \ref{c3} shows that, for every
\(s\notin\{i,t\}\) and every \(w\in V_s\),
\[
c^w_{yv}=0.
\]
Consequently,
\begin{equation*}
[X_y,X_v]\in D_i\oplus D_t,
\qquad
v\in V_i,\quad y\in V_t,\quad i\neq t.
\end{equation*}
Since the corresponding vertex fields form local frames of
\(D_i\) and \(D_t\), the same statement holds for arbitrary local
sections, which gives
\begin{equation}
\label{mixed_Di_Dt_bracket}
[D_t,D_i]\subset D_i\oplus D_t,
\qquad i\neq t.
\end{equation}

To prove Item (1), fix \(y\in V_t\) with \(t\neq i\).
For \(v\in V_i\), the inclusion just obtained gives
\[
[X_y,X_v]\in D_i\oplus D_t
\subset D_i^2+D_t^2.
\]
It remains to consider \(e\in E_i\). Write
\(e=\overrightarrow{vw}\), where \(v,w\in V_i\). By
\eqref{quasinorm_basis}, we have
\[
X_e=[X_v,X_w].
\]
Therefore, the Jacobi identity gives
\[
[X_y,X_e]
=
[X_y,[X_v,X_w]]
=
[[X_y,X_v],X_w]+[X_v,[X_y,X_w]].
\]
Both inner brackets on the right-hand side are sections of
\(D_i\oplus D_t\) by \eqref{mixed_Di_Dt_bracket}. Since
\[
[D_i,D_i]\subset D_i^2
\]
by definition and
\[
[D_t,D_i]\subset D_i\oplus D_t
\]
by \eqref{mixed_Di_Dt_bracket}, each of the two terms on the
right-hand side belongs to
\[
D_i^2+D_t\subset D_i^2+D_t^2.
\]
Thus
\[
[X_y,X_e]\in D_i^2+D_t^2
\qquad
\text{for every }e\in E_i.
\]
Since the fields \(X_v\), \(v\in V_i\), and
\(X_e\), \(e\in E_i\), form a local frame of \(D_i^2\),
the Leibniz rule now yields
\[
\operatorname{ad}_{X_y}(D_i^2)\subset D_i^2+D_t^2,
\]
which proves Item (1).

We turn to Item (2). Let \(y\in V_i\). If
\(v\in V_i\), then
\[
[X_y,X_v]\in [D_i,D_i]\subset D_i^2.
\]
Now let \(e\in E_i\). Lemma \ref{bianchi} gives
\[
c^w_{ye}=0
\qquad
\text{for every }w\in E_t,\quad t\neq i,
\]
while the proposition immediately preceding the present one gives
\[
c^w_{ye}=0
\qquad
\text{for every }w\in V_t,\quad t\neq i.
\]
Hence, in the structure-function expansion
\[
[X_y,X_e]
=
\sum_{w\in V\cup E}
c^w_{ye}X_w,
\]
all terms whose upper index belongs to a component different from
\(\Gamma_i\) vanish. Therefore,
\[
[X_y,X_e]\in
\operatorname{span}
\{X_w:w\in V_i\cup E_i\}
=
D_i^2.
\]
Using again that these vertex and edge fields form a local frame of
\(D_i^2\), we obtain from the Leibniz rule that
\[
\operatorname{ad}_{X_y}(D_i^2)\subset D_i^2.
\]
This proves Item (2) and completes the proof.
\end{proof}

These inclusions imply the following involutivity statement.
\begin{Coro}\label{involutive}
Every partial sum $\bigoplus_{i\in I}D_i^2$ is involutive. In particular, each of the distributions $D_i^2$, as well as each pairwise sum of them, is involutive.
\end{Coro}
\begin{proof}
For brevity, we write $F_i=D_i^2$. By Proposition \ref{ady}, we have $[D_i,F_i]\subset F_i$ and $[D_j,F_i]\subset F_i+F_j$ for $i\ne j$. Recall that $F_i$ is spanned by $D_i$ together with the brackets of the frame fields of the latter. The brackets of the fields of the first type with sections of the squares are controlled by the above inclusions, so it suffices to consider the brackets of the form $[[X_a,X_b],Y]$. First, if the last field here is a section of the same square, then, expanding such a bracket by the Jacobi identity and using the first of the above inclusions twice, we get $[F_i,F_i]\subset F_i$. If, instead, $Y\in F_j$, then the same identity together with the mixed inclusions gives $[F_i,F_j]\subset F_i+F_j$. The Leibniz rule extends these assertions from frame fields to arbitrary sections. Consequently, the bracket of two sections of any partial sum stays in that sum, that is, every partial sum is involutive.
\end{proof}

\section{Completing the proof of Theorems \ref{mainthm} and \ref{mainthm_eisen} by rotating the frames on the \texorpdfstring{$D_i$}{Di}}
\label{rotation_sec}

We now complete the proof by adapting the frame-rotation argument of \cite[Section~4.7]{Zaifeng}.  The affine argument in that paper produces an orthonormal rotated frame.  In the projective case the same construction produces, in general, only a conformal copy of the metric on each eigenspace $D_i$.  The first-divisibility identities of Proposition \ref{first divi} determine the corresponding conformal factor and give precisely the coefficients in Definition \ref{Gen_Levi-Civita}.

The vanishing results of Sections \ref{transverse_vanishing_sec} and \ref{other_sec} were obtained in quasi-normal frames, and hence only near the points of the regular locus $\Omega$, that is, of the open dense subset of the stable neighborhood on which the characteristic tuple is locally constant. Their geometric content, however, consists of inclusions between the brackets of the eigenspace distributions $D_i$ and of their squares $D_i^2$, and each of these inclusions expresses the vanishing of a smooth tensor defined on the whole stable neighborhood. In Subsection \ref{tuple_continuation_subsec} we use this observation to extend these inclusions from $\Omega$ to the whole stable neighborhood by continuity. After that, in Subsection \ref{rotation_argument_subsec}, the rotation argument is carried out near an arbitrary stable point, with an arbitrary weakly adapted frame in place of a quasi-normal frame. Apart from these vanishing results, which are used only on $\Omega$, this section does not use quasi-normal frames; such frames need not exist near the points where the characteristic tuple jumps. The other results used below, namely Theorem \ref{Tanaka_decomp_prop} and Propositions \ref{eigenconst} and \ref{first divi}, hold near every stable point and do not require quasi-normal frames.

\subsection{Extension of the bracket relations across jumps of the characteristic tuple}
\label{tuple_continuation_subsec}

Throughout this section, $\mathcal U$ denotes a connected stable neighborhood on which $g_1$ and $g_2$ are projectively equivalent, and $\Omega\subset\mathcal U$ denotes the open dense subset on which the characteristic tuple of $D$ is locally constant. We write $m_i=\rank D_i$, and $d_i$ for the rank of the summand $W_i$ of the decomposition
\[
T\mathcal U/D=\bigoplus_{i=1}^kW_i,\qquad W_i=\operatorname{im}\bigl(\Lambda^2D_i\longrightarrow T\mathcal U/D\bigr),
\]
given by Theorem \ref{Tanaka_decomp_prop}, where the map is induced by the Levi bracket. By this theorem, every $d_i$ is constant on $\mathcal U$ and every $W_i$ is a smooth subbundle. Under the assumed projective equivalence, this decomposition uses only the step $2$ condition and the stability of the points of $\mathcal U$; the characteristic tuple plays no role in it.

\begin{Prop}[Extension of the bracket relations]
\label{tuple_continuation_prop}
The following statements hold on the whole neighborhood $\mathcal U$.
\begin{enumerate}
\item For every $i$, we have $D_i^2\cap D=D_i$. Consequently, $F_i:=D_i^2$ is a smooth subbundle of rank $m_i+d_i$, and
\begin{equation}
\label{rotation_tangent_decomp}
T\mathcal U=\bigoplus_{i=1}^kF_i.
\end{equation}
\item For all $i\neq t$,
\begin{equation}
\label{extension_mixed}
[D_t,D_i]\subset D_t\oplus D_i.
\end{equation}
\item For every $i$ and every $t\neq i$,
\begin{equation}
\label{extension_adjoint}
[D_i,F_i]\subset F_i,\qquad [D_t,F_i]\subset F_i\oplus F_t.
\end{equation}
\item For every nonempty subset $J\subset[1:k]$, the distribution
\begin{equation}
\label{rotation_delta}
\Delta_J:=\bigoplus_{i\in J}F_i
\end{equation}
is involutive.
\end{enumerate}
In particular, the inclusions of Proposition \ref{ady} and the involutivity statements of Corollary \ref{involutive} hold on the whole of $\mathcal U$.
\end{Prop}

\begin{proof}
Each of the statements (1)--(3) is equivalent to the vanishing of a smooth tensor on $\mathcal U$. Near every point of $\Omega$, quasi-normal frames are available, and the results of the preceding sections show that these tensors vanish on $\Omega$. Since $\Omega$ is dense, they vanish on the whole of $\mathcal U$ by continuity. Let us make this precise.

(1) The map $\Lambda^2D_i\to T\mathcal U/D$ induced by the Levi bracket has constant rank $d_i$, so its kernel $K_i$ is a smooth subbundle of $\Lambda^2D_i$. For sections $X_r,Y_r$ of $D_i$ such that $\sum_rX_r\wedge Y_r$ is a section of $K_i$, set
\[
\kappa_i\Bigl(\sum_rX_r\wedge Y_r\Bigr)=\sum_r[X_r,Y_r]\bmod D_i .
\]
By the definition of $K_i$, the right-hand side is a section of $D/D_i$. Since $[fX,Y]=f[X,Y]-(Yf)X$ and $X$ is a section of $D_i$, the map $\kappa_i\colon K_i\to D/D_i$ is $C^\infty(\mathcal U)$-linear, that is, it is a smooth tensor. Now fix $q\in\mathcal U$. Every element of $D_i^2(q)$ has the form $v=x+\sum_r[X_r,Y_r](q)$ with $x\in D_i(q)$. If $v\in D(q)$, then $\sum_rX_r\wedge Y_r(q)\in K_i(q)$, and $\kappa_i$ maps this element to the class of $v$ modulo $D_i(q)$; moreover, every element of $K_i(q)$ arises in this way. Therefore $D_i^2(q)\cap D(q)=D_i(q)$ if and only if $\kappa_i$ vanishes at $q$. On $\Omega$, this equality is given by Propositions \ref{c1bis} and \ref{c2}. Hence $\kappa_i=0$ on $\Omega$, and therefore on $\mathcal U$.

Consequently, $F_i\cap D=D_i$ on $\mathcal U$. Thus the map $F_i/D_i\to T\mathcal U/D$ is injective, its image is $W_i$, and $\dim F_i(q)=m_i+d_i$ at every point. Since $F_i$ is spanned near every point by the sections of $D_i$ and their brackets, it is a smooth subbundle. To prove \eqref{rotation_tangent_decomp}, let $Y_i\in F_i(q)$ satisfy $\sum_iY_i=0$. Reducing modulo $D$ and using the directness of the sum $\bigoplus_iW_i$, we get $Y_i\in F_i\cap D=D_i$ for every $i$, and then all $Y_i$ vanish because $D=\bigoplus_iD_i$. Thus the sum is direct. Moreover, since $D$ has step $2$,
\[
\sum_{i=1}^k\rank F_i=\sum_{i=1}^k(m_i+d_i)=\rank D+\rank\bigl(T\mathcal U/D\bigr)=\dim\mathcal U,
\]
which proves \eqref{rotation_tangent_decomp}.

(2) By Theorem \ref{Tanaka_decomp_prop}, the summands $\mathfrak m^t$ and $\mathfrak m^i$ commute, so that $[D_t,D_i]\subset D$. Hence the map
\[
D_t\otimes D_i\longrightarrow D/(D_t\oplus D_i),\qquad (Y,X)\longmapsto[Y,X]\bmod(D_t\oplus D_i),
\]
is well defined, and it is $C^\infty(\mathcal U)$-bilinear, because $[fY,X]=f[Y,X]-(Xf)Y$ and $[Y,fX]=f[Y,X]+(Yf)X$. On $\Omega$ this tensor vanishes by \eqref{mixed_Di_Dt_bracket}, which was derived from Proposition \ref{c3} in the proof of Proposition \ref{ady}. Hence it vanishes on $\mathcal U$, which proves \eqref{extension_mixed}.

(3) By (1), $F_i$ is a smooth subbundle, and the map
\[
D_i\otimes F_i\longrightarrow T\mathcal U/F_i,\qquad (X,V)\longmapsto[X,V]\bmod F_i,
\]
is $C^\infty(\mathcal U)$-bilinear by the same Leibniz rule, because $D_i\subset F_i$. On $\Omega$ it vanishes by Item (2) of Proposition \ref{ady}, hence everywhere, and this proves the first inclusion in \eqref{extension_adjoint}. For the second one, let $Y$ be a section of $D_t$, $t\neq i$. By \eqref{extension_mixed}, $[Y,X]$ is a section of $D_t\oplus D_i$ for every section $X$ of $D_i$. If $X$ and $X'$ are sections of $D_i$, then
\[
[Y,[X,X']]=[[Y,X],X']+[X,[Y,X']]
\]
by the Jacobi identity. Since $[Y,X]$ and $[Y,X']$ are sections of $D_t\oplus D_i$, another application of \eqref{extension_mixed}, together with $[D_i,D_i]\subset F_i$, shows that both terms on the right-hand side are sections of $F_t\oplus F_i$. As $F_i$ is spanned by the sections of $D_i$ and their brackets, the Leibniz rule gives $[D_t,F_i]\subset F_i\oplus F_t$.

(4) By \eqref{extension_adjoint}, the brackets of the sections of $D_i$ with the sections of $F_i$ and of $F_j$, $j\neq i$, are sections of $F_i$ and of $F_i\oplus F_j$, respectively. Since $F_i$ is spanned by the sections of $D_i$ and their brackets, it remains to consider the brackets $[[X,X'],V]$, where $X$ and $X'$ are sections of $D_i$ and $V$ is a section of $F_j$, with $j$ arbitrary. By the Jacobi identity,
\[
[[X,X'],V]=[X,[X',V]]-[X',[X,V]],
\]
and two applications of \eqref{extension_adjoint} show that this bracket is a section of $F_i$ if $j=i$, and of $F_i\oplus F_j$ if $j\neq i$. By the Leibniz rule, $[F_i,F_i]\subset F_i$ and $[F_i,F_j]\subset F_i\oplus F_j$ for $i\neq j$, and expanding the brackets of sums proves the involutivity of every $\Delta_J$.
\end{proof}

\begin{rk}
\label{bracket_frames_rk}
In terms of structure functions, Proposition \ref{tuple_continuation_prop} implies that the families of structure functions listed in Proposition \ref{outline_transverse_vanishing} vanish near every stable point, in every weakly adapted frame $\{X_s\}_{s\in V\cup E}$ adapted to the splitting \eqref{rotation_tangent_decomp} in the following sense: the index set $E$ of the fields $X_e$ that complete the horizontal frame $\{X_v\}_{v\in V}$ of $D$ to a frame of $T\mathcal U$ splits as $E=\bigsqcup_iE_i$, and $X_e$ is a section of $D_i^2$ for $e\in E_i$. Then the classes of the fields $X_e$, $e\in E_i$, modulo $D$ form a basis of $W_i$; for instance, one can take for these fields brackets of pairs of horizontal fields $X_v,X_w$ with $v,w\in V_i$, as for the edge fields of a quasi-normal frame. For other choices of the fields $X_e$ this is no longer true in general: for $v,w\in V_i$, the components of $[X_v,X_w]$ along $D_t$, $t\neq i$, change when the fields $X_e$, $e\in E_i$, are modified by sections of $D_t$. On the other hand, the structure functions $c^s_{lr}$ with $l\in V_t$ and $r,s\in V_i$, $t\neq i$, which are also indexed by two distinct components, do not vanish in general. They are not the components of a tensor, and they are precisely the quantities that are integrated in the rotation argument below.
\end{rk}

\medskip\noindent\textbf{An example with a jumping tuple.}
Let us show by an example that the characteristic tuple may indeed jump at stable points, so that the extension from $\Omega$ to $\mathcal U$ in Proposition \ref{tuple_continuation_prop} is not vacuous. On $\mathbb R^7$ with coordinates $(x_1,x_2,x_3,x_4,z_1,z_2,z_3)$, let $D_0$ be the distribution spanned by the vector fields
\[
 X_1=\partial_{x_1},\quad
 X_2=\partial_{x_2}+x_1\partial_{z_1},\quad
 X_3=\partial_{x_3}+x_1\partial_{z_2}+x_2\partial_{z_3},\quad
 X_4=\partial_{x_4}+\tfrac12x_1^2\partial_{z_3}.
\]
The only nonzero brackets between these fields are
\[
 [X_1,X_2]=\partial_{z_1},\quad
 [X_1,X_3]=\partial_{z_2},\quad
 [X_2,X_3]=\partial_{z_3},\quad
 [X_1,X_4]=x_1\partial_{z_3}.
\]
Thus $D_0$ has rank $4$ and step $2$ everywhere. At the points where $x_1\ne0$, the operator $\operatorname{ad}X_1$ is onto the three-dimensional center, and the characteristic tuple is $(3)$. At the points where $x_1=0$, the symbol is the direct sum of the free step-two nilpotent Lie algebra on three generators and an abelian horizontal line, so its characteristic tuple is $(2,1)$. Let $h$ be the sub-Riemannian metric on this distribution for which these four fields are orthonormal. On $\mathbb R^7\times\mathbb R_s$, consider the distribution $D=D_0\oplus\langle\partial_s\rangle$. On this distribution, the metrics
\[
 g_1=h+ds^2,\qquad g_2=c\,h+d\,ds^2,
 \qquad c,d>0,\quad c\ne d,
\]
form an affinely equivalent pair of product type. As for every affinely equivalent pair, every point is stable, including the points of the hypersurface $x_1=0$, where the characteristic tuple jumps. These points do not belong to the regular locus of the characteristic tuple, so the graph calculations of the preceding sections do not apply near them, and at these points the conclusions of the main theorems rely on Proposition \ref{tuple_continuation_prop}.

\subsection{The rotation argument}
\label{rotation_argument_subsec}

Let now $q_0$ be an arbitrary point of $\mathcal U$, and let $\{X_s\}_{s\in V\cup E}$ be a weakly adapted frame near $q_0$; such frames exist near every stable point. Only the horizontal fields $X_v$, $v\in V$, will be used: they form a $g_1$-orthonormal frame of $D$, and $\{X_v\}_{v\in V_i}$ is a frame of $D_i$ for every $i$. As before, $c^s_{lr}$ denotes the coefficient of $X_s$ in $[X_l,X_r]$. If $l\in V_t$ and $r\in V_i$ with $t\neq i$, then $[X_l,X_r]$ is a section of $D_t\oplus D_i$ by \eqref{extension_mixed}, so these coefficients do not depend on the choice of the fields $X_e$, $e\in E$, completing the horizontal frame. Every quasi-normal frame is weakly adapted, but quasi-normal frames are not needed here.

By Proposition \ref{tuple_continuation_prop}, we have the direct sum decomposition \eqref{rotation_tangent_decomp}, and every distribution $\Delta_J$ defined in \eqref{rotation_delta} is involutive. We denote by $\mathcal F_J(q)$ the local leaf of $\Delta_J$ through $q$.

Fix $i$ and put $S_i=[1:k]\setminus\{i\}$.  By \eqref{rotation_tangent_decomp}, the foliations tangent to $D_i^2$ and $\Delta_{S_i}$ are transverse.  After shrinking $\mathcal U$ around $q_0$, their leaves form a local product: every $\mathcal F_{\{i\}}(q_0)$ and $\mathcal F_{S_i}(q)$ meet in exactly one point.  Hence the map
\begin{equation}
\label{rotation_projection}
\pi_i:=\pi_i^{q_0}:\mathcal U\longrightarrow \mathcal F_{\{i\}}(q_0),
\qquad
\pi_i(q):=\mathcal F_{\{i\}}(q_0)\cap \mathcal F_{S_i}(q),
\end{equation}
is a smooth submersion satisfying
\begin{equation}
\label{rotation_projection_kernel}
\ker d\pi_i=\Delta_{S_i},
\end{equation}
and its restriction to every $\mathcal F_{\{i\}}(q)$ is a local diffeomorphism onto $\mathcal F_{\{i\}}(q_0)$.

If $k=1$, then $D_1=D$, $D_1^2=TM$, $\pi_1=\mathrm{id}$, and no rotation is needed.  Writing $\lambda_1=\alpha_1^2$, set $\mu_1=\lambda_1^{-1/2}$, $\gamma_1=1$ (the empty product), and $b_1=g_1$.  Then $\lambda_1=\mu_1^{-2}$ and $g_2=\lambda_1g_1$, exactly as in Definition \ref{Gen_Levi-Civita}.  If $\dim D_1>1$, Proposition \ref{eigenconst} and bracket generation of $D_1=D$ imply that $\lambda_1$, hence $\mu_1$, is constant.  Thus both the projective and affine statements follow in this case.  In the remainder of this subsection we assume $k\geq2$. In particular, we will assume the pair $(g_1,g_2)$ is non-conformal so that Proposition \ref{first divi} applies.

For every $v\in V_i$, define $\widetilde X_v$ to be the unique section of $D_i^2$ satisfying
\begin{equation}
\label{rotated_lift}
 d\pi_i\bigl(\widetilde X_v(q)\bigr)=X_v\bigl(\pi_i(q)\bigr).
\end{equation}
The uniqueness follows from \eqref{rotation_tangent_decomp}--\eqref{rotation_projection_kernel}, because
$d\pi_i|_{D_i^2(q)}:D_i^2(q)\to D_i^2(\pi_i(q))$ is an isomorphism.

\begin{Lemma}
\label{rotation_section_lemma}
For every $i$ and every $v\in V_i$, the vector field $\widetilde X_v$ is a section of $D_i$.  Consequently,
\begin{equation}
\label{rotation_projection_Di}
d\pi_i\bigl(D_i(q)\bigr)=D_i\bigl(\pi_i(q)\bigr).
\end{equation}
\end{Lemma}

\begin{proof}
Set $H_{\widehat i}:=\bigoplus_{t\neq i}D_t$.  We first show that every section $V$ of $H_{\widehat i}$ is an infinitesimal symmetry of $D_i\oplus\Delta_{S_i}$, namely
\begin{equation}
\label{rotation_infsymm}
[V,D_i\oplus\Delta_{S_i}]\subset D_i\oplus\Delta_{S_i}.
\end{equation}
It is enough to take $V\in D_t$ with $t\neq i$.  If $Y\in D_i$, then $[V,Y]\in D_t\oplus D_i\subset D_i\oplus\Delta_{S_i}$ by \eqref{extension_mixed}.  If instead $Y$ is a section of $\Delta_{S_i}$, then $[V,Y]\in\Delta_{S_i}$, because $V$ is also a section of $\Delta_{S_i}$, which is involutive by Proposition \ref{tuple_continuation_prop}.  This proves \eqref{rotation_infsymm}.

Let $p\in\mathcal F_{\{i\}}(q_0)$ and let $q$ lie on the local leaf $\mathcal F_{S_i}(p)$.  The distribution $H_{\widehat i}$ is bracket generating on this leaf: its Lie closure contains each $D_t^2=D_t+[D_t,D_t]$, $t\neq i$, and is contained in the involutive distribution $\Delta_{S_i}$.  Hence, after shrinking $\mathcal U$ if necessary, the Rashevskii--Chow theorem connects $p$ to $q$ by a finite concatenation of integral curves of local sections of $H_{\widehat i}$.

For such a section $V$, relation \eqref{rotation_infsymm} implies that the local flow of $V$ preserves $D_i\oplus\Delta_{S_i}$.  Since $V$ is tangent to a fiber of $\pi_i$, we also have $\pi_i\circ e^{tV}=\pi_i$.  Writing $\phi_t=e^{tV}$, these two facts and \eqref{rotation_projection_kernel} give
\begin{align*}
d\pi_i\bigl(D_i(\phi_t(x))\bigr)
&=d\pi_i\bigl((D_i\oplus\Delta_{S_i})(\phi_t(x))\bigr)\\
&=d\pi_i\bigl((\phi_t)_*(D_i\oplus\Delta_{S_i})(x)\bigr)\\
&=d\pi_i\bigl((D_i\oplus\Delta_{S_i})(x)\bigr)
=d\pi_i\bigl(D_i(x)\bigr),
\end{align*}
where in the third equality we used $\pi_i\circ\phi_t=\pi_i$.  Starting from $p$, where $\pi_i$ is the identity on $\mathcal F_{\{i\}}(q_0)$, and iterating along the above concatenation gives
\[
d\pi_i(D_i(q))=D_i(p),\qquad p=\pi_i(q).
\]
The unique $D_i^2$-lift in \eqref{rotated_lift} therefore belongs to $D_i$, which proves the first assertion. The displayed equality proves \eqref{rotation_projection_Di}.
\end{proof}

\begin{Lemma}
\label{rotation_commute_lemma}
If $i\neq j$, $v\in V_i$, and $w\in V_j$, then
\begin{equation}
\label{rotation_commute}
[\widetilde X_v,\widetilde X_w]=0.
\end{equation}
\end{Lemma}

\begin{proof}
Let $V$ be any section of $\Delta_{S_i}$.  The vector field $V$ is $\pi_i$-related to the zero field, while $\widetilde X_v$ is $\pi_i$-related to $X_v$. Thus, pushforward of the bracket by $\pi_i$ gives
\[
d\pi_i[V,\widetilde X_v]=0,
\]
so, by \eqref{rotation_projection_kernel},
\begin{equation}
\label{rotation_vertical_bracket}
[V,\widetilde X_v]\in\Delta_{S_i}.
\end{equation}
Taking $V=\widetilde X_w$ and then interchanging $i$ and $j$ yields
\[
[\widetilde X_v,\widetilde X_w]
\in\Delta_{S_i}\cap\Delta_{S_j}.
\]
At the same time, $\widetilde X_v\in D_i^2$ and $\widetilde X_w\in D_j^2$, and the involutivity of $D_i^2\oplus D_j^2$ gives
\[
[\widetilde X_v,\widetilde X_w]\in D_i^2\oplus D_j^2.
\]
By the direct sum \eqref{rotation_tangent_decomp},
\[
\Delta_{S_i}\cap\Delta_{S_j}\cap(D_i^2\oplus D_j^2)=\{0\},
\]
which proves \eqref{rotation_commute}.
\end{proof}

By Lemma \ref{rotation_section_lemma}, $\{\widetilde X_v\}_{v\in V_i}$ is a frame of $D_i$.  Thus there is a smooth matrix-valued function
$T_i=(t^i_{vr})_{v,r\in V_i}\in\mathrm{GL}(|V_i|)$ such that
\begin{equation}
\label{rotation_transition}
\widetilde X_v=\sum_{r\in V_i}t^i_{vr}X_r,
\qquad v\in V_i.
\end{equation}
The commutativity in Lemma \ref{rotation_commute_lemma} gives the transverse differential equations satisfied by $T_i$.

\begin{Lemma}
\label{rotation_transition_pde}
If $t\neq i$, $l\in V_t$, and $v,s\in V_i$, then
\begin{equation}
\label{rotation_commutativity}
X_l(t^i_{vs})=-\sum_{r\in V_i}c^s_{lr}t^i_{vr}.
\end{equation}
\end{Lemma}

\begin{proof}
Fix $a\in V_t$ and expand $0=[\widetilde X_a,\widetilde X_v]$ using \eqref{rotation_transition} for both components.  Taking the $D_i$-component and using \eqref{extension_mixed} gives, for each $s\in V_i$,
\[
0=\sum_{l\in V_t}t^t_{al}
\left(
X_l(t^i_{vs})+\sum_{r\in V_i}c^s_{lr}t^i_{vr}
\right).
\]
Since $T_t$ is invertible, the expression in parentheses vanishes for every $l\in V_t$, which is \eqref{rotation_commutativity}.
\end{proof}

Define the Gram matrix of the rotated frame on $D_i$ by
\begin{equation}
\label{rotation_gram}
G^i_{vw}:=g_1(\widetilde X_v,\widetilde X_w)
=\sum_{r\in V_i}t^i_{vr}t^i_{wr},
\qquad v,w\in V_i.
\end{equation}
For $l\in V_t$, $t\neq i$, Item (3) of Proposition \ref{first divi} gives
\begin{equation}
\label{rotation_offdiag_skew}
c^s_{lr}+c^r_{ls}=0,
\qquad r,s\in V_i,\quad r\neq s.
\end{equation}
Moreover, Item (1) of Proposition \ref{first divi}, with $r\in V_i$, gives
\begin{equation}
\label{rotation_kappa}
X_l\left(\frac{\lambda_i}{\lambda_t}\right)
=2c^r_{lr}\left(1-\frac{\lambda_i}{\lambda_t}\right),
\qquad \lambda_s:=\alpha_s^2.
\end{equation}
Because the left-hand side does not depend on $r$ and $\lambda_i\neq\lambda_t$, the quantity $c^r_{lr}$ is independent of $r\in V_i$.  Denote its common value by $\kappa_{li}$.  Differentiating \eqref{rotation_gram} and using \eqref{rotation_commutativity}, \eqref{rotation_offdiag_skew}, and \eqref{rotation_kappa}, we obtain
\begin{align}
X_l(G^i_{vw})
&=-\sum_{r,s\in V_i}(c^s_{lr}+c^r_{ls})t^i_{vr}t^i_{ws}\notag\\
&=-2\kappa_{li}G^i_{vw}.
\label{rotation_gram_ode}
\end{align}
On the base leaf $\mathcal F_{\{i\}}(q_0)$, the projection $\pi_i$ is the identity, so uniqueness in \eqref{rotated_lift} gives $\widetilde X_v=X_v$ and hence $G^i=I$.  Fix $v_0\in V_i$.  Since $G^i_{v_0v_0}>0$, equation \eqref{rotation_gram_ode} implies
\[
X_l\left(\frac{G^i_{vw}}{G^i_{v_0v_0}}\right)=0
\qquad(l\in V_t,\ t\neq i).
\]
The fields $\{X_l:l\in V_t,\ t\neq i\}$ bracket generate $\Delta_{S_i}$ on each fiber of $\pi_i$, as in the proof of Lemma \ref{rotation_section_lemma}.  Therefore the above ratios are constant on every such fiber and equal to their values on the base leaf.  Consequently there is a smooth positive function $\sigma_i$ such that
\begin{equation}
\label{rotation_conformal_frame}
G^i_{vw}=\sigma_i\delta_{vw},
\qquad
\sigma_i=1\quad\text{on }\mathcal F_{\{i\}}(q_0),
\end{equation}
and \eqref{rotation_gram_ode} gives
\begin{equation}
\label{rotation_sigma_derivative}
X_l(\log\sigma_i)=-2\kappa_{li},
\qquad l\in V_t,
\quad t\neq i.
\end{equation}
Thus the rotated frame is orthogonal in the projective case, with the same squared length in every direction of a fixed $D_i$.  In the affine case all $\lambda_i$ are constant, so \eqref{rotation_kappa} gives $\kappa_{li}=0$ and \eqref{rotation_conformal_frame} reduces to $G^i=I$, recovering the orthonormal rotation of \cite[Section~4.7]{Zaifeng}.

It remains to identify the scalar $\sigma_i$ with the Levi-Civita weights.  Set
\begin{equation}
\label{rotation_P_B}
P:=\left(\prod_{r=1}^k\lambda_r\right)^{1/(k+1)},
\qquad
B_i:=\frac{P}{\lambda_i}.
\end{equation}
All these functions are positive and smooth.  Item (2) of Proposition \ref{first divi} says that, for $l\in V_t$ and $j\neq t$,
\[
0=X_l\left(\frac{\alpha_j^2}{\alpha_t}\right),
\]
and therefore
\begin{equation}
\label{rotation_eigen_derivatives}
X_l(\log\lambda_j)=\frac12X_l(\log\lambda_t),
\qquad j\neq t.
\end{equation}
Summing \eqref{rotation_eigen_derivatives} over the $k-1$ indices $j\neq t$ gives
\begin{equation}
\label{rotation_P_derivative}
X_l(\log P)=\frac12X_l(\log\lambda_t).
\end{equation}
Hence, whenever $i\neq t$,
\begin{equation}
\label{rotation_B_basic}
X_l(B_i)=0.
\end{equation}
Thus $B_i$ is constant along the bracket-generating distribution $H_{\widehat i}$ on every fiber of $\pi_i$, and there is a unique positive smooth function $\mu_i$ on $\mathcal F_{\{i\}}(q_0)$ such that
\begin{equation}
\label{rotation_beta_pullback}
B_i=\pi_i^*\mu_i.
\end{equation}
Moreover,
\begin{equation}
\label{rotation_B_ratio}
\frac{B_i}{B_j}=\frac{\lambda_j}{\lambda_i},
\end{equation}
so $\mu_i(q_0)\neq\mu_j(q_0)$ for $i\neq j$.

Suppose now that $\dim D_i>1$.  Proposition \ref{eigenconst} gives $X_v(\lambda_i)=0$ for every $v\in V_i$.  Applying \eqref{rotation_eigen_derivatives} with $t=i$ then gives $X_v(\lambda_j)=0$ for every $j\neq i$ as well.  Consequently $X_v(B_i)=0$ for all $v\in V_i$.  Since $D_i$ is bracket generating on every leaf of $\mathcal F_{\{i\}}$ (indeed $T\mathcal F_{\{i\}}=D_i^2=D_i+[D_i,D_i]$), the function $\mu_i$ in \eqref{rotation_beta_pullback} is constant on $\mathcal F_{\{i\}}(q_0)$.  This is exactly the constancy requirement in Definition \ref{Gen_Levi-Civita} (no such requirement is imposed if $\dim D_i=1$).

Define
\begin{equation}
\label{rotation_gamma}
\gamma_i:=\prod_{j\neq i}|B_j-B_i|.
\end{equation}
The factors never vanish by \eqref{rotation_B_ratio} and the distinctness of the eigenvalues.  Fix $l\in V_t$ with $t\neq i$.  By \eqref{rotation_B_basic}, every $B_j$ with $j\neq t$ is annihilated by $X_l$, while \eqref{rotation_P_derivative} gives
\[
X_l(\log B_t)=-\frac12X_l(\log\lambda_t).
\]
Therefore
\begin{align}
X_l(\log\gamma_i)
&=X_l\log|B_t-B_i|\notag\\
&=-\frac12\frac{B_t}{B_t-B_i}X_l(\log\lambda_t)\notag\\
&=-\frac12\frac{\lambda_i}{\lambda_i-\lambda_t}X_l(\log\lambda_t).
\label{rotation_gamma_derivative}
\end{align}
On the other hand, dividing \eqref{rotation_kappa} by $\lambda_i/\lambda_t$ and then using \eqref{rotation_eigen_derivatives} with $j=i$ gives
\[
-\frac12X_l(\log\lambda_t)
=X_l\log\left(\frac{\lambda_i}{\lambda_t}\right)
=-2\kappa_{li}\frac{\lambda_i-\lambda_t}{\lambda_i}.
\]
Hence
\begin{equation}
\label{rotation_gamma_sigma_match}
X_l(\log\gamma_i)=-2\kappa_{li}=X_l(\log\sigma_i).
\end{equation}
It follows that
\begin{equation}
\label{rotation_rho_basic}
\rho_i:=\frac{\sigma_i}{\gamma_i}
\end{equation}
is constant on every fiber of $\pi_i$.  Thus there is a positive smooth function $\overline\rho_i$ on $\mathcal F_{\{i\}}(q_0)$ such that
$\rho_i=\pi_i^*\overline\rho_i$.

Let $h_i$ be the restriction of $g_1$ to $D_i$ along the base leaf $\mathcal F_{\{i\}}(q_0)$.  By \eqref{rotated_lift} and \eqref{rotation_conformal_frame}, as metrics on $D_i$ we have
\begin{equation}
\label{rotation_metric_sigma}
g_1|_{D_i}=\sigma_i\,\pi_i^*h_i.
\end{equation}
Set
\begin{equation}
\label{rotation_bi}
b_i:=\overline\rho_i\,h_i.
\end{equation}
Then \eqref{rotation_rho_basic}--\eqref{rotation_metric_sigma} yield
\begin{equation}
\label{rotation_metric_g1_block}
g_1|_{D_i}=\gamma_i\,\pi_i^*b_i.
\end{equation}
Since the $D_i$'s are mutually $g_1$-orthogonal eigenspaces of the $g_1$-self-adjoint transition operator $S$, summing \eqref{rotation_metric_g1_block} gives
\begin{equation}
\label{rotation_metric_g1}
g_1=\sum_{i=1}^k\gamma_i\,\pi_i^*b_i.
\end{equation}
Furthermore $S|_{D_i}=\lambda_i\operatorname{Id}$, so
\begin{equation}
\label{rotation_metric_g2}
g_2=\sum_{i=1}^k\lambda_i\gamma_i\,\pi_i^*b_i.
\end{equation}
Finally, from \eqref{rotation_P_B},
\[
\prod_{r=1}^kB_r^{-1}
=\frac{\prod_r\lambda_r}{P^k}=P,
\]
and hence
\begin{equation}
\label{rotation_lambda_formula}
\lambda_i
=B_i^{-1}\prod_{r=1}^kB_r^{-1}
=\frac{1}{\pi_i^*\mu_i}
\prod_{r=1}^k\frac{1}{\pi_r^*\mu_r}.
\end{equation}
Together with \eqref{rotation_gamma} and \eqref{rotation_beta_pullback}, equations \eqref{rotation_metric_g1}, \eqref{rotation_metric_g2}, and \eqref{rotation_lambda_formula} are exactly the metric identities in Definition \ref{Gen_Levi-Civita}.

We have also verified the two distributional requirements in that definition: the leaves $\mathcal F_i:=\mathcal F_{\{i\}}$ have tangent bundle $D_i^2$, the distribution $D_i$ is bracket generating on each such leaf, and \eqref{rotation_projection_Di} is the required projection property.  Consequently, the forward implication of Theorem \ref{mainthm} holds near the arbitrary stable point $q_0$. The converse implication was established in \cite[Proposition 5.5]{jmz2019}. This completes the proof of Theorem \ref{mainthm}.

For Theorem \ref{mainthm_eisen}, affine equivalence implies that all $\lambda_i$ are constant \cite[Proposition 4.7]{jmz2019}.  Hence $P$ and every $B_i=\pi_i^*\mu_i$ in \eqref{rotation_P_B} are constant, so the functions $\mu_i$ are constant.  As noted after \eqref{rotation_conformal_frame}, the same constancy also forces $\kappa_{li}=0$, and the rotated frame is actually $g_1$-orthonormal, exactly as in the affine argument mentioned at the beginning of this section.  Since every point is stable for an affinely equivalent pair, this proves the forward implication of Theorem \ref{mainthm_eisen} near every point, and the converse implication is again given by \cite[Proposition 5.5]{jmz2019}. This completes the proof of Theorem \ref{mainthm_eisen}.

\section{Comparison with the Weyl type theorems of~\texorpdfstring{\cite{JMZ21}}{[JMZ21]}}\label{weyl_comparison_sec}

In~\cite{JMZ21} the sub-Riemannian Weyl theorem was proved under additional assumptions of two kinds. Theorems~1.11 and~1.12 of~\cite{JMZ21} require that the set of \emph{complex non-strictly normal geodesics} through a point (i.e.\ of complex normal geodesics which are simultaneously abnormal) is not too large, while Theorems~1.14 and~1.15 of~\cite{JMZ21} require that all \emph{complex abnormal extremals} are \emph{of minimal order}. In Theorems~1.11 and~1.14 of~\cite{JMZ21}, which are formulated in the smooth category, these assumptions are imposed on the nilpotent approximations of the sub-Riemannian structure at the points of an open dense set; in Theorems~1.12 and~1.15 of~\cite{JMZ21} they are imposed on the (real analytic) structure itself. In addition, a generic real analytic sub-Riemannian metric on a given distribution and an arbitrary sub-Riemannian metric on a generic real analytic distribution are Weyl rigid~\cite[Corollaries~1.13 and~1.19]{JMZ21}. Corollary~\ref{Weyl_thm} requires none of these assumptions, but it concerns step $2$ distributions only.

In this section we discuss for which step $2$ distributions Corollary~\ref{Weyl_thm} is actually new. For this purpose, we first express the fundamental algebraic system of a nilpotent approximation, together with its complex normal and abnormal extremals, in terms of the bracket of the Tanaka symbol and of the inner product given by the metric. This allows us to analyze the condition of minimal order and the condition on non-strictly normal geodesics separately, and to decide, for several natural families of symbols and metrics, whether these conditions hold. We then show that the octonionic H-type metrics with centers of dimension six and seven satisfy the condition on non-strictly normal geodesics, although for them one of the cases in the proof of the corresponding theorem requires an additional argument. Next, we construct structures, including structures with indecomposable symbols, for which the hypotheses of the theorems recalled above are not satisfied. Finally, we explain why the two-layer argument on which our proof is based does not extend to distributions of higher step. Our conclusions are summarized in Table~\ref{weyl_comparison_table}, which records, for each of the families considered below, which of the hypotheses hold.

Before turning to these questions, we record a consequence of the step $2$ Weyl theorem which will be used in several of our examples. It follows directly from the splitting of the Tanaka symbol for projectively equivalent metrics, and its proof does not require the classification argument of the preceding sections.
\begin{Coro}[Projective rigidity]\label{wc:projective_rigidity}
If a connected step $2$ distribution has indecomposable Tanaka symbol at every point of an open dense set, then every sub-Riemannian metric on it is projectively rigid, i.e., every sub-Riemannian metric projectively equivalent to it is a constant multiple of it.
\end{Coro}
\begin{proof}
Let two sub-Riemannian metrics on such a distribution be projectively equivalent. Since the Tanaka symbol is indecomposable on an open dense set, the conformal-projective rigidity result of \cite{jmz2019}, recalled in \cite[Theorem 1.7]{JMZ21}, shows that the two metrics are conformal. Being both conformal and projectively equivalent, they are constantly proportional by Corollary \ref{Weyl_thm}, i.e., the conformal factor is constant.
\end{proof}
\subsection{Nilpotent approximations of step 2 distributions and their fundamental algebraic system}\label{wc:nilp_subsec}
Fix a point $q\in M$ and set $\m:=\m(q)=\m_{-1}\oplus\m_{-2}$, $m=\dim\m_{-1}$, $r=\dim\m_{-2}=n-m$, and $g:=g_1(q)$, an inner product on $\m_{-1}=D(q)$. Let $(G,\widehat D,\widehat g)$ be the nilpotent approximation of $(M,D,g_1)$ at $q$, i.e., the left-invariant sub-Riemannian structure on the simply connected Lie group $G$ with Lie algebra $\m$, defined by the subspace $\m_{-1}$ and the inner product $g$. Since the hypotheses of~\cite{JMZ21} are formulated in terms of complex extremals, from now on all vector spaces are complexified and $g$ is extended to a complex bilinear form; we keep the same notation for the complexified objects.

We identify $T^*G$ with $G\times\m^*$ by left translations. Given $\lambda\in\m^*$, let $Z=Z(\lambda)\in\m_{-1}$ be defined by $g(Z,\cdot)=\lambda|_{\m_{-1}}$, and let $\eta=\eta(\lambda):=\lambda|_{\m_{-2}}\in\m_{-2}^*$. Given $\mu\in\m_{-2}^*$, define the $2$-form $\omega_\mu$ on $\m_{-1}$ and the $g$-skew-symmetric operator $J_\mu$ on $\m_{-1}$ by
\begin{equation}\label{wc:J}
\omega_\mu(X,Y):=\mu\big([X,Y]\big)=g(J_\mu X,Y),\qquad X,Y\in\m_{-1}
\end{equation}
(cf.~\eqref{J_eta}). Note that $\omega_\mu\neq0$ for $\mu\neq0$, because $[\m_{-1},\m_{-1}]=\m_{-2}$; hence $\rank\omega_\mu\geq2$ for $\mu\neq0$. The sub-Riemannian Hamiltonian of $\widehat g$ is $\hh(\lambda)=\frac12g\big(Z(\lambda),Z(\lambda)\big)$, and the hypersurface $\{\hh=0\}$ is called the \emph{complex null cone}. Note that it is strictly larger than $\widehat D^\perp=\{Z=0\}$ (cf.~\cite[Remark~1.9]{JMZ21}).

Since $[\m_{-1},\m_{-2}]=0$, the component $\eta$ is preserved along the normal extremals of $\widehat g$, and the normal extremal with initial point $(e,\lambda)$, where $e$ is the identity of $G$, is $t\mapsto\big(\gamma(t),\lambda(t)\big)$ with
\begin{equation}\label{wc:nilp_extremal}
Z\big(\lambda(t)\big)=e^{tJ_\eta}Z,\qquad \eta\big(\lambda(t)\big)=\eta,\qquad \gamma(t)^{-1}\dot\gamma(t)=e^{tJ_\eta}Z,
\end{equation}
where $\gamma(t)^{-1}\dot\gamma(t)\in\m_{-1}$ denotes the left-trivialized velocity. For $p\geq1$ and $\lambda\in\m^*$ define the linear map $\cA^p(\lambda)\colon\m_{-2}^*\to\m_{-1}^*$ by
\begin{equation}\label{wc:Ap}
\cA^p(\lambda):=\Big(\ad\big(J_\eta^{\,p-1}Z\big)\Big)^*,\qquad\text{i.e.,}\qquad \cA^p(\lambda)\mu=\omega_\mu\big(J_\eta^{\,p-1}Z,\cdot\big)=g\big(J_\mu J_\eta^{\,p-1}Z,\cdot\big).
\end{equation}
These maps describe the layers of the fundamental algebraic system of the nilpotent approximation. More precisely, let $\{X_s\}_{s\in V\cup E}$ be a left-invariant frame on $G$ such that $\{X_v\}_{v\in V}$ is a $g$-orthonormal basis of $\m_{-1}$ and $\{X_e\}_{e\in E}$ is a basis of $\m_{-2}$. Then the $p$-th layer of the matrix of the fundamental algebraic system of $(G,\widehat D,\widehat g)$ is the matrix of $\cA^p(\lambda)$ with respect to the bases of $\m_{-2}^*$ and $\m_{-1}^*$ dual to $\{X_e\}_{e\in E}$ and $\{X_v\}_{v\in V}$. This follows from the nilpotent calculation in Appendix \ref{app:eigen} and from \eqref{all_layer_operator}, applied with $k=1$; here we use that the matrix $A$ of the fundamental algebraic system depends only on $g_1$ and on the frame.

\begin{Def}\label{wc:ample_def}
A covector $\lambda\in\m^*$ is called \emph{ample} if the linear map $\cA(\lambda):=\big(\cA^1(\lambda),\cA^2(\lambda),\ldots\big)$ from $\m_{-2}^*$ to $\prod_{p\geq1}\m_{-1}^*$ is injective (cf.~\cite[Proposition~3.11]{jmz2019}). The set of non-ample covectors is denoted by $\cN\subset\m^*$.
\end{Def}

By the Cayley--Hamilton theorem, $J_\eta^{\,p}Z$ is a linear combination of $Z,J_\eta Z,\ldots,J_\eta^{\,m-1}Z$ with coefficients polynomial in $\eta$, so $\lambda$ is ample if and only if the matrix
\begin{equation}\label{wc:Am}
\cA_{\leq m}(\lambda):=\begin{pmatrix}\cA^1(\lambda)\\ \vdots\\ \cA^m(\lambda)\end{pmatrix}
\end{equation}
has rank $r$. Thus $\cN$ is the common zero set of the $r\times r$ minors of $\cA_{\leq m}(\lambda)$, which are polynomials in $\lambda$. By~\eqref{wc:Ap}, a covector $\lambda=(Z,\eta)$ belongs to $\cN$ if and only if there exists $\mu\neq0$ in $\m_{-2}^*$ such that
\begin{equation}\label{wc:nonample_cond}
J_\mu J_\eta^{\,j} Z=0\qquad\text{for all }j\geq0 .
\end{equation}
In particular $\cN$ is invariant under the separate rescalings $Z\mapsto sZ$, $\eta\mapsto s'\eta$, and, by~\eqref{wc:nilp_extremal}, under the normal extremal flow.

\begin{Lemma}\label{wc:nonstrict_lemma}
Let $\lambda=(Z,\eta)\in\m^*$ with $Z\neq0$, and let $\gamma$ be the complex normal geodesic of $\widehat g$ with initial covector $\lambda$. Then the space of abnormal lifts of $\gamma$, with the zero lift adjoined, is naturally identified with
\begin{equation}\label{wc:kerA}
\ker\cA(\lambda)=\big\{\mu\in\m_{-2}^*:\ J_\mu J_\eta^{\,j}Z=0\ \text{ for all }j\geq0\big\}.
\end{equation}
In particular, $\gamma$ is non-strictly normal if and only if $\lambda\in\cN$, and in this case its corank in the sense of~\cite[\S1]{JMZ21} is equal to $\dim\ker\cA(\lambda)$.
\end{Lemma}

\begin{proof}
In the left trivialization, the annihilator $\widehat D^\perp$ is $G\times\m_{-2}^*$, where $\m_{-2}^*$ is identified with the covectors on $\m$ vanishing on $\m_{-1}$. By the Pontryagin maximum principle, a curve $t\mapsto(\gamma(t),\mu(t))$ in $\widehat D^\perp$ with $\mu(t)\neq0$ is an abnormal extremal if and only if $\dot\mu(t)=\big(\ad\xi(t)\big)^*\mu(t)$ and the right-hand side of the latter identity vanishes on $\m_{-1}$, where $\xi(t)=\gamma(t)^{-1}\dot\gamma(t)\in\m_{-1}$. Since $[\m_{-1},\m_{-2}]=0$, the restriction of $\big(\ad\xi\big)^*\mu$ to $\m_{-2}$ vanishes, so that $\mu(t)\equiv\mu$ is constant, and the remaining condition reads $\omega_\mu\big(\xi(t),\cdot\big)\equiv0$, i.e.\ $J_\mu\xi(t)\equiv0$. For the geodesic $\gamma$ we have $\xi(t)=e^{tJ_\eta}Z$ by~\eqref{wc:nilp_extremal}, and $J_\mu e^{tJ_\eta}Z\equiv0$ if and only if~\eqref{wc:nonample_cond} holds. This proves the identification~\eqref{wc:kerA}, the second equality in~\eqref{wc:kerA} being~\eqref{wc:Ap}. The last statement follows from the definition of the corank (see~\cite[Remark~1.10]{JMZ21}).
\end{proof}

Note also that if $\mu\in\ker\cA(\lambda)$, then $(J_\eta+J_\mu)^jZ=J_\eta^jZ$ for all $j$, so the covectors $(Z,\eta+\mu)$, $\mu\in\ker\cA(\lambda)$, are exactly the normal lifts of $\gamma$, in accordance with~\cite[Remark~1.10]{JMZ21}.

\subsection{The condition of minimal order}\label{wc:min_order_subsec}
Recall (\cite[\S5.1]{JMZ21}) that for a distribution $D$ one sets
\begin{align*}
\widetilde W_D&:=\{\lambda\in D^\perp\setminus\{\text{zero section}\}:\ \ker\sigma|_{D^\perp}(\lambda)\neq0\},\\
W_D&:=\{\lambda\in\widetilde W_D:\ \dim\ker\sigma|_{\widetilde W_D}(\lambda)=1\},
\end{align*}
where $\sigma$ is the (complexified) canonical symplectic form and $W_D$ is defined on the smooth part of $\widetilde W_D$. Recall also that a complex abnormal extremal $\Gamma$ is called \emph{of minimal order} if $\Gamma\cap W_D$ has full measure in $\Gamma$ \cite[Definition~5.2]{JMZ21}.

\begin{Prop}\label{wc:min_order_prop}
Let $(G,\widehat D)$ be a step $2$ Carnot group as above, $m=\dim\m_{-1}$, $r=\dim\m_{-2}$.
\begin{enumerate}
\item If $m$ is even and $\omega_\mu$ is nondegenerate for every $\mu\neq0$ (over $\C$ this happens if and only if $r=1$ and $\widehat D$ is contact), then $\widehat D$ has no complex abnormal extremals.
\item If $m$ is even and $\omega_\mu$ is degenerate for some $\mu\neq0$ (this is always the case if $r\geq2$), then $W_{\widehat D}=\varnothing$ while $\widehat D$ has complex abnormal extremals; in particular, no complex abnormal extremal of $\widehat D$ is of minimal order.
\item If $m$ is odd, then all complex abnormal extremals of $\widehat D$ are of minimal order if and only if $\dim\ker\omega_\mu=1$ for every $\mu\in\m_{-2}^*\setminus\{0\}$.
\end{enumerate}
Consequently, the hypothesis of~\cite[Theorem~1.14]{JMZ21} holds for a step $2$ distribution $D$ if and only if at the points of an open dense subset of $M$ the Tanaka symbol of $D$ either is contact or has odd rank and satisfies $\dim\ker\omega_\mu=1$ for all nonzero complex $\mu\in\m_{-2}^*$.
\end{Prop}

\begin{proof}
In the left trivialization, the tangent space to $\widehat D^\perp=G\times\m_{-2}^*$ at $(g',\mu)$ consists of pairs $(X,\xi)$, $X\in\m$, $\xi\in\m_{-2}^*$, and the canonical symplectic form is given, up to sign, by $\sigma\big((X_1,\xi_1),(X_2,\xi_2)\big)=\xi_1(X_2)-\xi_2(X_1)-\mu\big([X_1,X_2]\big)$. Testing $(X_1,\xi_1)\in\ker\sigma|_{\widehat D^\perp}$ against $(0,\xi_2)$ shows that $X_1\in\m_{-1}$; testing it against $(X_2,0)$ with $X_2\in\m_{-2}$ gives $\xi_1=0$, and with $X_2\in\m_{-1}$ gives $X_1\in\ker\omega_\mu$. Hence
\begin{equation}\label{wc:ker_sigma}
\ker\sigma|_{\widehat D^\perp}(g',\mu)=\{(Y,0):\ Y\in\ker\omega_\mu\},
\end{equation}
so that $\widetilde W_{\widehat D}=G\times\{\mu\neq0:\ \omega_\mu\ \text{is degenerate}\}$, and the complex abnormal extremals are the curves $t\mapsto(\gamma(t),\mu)$ with $\mu$ constant and $\gamma(t)^{-1}\dot\gamma(t)\in\ker\omega_\mu$ (cf.\ the proof of Lemma~\ref{wc:nonstrict_lemma}).

Item (1) is immediate from~\eqref{wc:ker_sigma}. Note that if $r\geq2$, then the Pfaffian of $\omega_\mu$, a homogeneous polynomial of degree $m/2$ in $\mu$ on the complex space $\m_{-2}^*$ of dimension at least two, has nonzero zeros, so $\omega_\mu$ is degenerate for some $\mu\neq0$. Under the assumption of item (2) we have $\widetilde W_{\widehat D}\neq\varnothing$, and for every $\mu\neq0$ with $\omega_\mu$ degenerate and every nonzero $Y\in\ker\omega_\mu$ the curve $t\mapsto(\exp tY,\mu)$ is a complex abnormal extremal. Since $\widetilde W_{\widehat D}$ is invariant under left translations, the vectors $(Y,0)$, $Y\in\ker\omega_\mu$, are tangent to it; by~\eqref{wc:ker_sigma} they belong to $\ker\sigma|_{\widetilde W_{\widehat D}}$. Since $m$ is even, $\dim\ker\omega_\mu\geq2$ whenever $\omega_\mu$ is degenerate, so $\ker\sigma|_{\widetilde W_{\widehat D}}$ is never one-dimensional and $W_{\widehat D}=\varnothing$. If $m$ is odd, then $\widetilde W_{\widehat D}=\widehat D^\perp$ minus the zero section, so by~\eqref{wc:ker_sigma} we get $W_{\widehat D}=G\times\{\mu:\ \dim\ker\omega_\mu=1\}$. Since $\mu$ is constant along abnormal extremals, an abnormal extremal with $\dim\ker\omega_\mu=1$ lies in $W_{\widehat D}$, while the abnormal extremal $t\mapsto(\exp tY,\mu)$ with $\dim\ker\omega_\mu\geq3$ and $0\ne Y\in\ker\omega_\mu$ does not meet $W_{\widehat D}$ at all. The last statement of the proposition follows, because the hypothesis of~\cite[Theorem~1.14]{JMZ21} concerns the nilpotent approximations at the points of an open dense subset.
\end{proof}

For example, the Tanaka symbols of rank $3$ distributions in dimensions $5$ and $6$ and of quasi-contact distributions satisfy the condition of item (3) (cf.\ the list after~\cite[Corollary~1.16]{JMZ21}), whereas, for instance, the Tanaka symbols of rank $4$ distributions in dimension $6$, of free step $2$ distributions of rank at least $4$, and of all step $2$ distributions of even rank with $\dim\m_{-2}\geq2$ do not. A refinement of the condition of minimal order is given in~\cite[\S5.5]{JMZ21}; we shall examine it for the central amalgams constructed below, see Remark~\ref{wc:weak_order_rem}.

\subsection{The condition on non-strictly normal geodesics}\label{wc:nonstrict_subsec}
The hypothesis of~\cite[Theorem~1.11]{JMZ21} is that, for every positive integer $\kappa$, there is no $(n-2-\kappa)$-parametric family of corank $\kappa$ complex non-strictly normal geodesics of the nilpotent approximation through a point. Here, following~\cite[\S1]{JMZ21}, we say that a family of corank $\kappa$ geodesics through a point $q'$ is $s$-parametric if the corresponding set of initial covectors in $T^*_{q'}G\cap\hh^{-1}(c)$, for some $c\neq0$, has dimension $s+\kappa$. Note that the definition in~\cite{JMZ21} involves only the nonzero levels of the Hamiltonian; this point will play an important role below. By left-invariance, it is enough to consider $q'=e$.

\begin{Prop}\label{wc:criterion_prop}
The nilpotent approximation $(G,\widehat D,\widehat g)$ satisfies the hypothesis of~\cite[Theorem~1.11]{JMZ21} if and only if every irreducible component of codimension one of the non-ample set $\cN$ is contained in the complex null cone $\{\hh=0\}$. If $m\geq3$, this is equivalent to the property that the greatest common divisor of the $r\times r$ minors of the matrix $\cA_{\leq m}(\lambda)$ from~\eqref{wc:Am} is, up to a nonzero constant factor, a power of $\hh$, including the zeroth power.
\end{Prop}

\begin{proof}
We first note two general properties of $\cN$. First, $\cN$ is a proper subset of $\m^*$, i.e., not all $r\times r$ minors of $\cA_{\leq m}(\lambda)$ vanish identically. Indeed, real ample covectors exist: by the proof of item (1) of Lemma~\ref{Jacobi_lem}, applied to the real symbol $\m$ with the inner product $g$, for a generic real covector $\lambda=(Z,\eta)$ the dimensions of the successive increments of the subspaces $\sum_{j=0}^{s}[J_\eta^jZ,\m_{-1}]$, $s\geq0$, are equal to $r_0,\ldots,r_\omega$, whose sum is $r$ by~\eqref{sum_r}. Hence these subspaces exhaust $\m_{-2}$, and~\eqref{wc:nonample_cond} has no solution $\mu\neq0$. Second, every irreducible component $C$ of $\cN$ is a cone. Indeed, $\cN$ is invariant under the rescalings $\lambda\mapsto s\lambda$, $s\in\C^*$. Since the group $\C^*$ is connected, the closure of the image of $\C^*\times C$ under this action is an irreducible closed subset of $\cN$ containing $C$; by the maximality of $C$, it coincides with $C$.

Assume that $S$ is an irreducible component of $\cN$ of dimension $n-1$ which is not contained in $\{\hh=0\}$. Since $S$ is a cone and $\hh$ is homogeneous, $\hh|_S$ takes all values, and for every $c\neq0$ the set $S\cap\hh^{-1}(c)$ has dimension $n-2$. The function $\lambda\mapsto\dim\ker\cA(\lambda)$ is constant, say equal to $\kappa\geq1$, on a Zariski open dense subset $S^\circ$ of $S$. By Lemma~\ref{wc:nonstrict_lemma}, the $(n-2)$-dimensional set $S^\circ\cap\hh^{-1}(c)$ consists of initial covectors of corank $\kappa$ non-strictly normal geodesics. In the definition of an $s$-parametric family recalled above, the dimension of the set of initial covectors already includes the $\kappa$-dimensional affine space of normal lifts of each geodesic of the family (see the remark after Lemma~\ref{wc:nonstrict_lemma}). Hence no quotient construction is needed: since $n-2=(n-2-\kappa)+\kappa$, these geodesics form an $(n-2-\kappa)$-parametric family through $e$. Therefore, the hypothesis of~\cite[Theorem~1.11]{JMZ21} fails.

Conversely, assume that all components of $\cN$ of dimension $n-1$ lie in $\{\hh=0\}$. Every other component $C$ of $\cN$ is a cone of dimension at most $n-2$, so, for $c\neq0$, the set $C\cap\hh^{-1}(c)$ is either empty or of dimension at most $n-3$. Hence, for every $c\neq0$, the set $\cN\cap\hh^{-1}(c)$ has dimension at most $n-3$, and, by Lemma~\ref{wc:nonstrict_lemma}, a family of corank $\kappa$ complex non-strictly normal geodesics through $e$ has at most $n-3-\kappa$ parameters. This proves the first statement. The second one follows from the fact that the irreducible components of codimension one of the common zero set of a collection of polynomials, not all identically zero, are exactly the zero sets of the irreducible common factors of these polynomials, and that $\hh$ is irreducible if $m\geq3$.
\end{proof}

The following simple observation gives a sufficient condition which does not depend on the metric. Let
\[
\mathcal S_{\m}:=\big\{Z\in\m_{-1}:\ \ad Z\colon\m_{-1}\to\m_{-2}\ \text{is not onto}\big\}.
\]

\begin{Coro}\label{wc:ad_surj_cor}
For every inner product $g$ on $\m_{-1}$ we have $\cN\subset\{\lambda:\ Z(\lambda)\in\mathcal S_{\m}\}$. Consequently, if the (complex) codimension of $\mathcal S_{\m}$ in $\m_{-1}$ is at least two, then the nilpotent approximation satisfies the hypothesis of~\cite[Theorem~1.11]{JMZ21} for every sub-Riemannian metric. This is the case, in particular, if
\begin{enumerate}
\item $\m$ is a direct sum of contact (i.e.\ Heisenberg) symbols, $\m=\bigoplus_{i=1}^s\h_{2k_i+1}$; this includes the Tanaka symbols of generic rank $4$ distributions on $6$-dimensional manifolds, both in the hyperbolic and in the elliptic case, because their complexifications are isomorphic to $\h_3\oplus\h_3$;
\item $\m=\h_{2k+1}\oplus\C^a$, $k\geq1$, $a\geq0$, with the second summand in degree $-1$, i.e.\ an arbitrary step $2$ symbol of corank one; the case $a=1$ corresponds to quasi-contact distributions.
\end{enumerate}
\end{Coro}

\begin{proof}
Taking $j=0$ in~\eqref{wc:nonample_cond}, we get $J_\mu Z=0$ for some $\mu\neq0$, i.e.\ $\mu$ annihilates $[Z,\m_{-1}]$, so that $Z\in\mathcal S_{\m}$. If $\codim\mathcal S_{\m}\geq2$, then $\cN$ is contained in a subset of codimension at least two of $\m^*$, and Proposition~\ref{wc:criterion_prop} applies. In case (1), writing $Z=(Z_1,\ldots,Z_s)$, the map $\ad Z$ is onto if and only if $Z_i\neq0$ for all $i$, so $\mathcal S_{\m}$ is the union of the subspaces $\{Z_i=0\}$, which have codimension $2k_i\geq2$. In case (2), $\mathcal S_{\m}$ is the radical $\C^a$ of the Levi form, which has codimension $2k\geq2$.
\end{proof}

Note that the condition $\codim\mathcal S_{\m}\geq2$ implies that $\m$ is $\ad$-surjective in the sense of~\cite{Zaifeng}, i.e.\ $\ad Z\colon\m_{-1}\to\m_{-2}$ is onto for some $Z\in\m_{-1}$. For instance, for the symbol of rank $3$ distributions on $5$-dimensional manifolds, $\mathcal S_{\m}$ is a hyperplane; this case is covered by~\cite[Theorem~1.14]{JMZ21} instead (see Proposition~\ref{wc:min_order_prop}).

The next proposition explains why codimension one components of $\cN$ are rare.

\begin{Prop}\label{wc:cyclic_prop}
Assume that, for a generic $\eta\in\m_{-2}^*$, the operator $J_\eta$ is cyclic (i.e.\ its minimal polynomial coincides with its characteristic polynomial). Then every irreducible component of codimension one of $\cN$ is of the form $\{\lambda\in\m^*:\ \eta(\lambda)\in H\}$, where $H\subset\m_{-2}^*$ is an irreducible hypersurface; in particular, all covectors $\lambda$ with $\eta(\lambda)\in H$ are non-ample.
\end{Prop}

\begin{proof}
Fix $\eta$ such that $J_\eta$ is cyclic. A cyclic operator has only finitely many invariant subspaces $U_1,\ldots,U_N$. Indeed, $\m_{-1}$, regarded as a $\C[t]$-module on which $t$ acts by $J_\eta$, is isomorphic to $\C[t]/(p)$, where $p$ is the characteristic polynomial of $J_\eta$; hence the $J_\eta$-invariant subspaces, i.e., the submodules, correspond to the ideals of $\C[t]$ containing $p$, that is, to the finitely many monic divisors of $p$. If $(Z,\eta)\in\cN$ and $\mu\neq0$ satisfies~\eqref{wc:nonample_cond}, then the $J_\eta$-invariant subspace $\spn\{J_\eta^jZ\}_{j\geq0}$ is one of the $U_l$ and is contained in $\ker\omega_\mu$; conversely, if $U_l\subset\ker\omega_\mu$ for some $\mu\neq0$, then $(Z,\eta)\in\cN$ for all $Z\in U_l$. Hence the fiber $\cN_\eta:=\{Z:\ (Z,\eta)\in\cN\}$ is a finite union of subspaces contained in kernels of nonzero forms $\omega_\mu$, and, since $\rank\omega_\mu\geq2$, $\dim\cN_\eta\leq m-2$. Now let $S$ be an irreducible component of $\cN$ of dimension $n-1=m+r-1$. If the projection of $S$ to $\m_{-2}^*$ were dominant, then its generic fiber, which is contained in $\cN_\eta$ for a generic $\eta$, would have dimension $m-1$, a contradiction. Hence the closure $H$ of this projection has dimension at most $r-1$, and $S\subset\{\lambda:\ \eta(\lambda)\in H\}$; comparing dimensions, we get that $H$ is an irreducible hypersurface and $S=\{\lambda:\ \eta(\lambda)\in H\}$.
\end{proof}

In particular, codimension one components of $\cN$ can appear only if either $J_\eta$ is derogatory for a generic $\eta$, or there is a hypersurface of $\eta$'s over which all covectors are non-ample. For H-type metrics (see below) $J_\eta^2$ is a multiple of the identity, so the first possibility is realized.

For symbols of small rank, the dimension of the non-ample set can be estimated directly, for every inner product and without computing any minors.

\begin{Prop}[Small horizontal ranks]\label{wc:small_rank_prop}
For every inner product, the non-ample set has complex codimension at least two for each of the rank $3$ step $2$ symbols in dimensions $5$ and $6$, as well as for the free step $2$ symbol of rank $4$. Consequently, the corresponding nilpotent approximations satisfy the hypothesis of~\cite[Theorem~1.11]{JMZ21} for all these metrics.
\end{Prop}

\begin{proof}
We begin with the symbols of rank $3$. In this case we identify skew operators with vectors by $J_\eta X=\eta\times X$, using an oriented orthonormal basis. For the free symbol the parameter $\eta$ ranges over $\C^3$. Since a nonzero skew operator has a one-dimensional kernel, condition~\eqref{wc:nonample_cond}, for $Z\neq0$, is equivalent to $J_\eta Z\in\C Z$, or, equivalently, to
\[
g(Z,Z)\eta-g(Z,\eta)Z=0.
\]
For nonnull $Z$ this requires $\eta\in\C Z$, whereas for nonzero null $Z$ it requires $\eta\in Z^\perp$. Each of these loci has dimension $4$ in the $6$-dimensional space of $(Z,\eta)$, and the locus $Z=0$ has dimension $3$, which proves the claim for the free symbol. In dimension $5$, the parameters $\eta,\mu$ range over the complexification $P$ of a real two-plane, which is nondegenerate. In this case non-ampleness requires $Z\in P$ and $\eta\times Z\in\C Z$. Since $\eta\times Z\in P^\perp$, this forces $\eta\times Z=0$, and hence $\eta\in\C Z$. The resulting locus has dimension $3$ in the $5$-dimensional space of covectors, and the locus $Z=0$ has dimension $2$.

We now turn to the free symbol of rank $4$, for which we identify the dual of the center with $\mathfrak{so}(4,\C)$. A nonzero skew operator can annihilate a subspace $K$ exactly when $\dim K\leq2$, because such operators correspond to alternating forms on $\C^4/K$. Therefore $(Z,J)$ is non-ample exactly when its Krylov space $K=\spn\{J^jZ:j\geq0\}$ has dimension at most two. Such a space is always contained in a $J$-invariant two-plane $U$: if $\dim K<2$, it suffices to extend it using an eigenvector in the quotient. Consider now the incidence set of triples $(U,J,Z)$ with $JU\subset U$ and $Z\in U$. To estimate its dimension, recall that the orthogonal group has finitely many orbits on $\mathrm{Gr}(2,4)$, distinguished by the dimension $0,1,2$ of $U\cap U^\perp$ (in the maximal isotropic case there are two orbits of the connected group). On an orbit of dimension $d$, the stabilizer in the Lie algebra has dimension $6-d$, so the part of the incidence set lying over this orbit has dimension
\[
d+(6-d)+2=8.
\]
Since the image of the incidence set under the natural projection contains the non-ample set, the latter has dimension at most $8$ in the $10$-dimensional space $\C^4\times\mathfrak{so}(4,\C)$.
\end{proof}

\begin{rk}[Computer experiments]\label{wc:computations_rem}
Proposition~\ref{wc:criterion_prop} gives an effective algebraic test in terms of the maximal minors of $\cA_{\leq m}$. In practice, one can restrict these minors to a random affine line, reduce them modulo a prime, and compute their greatest common divisor. Such a computation is a convenient way to explore examples, but by itself it only provides evidence: in order to turn it into a proof, one has to record the explicit matrices, the chosen prime and line, and to give an argument excluding a bad specialization. Note also that the degree of this greatest common divisor counts the multiplicities of the common factors, so that, to obtain the degree of the reduced union of the divisorial components, one has to pass to its squarefree part. Similarly, computations with randomly sampled metrics may suggest statements about all metrics, or about generic metrics, but they cannot establish such statements. For this reason, we use computer experiments only as a guide: all cases listed in Table~\ref{weyl_comparison_table} are settled by the arguments given in this section, and none of these arguments relies on computer calculations.
\end{rk}

\subsection{Step 2 structures for which Corollary \ref{Weyl_thm} is new}\label{wc:new_subsec}
All the examples considered in this subsection are built from the octonions. Let $J_1,\ldots,J_7$ be the operators on $\R^8=\mathbb O$ given by left multiplication by the seven imaginary units of the octonions. These operators satisfy the Clifford relations $J_iJ_j+J_jJ_i=-2\delta_{ij}\,\Id$ (see~\cite{LawsonMichelsohn}), and for $1\leq r\leq7$ we use the first $r$ of them. Let $\n_r=\m_{-1}\oplus\m_{-2}$ with $\m_{-1}=\R^8$, $\m_{-2}=\R^r$ with basis $W_1,\ldots, W_r$, and
\[
[X,Y]=\sum_{i=1}^r\langle J_iX,Y\rangle\, W_i ,
\]
where $\langle\cdot,\cdot\rangle$ is the standard inner product on $\R^8$, and let $g=\langle\cdot,\cdot\rangle$. The corresponding Carnot groups are the H-type groups of Kaplan~\cite{kaplan}; $\n_7$ is the octonionic Heisenberg algebra, and the corresponding group has dimension $15$. In this case $J_\mu=\sum_i\mu_iJ_i$, where $\mu_i=\mu(W_i)$, and
\begin{equation}\label{wc:clifford}
J_\mu J_\nu+J_\nu J_\mu=-2\,(\mu,\nu)\,\Id,\qquad (\mu,\nu):=\sum_{i=1}^r\mu_i\nu_i .
\end{equation}

\begin{Prop}\label{wc:octonion_prop}
For the H-type algebras $\n_6$ and $\n_7$ with the metric $g$ above, the non-ample set is exactly the complex null cone:
\[
\cN=\{\hh=0\}=\{\lambda:\ g(Z(\lambda),Z(\lambda))=0\}.
\]
Moreover, $\dim\ker\cA(\lambda)\geq r-5$ for every $\lambda\in\cN$, so the corresponding null geodesics (if $Z(\lambda)\neq0$) have corank at least $r-5$. Consequently,
\begin{enumerate}
\item all complex normal geodesics with $\hh\neq0$ are strictly normal, so the hypothesis of~\cite[Theorem~1.11]{JMZ21} is satisfied;
\item the smooth part $\{\hh=0,\ Z\neq0\}$ of the zero level set is foliated by complex (null) normal extremals projecting to non-strictly normal geodesics, so the hypothesis of~\cite[Theorem~3.2]{JMZ21} is not satisfied;
\item no complex abnormal extremal is of minimal order (Proposition~\ref{wc:min_order_prop}(2)).
\end{enumerate}
\end{Prop}

\begin{proof}
By~\eqref{wc:clifford}, $g(J_iZ,J_jZ)=-\frac12g\big(Z,(J_iJ_j+J_jJ_i)Z\big)=\delta_{ij}\,g(Z,Z)$. If $g(Z,Z)\neq0$, the vectors $J_1Z,\ldots,J_rZ$ are therefore linearly independent, so $\mu\mapsto J_\mu Z$ is injective; hence already the first layer $\cA^1(\lambda)$ is injective and $\lambda$ is ample.

Let $g(Z,Z)=0$. Then $J_1Z,\ldots,J_rZ$ span a totally isotropic subspace of $(\C^8,g)$, whose dimension is at most $4$. Hence $K_Z:=\{\mu:\ J_\mu Z=0\}$ has dimension at least $r-4\geq2$, and $K_Z\cap\{\mu:\ (\mu,\eta)=0\}$ has dimension at least $r-5\geq1$. Let $\mu\ne0$ belong to the latter space. By~\eqref{wc:clifford}, $J_\mu J_\eta Z=-J_\eta J_\mu Z-2(\mu,\eta)Z=0$, and $J_\eta^2=-(\eta,\eta)\,\Id$, so that $\spn\{J_\eta^jZ\}_{j\geq0}=\spn\{Z,J_\eta Z\}\subset\ker J_\mu$. Thus~\eqref{wc:nonample_cond} holds, $\lambda\in\cN$, and $\dim\ker\cA(\lambda)\geq r-5$. Items (1)--(3) follow from Lemma~\ref{wc:nonstrict_lemma}, Proposition~\ref{wc:criterion_prop}, and Proposition~\ref{wc:min_order_prop}.
\end{proof}

For $Z\neq0$, the preceding argument in fact yields the following exact formula, which will be useful below:
\begin{equation}\label{wc:exact_kernel}
\ker\cA(Z,\eta)=K_Z\cap\eta^\perp,
\qquad K_Z=\{\mu:J_\mu Z=0\}.
\end{equation}
Indeed, once $J_\mu Z=0$, the second of the defining equations of the kernel takes the form $-2(\mu,\eta)Z=0$ by the Clifford relation, and, since $J_\eta^2$ is scalar, all further equations are consequences of the first two.

Using this formula and explicit covectors, we can show that the lower bound on the corank obtained above is attained at generic points of the null cone, and that for the octonionic algebras with smaller centers the non-ample set has codimension at least two.

\begin{Prop}[Explicit witnesses]\label{wc:witness_prop}
For $r=6,7$ the generic corank on the null cone is exactly $r-5$. For the octonionic restrictions $\n_r$ with $1\leq r\leq5$, the non-ample set has codimension at least two, so the hypothesis of~\cite[Theorem~1.11]{JMZ21} holds.
\end{Prop}

\begin{proof}
We use the octonion basis $e_0=1,e_1,\ldots,e_7$ with oriented multiplication triples
\[
(123),\ (145),\ (176),\ (246),\ (257),\ (347),\ (365),
\]
and we denote by $\varepsilon_j$ the standard central covectors. At $Z_0=e_0+\mathrm i e_1$ a direct computation gives
\[
J_3Z_0=\mathrm iJ_2Z_0,\qquad J_5Z_0=\mathrm iJ_4Z_0,
\qquad J_7Z_0=-\mathrm iJ_6Z_0,
\]
while the vectors $J_1Z_0,J_2Z_0,J_4Z_0,J_6Z_0$ are linearly independent. Hence, for $r=7$,
\[
K_{Z_0}=\spn\{-\mathrm i\varepsilon_2+\varepsilon_3,
-\mathrm i\varepsilon_4+\varepsilon_5,
\mathrm i\varepsilon_6+\varepsilon_7\},
\]
and for $r=6$ the last generator is absent. Taking $\eta=\varepsilon_2$ in~\eqref{wc:exact_kernel}, we obtain a covector whose corank is exactly $r-5$. Together with the lower bound in Proposition~\ref{wc:octonion_prop}, this proves the assertion about the generic corank, because the maximality of the rank is an open condition.

For $r\leq4$, the same vector $Z_0$, taken with $\eta=\varepsilon_2$ for $r\geq2$ and with $\eta=\varepsilon_1$ for $r=1$, gives a zero kernel. For $r=5$ we use instead
\[
Z_*=e_0+\mathrm i e_1+e_2+\mathrm i e_3,
\qquad \eta=\varepsilon_2.
\]
In this case $J_3Z_*=\mathrm iJ_2Z_*$, while $J_1Z_*,J_2Z_*,J_4Z_*,J_5Z_*$ are linearly independent. Thus $K_{Z_*}=\C(-\mathrm i\varepsilon_2+\varepsilon_3)$, and~\eqref{wc:exact_kernel} again gives a zero kernel. Note that all the vectors $Z$ used here are null. On the other hand, for every $r$ the Gram-matrix argument in the proof of Proposition~\ref{wc:octonion_prop} shows that $\cN\subset\{g(Z,Z)=0\}$. Since the null quadric is irreducible and, for $r\leq5$, contains the explicit ample covectors found above, it is not contained in $\cN$. Therefore $\codim\cN\geq2$.
\end{proof}

\begin{rk}[The zero energy level]\label{wc:c0_rem}
By Proposition~\ref{wc:octonion_prop}(1), the H-type metrics on $\n_6$ and $\n_7$ satisfy the hypothesis of~\cite[Theorem~1.11]{JMZ21}. Nevertheless, for these metrics one of the cases in the proof of this theorem requires an additional argument. Namely, the argument in the proof of~\cite[Theorem~3.2]{JMZ21}, which is based on the analysis of the denominators of the solution of the fundamental algebraic system, produces a hypersurface foliated by complex normal extremals projecting to non-strictly normal geodesics. In Case~2 of~\cite[\S4]{JMZ21}, this hypersurface is contained in a level set $\{\hh=c\}$ with $c\in\C$, and the contradiction is obtained in the same way as in the preceding case, by an appeal to the hypothesis on families of geodesics. This appeal does not address the case $c=0$, because the definition of a parametric family of geodesics used in the hypotheses of \cite[Theorems~1.11 and~1.12]{JMZ21} concerns only nonzero levels of the Hamiltonian. Thus the case $c=0$ requires an additional argument.

For $\n_6$ and $\n_7$, the null cone is a hypersurface of this kind: by Proposition~\ref{wc:octonion_prop}, it is the whole non-ample set, and its smooth part $\{\hh=0,\ Z\neq0\}$ is foliated by complex normal extremals projecting to non-strictly normal geodesics. Moreover, all maximal minors of $\cA_{\leq m}$ are divisible by the irreducible polynomial $\hh$, because their common zero set is its zero set. Consequently, considerations based on these minors do not exclude $\hh$ as a factor of the denominator, and an additional argument is needed to rule out this possibility. Note, however, that these examples do not establish that an actual reduced rational solution has a pole along the null cone, and they do not contradict the conclusion of~\cite[Theorem~1.11]{JMZ21}: in fact, Corollary~\ref{Weyl_thm} provides an independent proof of this conclusion in step $2$.
\end{rk}

\begin{Prop}\label{wc:products_prop}
Let $\m=\n_r\oplus\m''$ with $r\in\{6,7\}$, where $\m''$ is an arbitrary nonzero fundamental graded Lie algebra of step at most $2$, and let the metric on $\m_{-1}=\R^8\oplus\m''_{-1}$ be the orthogonal sum $g\oplus g''$ of the H-type metric $g$ on $\R^8$ and an arbitrary inner product $g''$ on $\m''_{-1}$. Then $\cN$ contains the irreducible hypersurface $\{\lambda:\ g(Z',Z')=0\}$, where $Z'$ denotes the $\R^8$-component of $Z(\lambda)$, and this hypersurface is not contained in the null cone. Consequently, neither the hypothesis of~\cite[Theorem~1.11]{JMZ21} nor the hypothesis of~\cite[Theorem~1.14]{JMZ21} is satisfied by the corresponding nilpotent approximation.
\end{Prop}

\begin{proof}
Write $\lambda=(Z',Z'',\eta',\eta'')$ according to the decomposition of $\m$; then $J_\eta=J_{\eta'}\oplus J''_{\eta''}$. If $g(Z',Z')=0$, choose $\mu'\neq0$ as in the proof of Proposition~\ref{wc:octonion_prop} (with $Z$ and $\eta$ replaced by $Z'$ and $\eta'$) and set $\mu=(\mu',0)$. Then $J_\mu J_\eta^jZ=\big(J_{\mu'}J_{\eta'}^jZ',0\big)=0$ for all $j$, so $\lambda\in\cN$. The quadric $\{g(Z',Z')=0\}$ is irreducible because $g$ has rank $8\geq3$, and it is not contained in $\{\hh=0\}$, because $\hh=\frac12\big(g(Z',Z')+g''(Z'',Z'')\big)$. By Proposition~\ref{wc:criterion_prop}, the hypothesis of~\cite[Theorem~1.11]{JMZ21} fails. As for minimal order, if $\dim\m_{-1}=8+\dim\m''_{-1}$ is even, we use Proposition~\ref{wc:min_order_prop}(2). If it is odd, take $\mu=(\mu',0)$ with $\mu'\neq0$ and $(\mu',\mu')=0$; then $J_{\mu'}^2=0$ by~\eqref{wc:clifford}, so $\operatorname{im}J_{\mu'}\subset\ker J_{\mu'}$ and $\dim\ker J_{\mu'}\geq4$. Hence $\ker\omega_\mu\supset\ker J_{\mu'}\oplus\m''_{-1}$ has dimension at least $5$, and Proposition~\ref{wc:min_order_prop}(3) applies.
\end{proof}

For example, the products $\n_6\oplus\R$ and $\n_7\oplus\R$, with $\R$ in degree $-1$, have dimensions $15$ and $16$ and horizontal rank $9$. Another family of examples is formed by the products $\n_7\oplus\h_{2s+1}$, $s\geq1$, of dimension $16+2s$. In each of these cases the null cone of the octonionic factor is a hypersurface of non-ample covectors which is not contained in the null cone of the total metric. For the corresponding left-invariant real analytic structures, the same computation shows that the hypotheses of~\cite[Theorems~1.12 and~1.15]{JMZ21} are not satisfied either.

\begin{example}\label{wc:spheres_ex}
Consider the octonionic Hopf fibration $S^7\to S^{15}\to S^8$, let $\mathcal H$ be the distribution on $S^{15}$ orthogonal to its fibers with respect to the round metric, and let $g_{S^{15}}$ be the restriction of the round metric to $\mathcal H$. This sub-Riemannian structure is invariant under the transitive action of $\mathrm{Spin}(9)$ on $S^{15}=\mathrm{Spin}(9)/\mathrm{Spin}(7)$. The isotropy group $\mathrm{Spin}(7)$ acts on $\mathcal H_p\cong\R^8$ by the spin representation, which is irreducible, and on $T_pS^{15}/\mathcal H_p\cong\R^7$ by the vector representation. This horizontal structure is fat and has H-type tangent groups; see~\cite{BaudoinCho}. In particular its Levi bracket is nonzero. Since $\Lambda^2\R^8\cong\mathfrak{spin}(7)\oplus\R^7$ as $\mathrm{Spin}(7)$-modules, the Levi bracket $\Lambda^2\mathcal H_p\to T_pS^{15}/\mathcal H_p$ is, up to a constant factor, the Clifford multiplication, and $g_{S^{15}}(p)$ is, up to a constant factor, the standard inner product. Hence the nilpotent approximation of $(S^{15},\mathcal H,g_{S^{15}})$ at every point is the octonionic H-type group $\n_7$ with its H-type metric, and Remark~\ref{wc:c0_rem} applies. Similarly, the product of $(S^{15},\mathcal H,g_{S^{15}})$ with any step $2$ sub-Riemannian structure, for example with $S^3$ endowed with the standard contact sub-Riemannian structure of the Hopf fibration $S^1\to S^3\to S^2$, has nilpotent approximations as in Proposition~\ref{wc:products_prop}, so neither~\cite[Theorem~1.11]{JMZ21} nor~\cite[Theorem~1.14]{JMZ21} applies. The same observation applies to $S^{15}\times S^1$, where the whole tangent direction of $S^1$ is included in the distribution. Note that these conclusions concern Theorems~1.11 and~1.14, whose hypotheses are formulated in terms of nilpotent approximations. Whether the analytic criteria of Theorems~1.12 and~1.15 hold for these curved products depends on their actual complex geodesics, and deciding this would require a separate analysis. On the other hand, by Corollary~\ref{Weyl_thm}, all these structures are Weyl rigid.

Moreover, every sub-Riemannian metric on $(S^{15},\mathcal H)$ is projectively rigid. Indeed, $\n_7$ does not admit a product structure. To see this, suppose that $\R^8=A\oplus B$ with $[A,B]=0$ and $A,B\neq0$. Then for every $0\neq X\in A$ the space $B$ is orthogonal to $J_1X,\dots,J_7X$, which span the orthogonal complement of $X$, so that $B\subset\R X$. But $\R X\subset A$ and $A\cap B=0$, hence $B=0$, a contradiction. Therefore Corollary~\ref{wc:projective_rigidity} applies. Note that this conclusion does not follow from~\cite[Corollary~1.16]{JMZ21}, since the abnormal extremals of $\n_7$ are not of minimal order (Proposition~\ref{wc:octonion_prop}(3)).
\end{example}

\medskip\noindent\textbf{Indecomposable central amalgams.}\label{wc:amalgam_subsec}
The symbols of the product examples above are, by construction, decomposable. We now show that the same failure of the hypotheses occurs for indecomposable symbols as well: in every odd dimension at least $17$ there is a nonproduct analogue of these examples, obtained by identifying the center of a Heisenberg algebra with a line in the center of the octonionic H-type algebra.

\begin{Prop}\label{wc:amalgam_prop}
Let $s\geq1$, let $(U,\Omega)$ be a real symplectic vector space of dimension $2s$, and define the step $2$ algebra $\mathfrak c_s$ by
\[
(\mathfrak c_s)_{-1}=\mathbb O\oplus U,\qquad
(\mathfrak c_s)_{-2}=\R^7=\spn\{W_1,\ldots,W_7\}.
\]
On $\mathbb O$ use the bracket of $\n_7$; set $[\mathbb O,U]=0$ and
\[
[Y,Y']=\Omega(Y,Y')W_1,\qquad Y,Y'\in U.
\]
Equip the horizontal space with the orthogonal sum of the octonionic H-type metric and any positive definite inner product on $U$. Then:
\begin{enumerate}
\item $\dim\mathfrak c_s=15+2s$, and its non-ample set contains the irreducible hypersurface
\[
\Sigma=\{(X,Y,\eta):g(X,X)=0\},
\]
which is not contained in the total null cone. Thus the hypothesis of~\cite[Theorem~1.11]{JMZ21} fails.
\item The complexified graded algebra $\mathfrak c_s$ is indecomposable. Consequently every sub-Riemannian metric on a connected step $2$ distribution with symbol $\mathfrak c_s$ on an open dense set is projectively rigid.
\item The minimal-order hypothesis of~\cite[Theorem~1.14]{JMZ21} fails. For the corresponding left-invariant analytic structure, the hypotheses of Theorems~1.12 and~1.15 of that paper also fail.
\end{enumerate}
\end{Prop}

\begin{proof}
Let $B$ be the invertible skew operator on $U$ defined by $\Omega(Y,Y')=g_U(BY,Y')$. Then
\[
J_\mu=J_\mu^{\mathbb O}\oplus\mu_1B.
\]
For a null octonionic vector $X$, the space $K_X=\{\mu:J_\mu^{\mathbb O}X=0\}$ has dimension at least three, since, as in the proof of Proposition~\ref{wc:octonion_prop}, the images of a null vector under the Clifford operators span a totally isotropic subspace. Consequently, the space
\[
K_X\cap\eta^\perp\cap\{\mu_1=0\}
\]
has dimension at least one, and we can choose a nonzero $\mu$ in it. The operator $J_\mu$ vanishes on $U$, while on $\mathbb O$ the Clifford identity gives, exactly as before, $J_\mu^{\mathbb O}(J_\eta^{\mathbb O})^jX=0$ for every $j$. Therefore $J_\mu J_\eta^j(X,Y)=0$ for every $j$, which proves that $\Sigma\subset\cN$. The quadric defining this hypersurface has rank eight in $X$ and is therefore irreducible. Moreover, on $\Sigma$ the total Hamiltonian equals $g_U(Y,Y)/2$, which does not vanish identically, so the hypersurface is not contained in the total null cone. Since $\cN$ is a proper algebraic subset, $\Sigma$ is a divisorial component of it, and Proposition~\ref{wc:criterion_prop} proves (1).

We now turn to (2). The Pfaffian of the Levi pencil is, up to a nonzero constant,
\begin{equation}\label{wc:amalgam_pfaffian}
P(\mu)=Q(\mu)^2\mu_1^s,
\qquad Q(\mu)=\mu_1^2+\cdots+\mu_7^2.
\end{equation}
Indeed, the Clifford relation gives $\det J_\mu^{\mathbb O}=Q(\mu)^4$, so that the Pfaffian of the octonionic block is a nonzero constant times $Q^2$, while the Pfaffian of the block on $U$ is a nonzero constant times $\mu_1^s$. In particular, the pencil contains a nondegenerate form, and therefore the horizontal center is zero.

Assume now that $\mathfrak c_s$ splits into two nonzero graded ideals, with horizontal parts $H_A,H_B$ and central parts $W_A,W_B$. Then both central parts must be nonzero, since a factor with zero central part would be contained in the horizontal center, which we have just shown to be zero. Moreover, a generic Levi form is block diagonal with respect to this splitting and nondegenerate, so both horizontal dimensions are even, and its Pfaffian factors as $P_A(\mu_A)P_B(\mu_B)$, where the two factors depend on disjoint sets of central variables. The quadratic polynomial $Q$ in~\eqref{wc:amalgam_pfaffian} is irreducible over $\C$, so it must divide one of these factors. It would then be independent of the complementary central variables, which span a nonzero space. This is impossible, because $Q$ is a nondegenerate quadratic form of rank seven. Thus no such splitting exists. Since indecomposability is a property of the symbol alone, Corollary~\ref{wc:projective_rigidity} now yields the projective rigidity of every sub-Riemannian metric on such a distribution, and not only of the metric used in (1).

Finally, the horizontal rank $8+2s$ is even and the center has dimension seven, so the failure of the minimal-order hypothesis stated in (3) follows from Proposition~\ref{wc:min_order_prop}. For the left-invariant structure, the same computations describe the actual normal and abnormal extremals, and therefore the analogous analytic hypotheses fail as well.
\end{proof}

\begin{rk}[The refinement of minimal order]\label{wc:weak_order_rem}
The refinement of the condition of minimal order given in~\cite[\S5.5]{JMZ21}, which is still sufficient for the conclusion of the corresponding theorem, does not apply to the family $\mathfrak c_s$ either. To see this, consider the open part of the complex degeneracy locus
\[
Q(\mu)=0,\qquad \mu_1\neq0.
\]
At its points $\ker J_\mu^{\mathbb O}$ has dimension four: choosing $\nu$ with $(\mu,\nu)\neq0$ in the Clifford relation, we obtain $\ker J_\mu^{\mathbb O}=\operatorname{im}J_\mu^{\mathbb O}$. On the other hand, the block on $U$ is invertible there. As a consequence, on the smooth conical stratum $G\times\{Q=0,\ \mu_1\neq0\}$ the kernel of the restricted symplectic form has dimension five, consisting of four horizontal directions in $\ker\omega_\mu$ and one central direction normal to the quadric in the dual central space. To see this directly, one uses the formula for the canonical symplectic form preceding~\eqref{wc:ker_sigma}: testing against central base directions gives $\xi=0$, testing against horizontal directions gives $X_{-1}\in\ker\omega_\mu$, and testing against tangent covectors gives $X_{-2}\in(T_\mu\{Q=0\})^\perp$.

Now, any finite regular stratification must retain an open part of this smooth component in a stratum of full dimension, and the characteristic kernel of such a stratum cannot have the dimension required in~\cite[\S5.5]{JMZ21}; nor can a stratum of lower dimension absorb this open part. In the projectivized description, the same obstruction is already provided by the four independent horizontal characteristic directions. Thus the refinement based on the weak minimal order condition does not make the criterion applicable to this family.
\end{rk}

\begin{example}[Dimensions and compact examples]\label{wc:compact_amalgam_ex}
For $s=1$, the algebra $\mathfrak c_1$ is the quotient of $\n_7\oplus\h_3$ obtained by identifying the center of the Heisenberg algebra with $\R W_1$; its horizontal rank is $10$, its corank is $7$, and its dimension is $17$. For $s=2,3,\ldots$ we obtain ranks $12,14,\ldots$ and dimensions $19,21,\ldots$, respectively. Note that these algebras are not of H-type, since for a nonzero real $\mu$ with $\mu_1=0$ the operator $J_\mu$ has a nonzero kernel.

These examples also admit compact realizations. To construct them, we choose the integral octonion multiplication table used in the proof of Proposition~\ref{wc:witness_prop} and an integral standard symplectic form on $U$. In exponential coordinates the group law is then
\[
(v,z)(v',z')=(v+v',z+z'+\tfrac12[v,v']).
\]
The subgroup $\Gamma=2\mathbb Z^{8+2s}\times\mathbb Z^7$ is discrete and cocompact; it is closed under multiplication because the brackets are integral and its horizontal coordinates contain the factor two. Being left-invariant, the distribution and the indicated metric descend to $\Gamma\backslash G$. Locally the resulting structures coincide with those on the group, so they fail the criteria discussed above in the same way, and every metric on the descended distribution is projectively rigid by Corollary~\ref{wc:projective_rigidity}. Taking further direct products, one obtains additional Weyl rigid examples which still possess a non-ample hypersurface not contained in the null cone; their symbols, however, are decomposable.
\end{example}

\begin{rk}[Higher step]\label{wc:higher_step_rem}
The relation between Corollary~\ref{Weyl_thm} and~\cite{JMZ21} is different in higher step. Once the consistency of the principal fundamental algebraic system is established, the proof of Proposition~\ref{eigenconst} in Appendix~\ref{app:eigen} uses only its first two principal layers, and in step $2$ these two layers already force $\dd\alpha|_D=0$ in the conformal case. In step $3$ this is no longer the case. For example, let $\m$ be the Engel algebra, $[X_1,X_2]=X_3$, $[X_1,X_3]=X_4$, with $X_1,X_2$ orthonormal, and let $a_i=X_i(\alpha)(q)$ and $\varsigma=a_1u_1+a_2u_2$. A direct computation shows that, in the notation of~\cite[Lemma~2.3]{JMZ21}, the first two layers of the principal part of the fundamental algebraic system admit, for all values of $a_1,a_2$, the solution
\[
\Psi_3=a_1u_2-a_2u_1,\qquad \Psi_4=\frac{2\varsigma u_3}{u_1},
\]
whereas the third layer gives, on clearing the denominator,
\[
3u_1^2(a_1u_1+a_2u_2)u_4+2a_2(u_1^2+u_2^2)u_3^2=0.
\]
The coefficients of this polynomial identity force $a_1=a_2=0$. Thus no argument based on the first two layers alone can give $\dd\alpha|_D=0$ for Engel distributions, while Engel distributions are covered by~\cite[Corollary~1.16]{JMZ21}. This is why Proposition~\ref{eigenconst} and Corollary~\ref{Weyl_thm} are formulated for step $2$ distributions.
\end{rk}

\begin{table}[htbp]
\centering
\small
\begin{tabularx}{\textwidth}{@{}Xcc@{}}
\toprule
Symbol and metric & \cite[Thm.~1.14]{JMZ21} & \cite[Thm.~1.11]{JMZ21} \\
\midrule
Contact or quasi-contact, any metric & yes & yes \\
$\h_{2k+1}\oplus\R^a$, $a\geq2$, any metric & no & yes \\
Rank $3$, dimension $5$ or $6$, any metric & yes & yes \\
Sum of at least two contact symbols, any metric & no & yes \\
Free step $2$ of rank $4$, any metric & no & yes \\
Octonionic restriction $\n_r$, $2\leq r\leq5$, H-type metric & no & yes \\
$\n_6$, $\n_7$, H-type metric & no & yes$^{\dagger}$ \\
$\n_6\oplus\m''$, $\n_7\oplus\m''$, indicated product metric & no & no \\
Central amalgam $\mathfrak c_s$, $s\geq1$, indicated metric & no & no \\
\bottomrule
\end{tabularx}
\caption{Whether the \emph{hypotheses} of the theorems indicated in the header hold for the nilpotent approximations with the given symbols and metrics. The table summarizes the results of this section; in the product row, $\m''$ is any nonzero fundamental graded algebra of step at most $2$. $^{\dagger}$The hypothesis, which concerns only the nonzero energy levels, holds; however, the case of zero energy in the proof of this theorem requires an additional argument (Remark~\ref{wc:c0_rem}). Corollary~\ref{Weyl_thm} applies in all cases.}
\label{weyl_comparison_table}
\end{table}

Summarizing, for step $2$ distributions Corollary~\ref{Weyl_thm} removes the additional hypotheses imposed in~\cite{JMZ21}. The products and central amalgams constructed above show that this matters: they give explicit metrics for which these hypotheses are not satisfied. Moreover, the central amalgams have indecomposable symbols, which yields projective rigidity for every metric on distributions with these symbols. The proof is also of a different nature: it establishes the consistency of the principal fundamental algebraic system and then uses only the first two layers of this system, without complexification and without any analysis of denominators. In particular, the zero energy level, which requires an additional argument in the approach of that paper for the octonionic H-type metrics, plays no special role in it. Finally, the same argument gives Proposition~\ref{eigenconst} for arbitrary, not necessarily conformal, pairs of projectively equivalent metrics.

\appendix

\section{The Operator Form of the fundamental algebraic system}
\label{opp:fas}

Fix a particular $\Gamma_i$ and let
\[
\widehat W_i=\operatorname{span}\{X_e:e\in E_i\}\subset TM.
\]
Recall that, by Theorem \ref{Tanaka_decomp_prop},
\[
\mathfrak m^i_{-2}=W_i\cong D_i^2/(D_i^2\cap D),
\]
where $W_i\subset TM/D$ is the image of the Levi bracket of $D_i$.
Since the classes $\{X_e\bmod D:e\in E_i\}$ form the basis of $\mathfrak m^i_{-2}$ determined by the adapted frame, the canonical projection modulo $D$ restricts to an isomorphism
\[
\varpi_i:\widehat W_i\longrightarrow W_i=\mathfrak m^i_{-2},
\qquad
\varpi_i(X_e)=X_e\bmod D.
\]
Its transpose identifies the dual spaces: to a functional $\eta\in W_i^*$ there corresponds the functional
\[
\widehat\eta:=\eta\circ\varpi_i\in\widehat W_i^*,
\qquad\text{so that}\qquad
\widehat\eta\bigl(\varpi_i^{-1}([X,Y]\bmod D)\bigr)=\eta([X,Y]\bmod D)
\]
for all sections $X,Y$ of $D_i$. Below we work with functionals $\eta\in W_i^*$, so that all pairings with brackets of sections of $D_i$ are taken in $\mathfrak m^i_{-2}=D_i^2/(D_i^2\cap D)$; the space $\widehat W_i$ is used only to relate these functionals to the frame and to the fiber coordinates.
For $X\in D_i$, define
\[
\operatorname{ad}_i X:D_i\longrightarrow W_i,
\qquad
\operatorname{ad}_i X(V):=[X,V]\bmod D.
\]
Its adjoint $(\operatorname{ad}_i X)^*:W_i^*\to D_i^*$ is defined by
\begin{equation}
\label{Ldual}
\bigl((\operatorname{ad}_i X)^*\eta\bigr)(V)
:=
\eta([X,V]\bmod D),
\qquad
V\in D_i,\quad \eta\in W_i^*.
\end{equation}
Further, given $\eta\in W_i^*$, define $J^i_\eta\in \mathrm{End}(D_i)$ by
\begin{equation}
\label{J_eta}
g_1(J^i_\eta X,V)
=
\eta([X,V]\bmod D),
\qquad X,V\in D_i,
\end{equation}
or equivalently $\iota_{J^i_\eta X}g_1=(\operatorname{ad}_i X)^*\eta$, where $\iota$ denotes the interior product.

Now, let $\{\varepsilon^e:e\in E_i\}$ be the basis of $W_i^*$ dual
to the basis $\{X_e(q)\bmod D(q):e\in E_i\}$ of $W_i(q)$, and let $\{\theta^v: v \in V_i\}$ be the basis of $D^*_i$ dual to $\{X_v: v \in V_i\}$. Then the forms $\theta^v$, $v\in V_i$, and $\widehat\varepsilon^{\,e}=\varepsilon^e\circ\varpi_i$, $e\in E_i$, constitute the coframe of $D_i\oplus\widehat W_i$ dual to the frame $\{X_s\}_{s\in V_i\cup E_i}$.
Recall that, by Corollary \ref{block_diag}, the matrices
$\mathfrak{A}^p$ are block diagonal, with one diagonal block for each component $\Gamma_i$ of the graph $\Gamma$; denote the block corresponding to $\Gamma_i$ by $\mathfrak{A}^p_i$.
Given $\lambda\in T^*M\setminus D^\perp$, denote by $\mathcal{A}^p_i(\lambda)$ the linear map from $W^*_i$ to $D^*_i$ whose matrix with respect to the bases $\{\varepsilon^e:e\in E_i\}$ of $W_i^*$ and $\{\theta^v: v \in V_i\}$ of $D^*_i$ is $\mathfrak{A}^p_i(\lambda)$.

Further, given $\lambda \in T^*M$, let $Z(\lambda)\in D\bigl(\pi(\lambda)\bigr)$ be defined by
$$g_1\bigl(Z(\lambda), \cdot\bigr)=\lambda|_D(\cdot)\quad \Longleftrightarrow\quad \iota_{Z(\lambda)} g_1=\lambda|_D.$$
Let $Z_i(\lambda)\in D_i$ be the $D_i$-component of $Z(\lambda)$ with respect to the eigenspace decomposition of the transition operator $S$.
Finally, let $\eta_i(\lambda)\in W_i^*$ be the functional such that $\widehat{\eta_i(\lambda)}$ is the restriction of $\lambda$ to $\widehat W_i\bigl(\pi(\lambda)\bigr)$; in other words, $\widehat{\eta_i(\lambda)}$ is the $\widehat W_i^*$-component of $\lambda$ with respect to the dual of the pointwise splitting
\[
TM=D\oplus\bigoplus_{i=1}^k \widehat W_i
=\bigoplus_{i=1}^kD_i\oplus\bigoplus_{i=1}^k\widehat W_i
\]
determined by the adapted frame.
By construction,
\begin{align}
~&Z_i(\lambda)=\sum_{x\in V_i}u_xX_x\bigl(\pi(\lambda)\bigr), \label{z_i}\\
~&\eta_i(\lambda)=\sum_{e\in E_i}u_e\varepsilon^e\bigl(\pi(\lambda)\bigr).\label{eta_i}
\end{align}

\begin{rk}
\label{biject_rem}
Note that the map from $T_q^* M$ to $D(q)\times\displaystyle{\bigoplus_{i=1}^k W_i^*(q)}$ defined by $\lambda \to \bigl(Z(\lambda),\sum_{i=1}^k\eta_i(\lambda)\bigr)$ is a bijection.
\end{rk}
\begin{Lemma}
\label{A^_1 and 2_i_lemma}
The following identities hold:
\begin{align}
~& \mathcal{A}^1_i(\lambda)=
\bigl(\operatorname{ad}_i Z_i(\lambda)\bigr)^*\label{A1_interp}\\
~&\mathcal{A}^2_i(\lambda)=
\Bigl(\operatorname{ad}_i\bigl(J^i_{\eta_i(\lambda)}Z_i(\lambda)\bigr)\Bigr)^*.
\label{A2_interp}
\end{align}
\end{Lemma}
\begin{proof}
First, using \eqref{qjk} and \eqref{z_i}, we can write the $(v,e)$-entry $\mathfrak{A}^1_i(\lambda)_{ve}$ of $\mathfrak{A}^1_i(\lambda)$, where $v\in V_i$ and $e\in E_i$, by the following chain of equalities (to improve readability, and without risk of ambiguity, we omit $\lambda$ in the formulas below):
\begin{equation}
\label{A1_proof}
(\mathfrak{A}^1_i)_{v,e}=
q_{ve}\stackrel{\eqref{qjk}}{=}\sum\limits_{x\in V_i}c^e_{xv}u_x\stackrel{\eqref{z_i}}{=}\varepsilon^e\bigl([Z_i,X_v]\bmod D\bigr)\stackrel{\eqref{Ldual}}{=}\bigl((\operatorname{ad}_i Z_i)^*
\varepsilon^e
\bigr)(X_v),
\end{equation}
which implies \eqref{A1_interp}.

Second, the $(v,e)$-entry $\mathfrak{A}^2_i(\lambda)_{ve}$ of $\mathfrak{A}^2_i(\lambda)$ satisfies the following:
\begin{align}
    (\mathfrak{A}^2_i)_{v,e}
    &\stackrel{\eqref{elem.A.frak}}{=} (\vec h_1(q_{ve}))_{(1)}
    \stackrel{\eqref{principal_delta}}{=} \sum_{\substack{x,y \in V_i \\ s \in E_i}} c^e_{xv}c^{s}_{yx} u_yu_s
    \stackrel{\eqref{A1_proof}}{=} \sum_{\substack{x\in V_i \\ s \in E_i}} c^e_{xv} u_s \varepsilon^s\bigl([Z_i, X_x]\bmod D\bigr) \notag \\
    &\stackrel{\eqref{eta_i}}{=} \eta_i\Bigl(\bigl[Z_i, \sum_{x\in V_i} c^e_{xv} X_x\bigr]\bmod D\Bigr)
    \stackrel{\eqref{J_eta}}{=} g_1\Bigl(J^i_{\eta_i} Z_i, \sum_{x\in V_i} c^e_{xv} X_x\Bigr) \label{cxveXx}
\end{align}
Now, using the skew-symmetry of the endomorphisms $J_\eta$ and the orthonormality of the local basis $\{X_x\}_{x\in V_i}$ with respect to the inner product $g_1$, we have
\begin{align}
  ~&\sum\limits_{x\in V_i}c^e_{xv}X_x=\sum\limits_{x\in V_i}\varepsilon^e\bigl([X_x, X_v]\bmod D\bigr)X_x\stackrel{\eqref{J_eta}}{=}\sum\limits_{x\in V_i}g_1(J_{\varepsilon^e} X_x, X_v)X_x=\\ ~&-\sum\limits_{x\in V_i}g_1(J_{\varepsilon^e} X_v, X_x)X_x=-J_{\varepsilon^e} X_v.
\end{align}
Substituting this into \eqref{cxveXx}, we get
\begin{align}
(\mathfrak{A}^2_i)_{v,e}=g_1\bigl(J^i_{\eta_i} Z_i, -J^i_{\varepsilon^e} X_v\bigr)\stackrel{\eqref{J_eta}}{=}\varepsilon^e\bigl([J^i_{\eta_i} Z_i,X_v]\bmod D\bigr),
\end{align}
which completes the proof of \eqref{A2_interp}.
\end{proof}

More generally, the same argument gives the operator form of all the layers of the coefficient matrix:
\begin{equation}\label{all_layer_operator}
\mathcal A_i^p=(\ad_i((J^i_{\eta_i})^{p-1}Z_i))^*,\qquad p\ge1.
\end{equation}
Indeed, by \eqref{principal_delta}, we have $\delta Z_i=J^i_{\eta_i}Z_i$ and $\delta\eta_i=0$. Since each layer of the coefficient matrix of the principal part is obtained from the previous one by applying $\delta$, the formula for layer $p$ implies the formula for layer $p+1$, and the required identity follows by induction. Note that this argument concerns the coefficient matrix only; in particular, it does not use any of the statements on the vanishing of the last column of the principal part.

Finally, for every $\lambda\in T^*M$, define an element $\beta^{p}_i(\lambda)$ of $D_i^*\bigl(\pi(\lambda)\bigr)$ by
\begin{equation}
\label{defbeta}
\beta^{p}_i(\lambda) :=\sum\limits_{v \in V_i}\tilde {\mathfrak{b}}^p_v (\lambda)\theta^v.
\end{equation}

\section{The formula for \texorpdfstring{$R_w$}{Rj} in the projective case}
\label{app:rj}

For $w\in V_i$ set
\begin{equation}
\label{projective_residual}
\mathcal E_w:=\frac12\left(
u_w\sum_{v\in V_i}X_v\bigl(\alpha(w)^2\bigr)u_v
-X_w\bigl(\alpha(w)^2\bigr)\sum_{v\in V_i}u_v^2
\right).
\end{equation}
\renewcommand{\theequation}{\thesection.\arabic{equation}}

\begin{Lemma}
\label{rjlem}
For $w\in V_i$, one has
\begin{equation}
\label{rj_final}
R_w=
\sum_{\substack{s,v\in V\\s\neq v}}
\bigl(\alpha(w)^2-\alpha(s)^2\bigr)c^s_{vw}u_vu_s
+\alpha(w)^2
\sum_{s\in E}\sum_{v\in V}
c^s_{vw}u_vu_s
+\mathcal E_w .
\end{equation}
\end{Lemma}

\begin{proof}
We will simplify the expression for $R_w$ from \eqref{rj_unreduced}.

By \eqref{Hamiltonian_vf_action} and Item (1) of Proposition \ref{first divi},
\begin{equation}
\label{rj_structure_split}
\begin{split}
&\alpha(w)^2\vec h_1(u_w)
-\sum_{v,s\in V}\alpha(s)^2c^s_{vw}u_vu_s\\
={}&
\sum_{\substack{s,v\in V\\s\neq v}}
\bigl(\alpha(w)^2-\alpha(s)^2\bigr)c^s_{vw}u_vu_s
+\alpha(w)^2
\sum_{s\in E}\sum_{v\in V}
c^s_{vw}u_vu_s\\
&+\sum_{v\in V}\bigl(\alpha(w)^2-\alpha(v)^2\bigr)c^v_{vw}u_v^2\\
\stackrel{\mathrm{Prop.}\,\ref{first divi}\,(1)}{=}{}&
\sum_{\substack{s,v\in V\\s\neq v}}
\bigl(\alpha(w)^2-\alpha(s)^2\bigr)c^s_{vw}u_vu_s
+\alpha(w)^2
\sum_{s\in E}\sum_{v\in V}
c^s_{vw}u_vu_s\\
&-\frac12\alpha(w)^2
\sum_{v\in V}
X_w\left(\frac{\alpha(v)^2}{\alpha(w)^2}\right)u_v^2 .
\end{split}
\end{equation}

Note that by \cite[Lemma 4.4]{jmz2019}, the first
divisibility condition gives
\begin{equation}
\label{first_div_a}
\cfrac{\vec h_1(a^2)}{a^2}=\sum_{v\in V}
\frac{X_v(\alpha(v)^2)}{\alpha(v)^2}u_v.
\end{equation}
Thus it remains to analyze the sum of the first, third (after plugging in \eqref{first_div_a}), and fourth terms of \eqref{rj_unreduced}, along with the final terms of the expression in \eqref{rj_structure_split}, namely
\begin{equation}
\label{rj_derivative_part}
\begin{split}
&\sum_{v\in V}
X_v\bigl(\alpha(w)^2\bigr)u_vu_w
-\frac12\alpha(w)^2u_w
\sum_{v\in V}
\frac{X_v(\alpha(v)^2)}{\alpha(v)^2}u_v\\
&\qquad
-\frac12\sum_{v\in V}
X_w\bigl(\alpha(v)^2\bigr)u_v^2
-\frac12\alpha(w)^2
\sum_{v\in V}
X_w\left(\frac{\alpha(v)^2}{\alpha(w)^2}\right)u_v^2 .
\end{split}
\end{equation}

Assume that \(w\in V_i\). If \(v\notin V_i\), then by Item (2) of
Proposition \ref{first divi}
$$\alpha(v)^2X_v\bigl(\alpha(w)^2\bigr)-\frac12
\alpha(w)^2X_v\bigl(\alpha(v)^2\bigr)=0,$$
which implies that the coefficient of $u_v u_w$ in \eqref{rj_derivative_part} is equal to zero.
Similarly, the coefficient of the square of this external vertex variable, which comes from the last two terms of the same expression, vanishes as well, since
\[
-\tfrac12X_w(\alpha(v)^2)-\tfrac12\alpha(w)^2X_w\!\left(\frac{\alpha(v)^2}{\alpha(w)^2}\right)
=-X_w(\alpha(v)^2)+\frac{\alpha(v)^2}{2\alpha(w)^2}X_w(\alpha(w)^2)=0
\]
by the same identity with the roles of the two indices interchanged, i.e., with the derivative taken in the direction $X_w$.
Now, assume that $v\in V_i$. Then $\alpha(v)=\alpha(w)$,
and the contribution of the index \(v\) to the sum (over $v\in V$) in
\eqref{rj_derivative_part} is equal to
\begin{equation}
\label{extra1}
\frac12X_v\bigl(\alpha(w)^2\bigr)u_vu_w
-\frac12X_w\bigl(\alpha(w)^2\bigr)u_v^2,
\end{equation}
because the contribution of the fourth term in \eqref{rj_derivative_part} is zero, the sum of the contributions of the first and second terms of \eqref{rj_derivative_part} simplifies to the first term of \eqref{extra1}, and the third term contributes its second term. Note that \eqref{extra1} vanishes if \(v=w\).
Consequently, combining \eqref{rj_structure_split} with \eqref{rj_derivative_part} and using \eqref{extra1}, summed over all the indices of the component, together with \eqref{projective_residual}, we get
\eqref{rj_final}.
\end{proof}

\begin{rk}
\label{rjrem}
By \eqref{Hamiltonian_vf_action} (equivalently, by \eqref{rj_structure_split}), the sum of the first two terms of the right-hand side of \eqref{rj_final} is exactly the right-hand side of \eqref{RJ}. Hence Lemma \ref{rjlem} states that the expression for $R_j$ coming from \cite{jmz2019} differs from \eqref{RJ} by the residual term $\mathcal E_j$, and formula \eqref{RJ} will be established once we prove that $\mathcal E_j=0$; this is done in Lemma \ref{vanishlem} below. Without using this vanishing, one can repeat the derivation of \eqref{tb1} in the proof of Lemma \ref{lin_quad_lem}, taking the residual term into account; in this way we obtain that the quadratic vertex-homogeneous part of the first layer takes the form
\begin{equation}
\label{tb1_full}
\tilde{\mathfrak{b}}^{1}_w=\sum_{{\begin{array}{c} l,v \in V\\\mathcal C(l)\neq \mathcal C(w),\ v\ne l \end{array}}}(\alpha(w)^2-\alpha(l)^2)c_{vw}^lu_vu_l+\mathcal E_w ,
\qquad w\in V_i.
\end{equation}
\end{rk}

\section{The correction term in the column operation for the projective recursion}
\label{app:column}

In this appendix, we compute the term of the exact identity \eqref{exact_column_rec} that involves the derivatives of the coefficients of the column operation; in what follows, we call it the coefficient-derivative term. Note that this calculation uses only the first divisibility condition. The vanishing of the coefficient-derivative term is then deduced from Proposition \ref{eigenconst}, which is proved independently in the next appendix.

\begin{Lemma}\label{columnlem}
For every $e\in E_i$, the following identity holds:
\begin{equation}\label{column_correction_reduced}
\vec h_1(\alpha(e)^2)-\frac{\vec h_1(a)}a\alpha(e)^2
=\frac12\sum_{y\in V_i}X_y(\alpha_i^2)u_y.
\end{equation}
Consequently, the coefficient-derivative term in the recursive formula for the layers of the last column obtained by the column operation is equal to
\begin{equation}\label{omitted_contribution_full}
\frac12\sum_i\sum_{e\in E_i}\sum_{y\in V_i}
X_y(\alpha_i^2)u_yu_e a^p_{ve}.
\end{equation}
In particular, without using Proposition \ref{eigenconst}, the quadratic vertex-homogeneous parts in the $i$th block satisfy the following recursive formula, in which the contribution of this term is the last summand:
\begin{equation}\label{tbs_frak_full}
\begin{split}
\tilde{\mathfrak b}^{p+1}_v={}&\delta\tilde{\mathfrak b}^p_v+
\sum_{\substack{e\in E_i\\f\in E_t,\ t\ne i}}
(\alpha_i^2-\alpha_t^2)\mathfrak A^p_{ve}q_{ef}u_f\\
&+\frac12\left(\sum_{y\in V_i}X_y(\alpha_i^2)u_y\right)
\left(\sum_{e\in E_i}u_e\mathfrak A^p_{ve}\right),\qquad v\in V_i.
\end{split}
\end{equation}
\end{Lemma}
\begin{proof}
First, note that the first layer of the last column obtained by the column operation, computed directly from the unreduced formula for the first layer of the fundamental algebraic system, is quadratic in the horizontal variables because of the splitting of the symbol. Moreover, the exact recursive formula for this column, obtained directly by the product rule, contains no terms of vertex degree less than two. Hence, by induction, every layer of this column has vertex degree at least two, and this lower bound is established before the constancy of the eigenvalues is used. Therefore, taking the quadratic vertex-homogeneous components in the argument below is legitimate independently of Proposition \ref{eigenconst}.

Let us compute the coefficient-derivative term. When the Hamiltonian vector field is applied to the column operation, the full product rule has to be used, i.e., all three factors of each summand $\alpha(e)^2u_ea^p_{ve}$ are differentiated. The terms in which the first factor is differentiated combine with the part of $-(\vec h_1a/a)b^p_v$ coming from the column operation to give
\[
\sum_e\left(\vec h_1(\alpha(e)^2)-\frac{\vec h_1a}a\alpha(e)^2\right)u_ea^p_{ve}.
\]
By \eqref{first_div_a}, the coefficient of $u_y$ in the expression in parentheses is equal to
\[
X_y(\alpha_i^2)-\frac{\alpha_i^2}{2\alpha(y)^2}X_y(\alpha(y)^2).
\]
For $y\notin V_i$ this coefficient is zero by item (2) of Proposition \ref{first divi}, while for $y\in V_i$ the two terms combine, since the eigenvalues coincide, into $\tfrac12X_y(\alpha_i^2)$. This proves \eqref{column_correction_reduced}, and substituting this identity into the first sum displayed in this proof, we obtain \eqref{omitted_contribution_full}. Further, taking the quadratic vertex-homogeneous components and using the block-diagonal structure of $\mathfrak A$, we get \eqref{tbs_frak_full}. The remaining terms of the recursion have exactly the vertex degrees that can be read off from \eqref{exact_column_rec}. In particular, it is the term $-c\tilde b^p$, and not $-cb^p$, that has vertex degree at least three and therefore does not contribute to the quadratic components.
\end{proof}

\section{Proof of Proposition \ref{eigenconst}}
\label{app:eigen}

Note that the proof given below uses only the step $2$ assumption, the stability of the base point, the first divisibility condition, and the results of \cite{jmz2019} on the decomposition of the Tanaka symbol and on ample covectors. In particular, it uses neither the constancy of the characteristic tuples nor quasi-normal frames, and it does not rely on any of the vanishing statements proved in Sections \ref{principal_part_sec}--\ref{other_sec}. The proof consists of four steps. First, we rewrite the fundamental algebraic system in terms of shifted unknowns, keeping track of all the derivatives of the eigenvalues. Second, we introduce a weighted grading of the fiber variables and show that the top-weight part of the resulting system has a unique rational solution. Third, we compute the internal parts of the first two layers of the right-hand side of this top-weight system. Finally, we contract the second layer with a suitable vector, which yields the required vanishing.

{\bf Step 1: the shifted system.} Recall that the decomposition of the Tanaka symbol and the local constancy of the ranks of its central summands in Theorem \ref{Tanaka_decomp_prop} were obtained directly from \cite[Theorem 7.1]{jmz2019}, independently of Proposition \ref{eigenconst}. Consequently, we can choose a smooth weakly adapted frame whose central classes form bases of the summands $W_i$; note that this choice requires no assumption on the characteristic tuples. Put $\lambda_i=\alpha_i^2$ and $\lambda_s=\lambda_i$ for $s\in V_i\cup E_i$, and set
\[
c=\frac{\vec h_1(a)}a=\frac12\sum_{v\in V}\frac{X_v\lambda_v}{\lambda_v}u_v,
\qquad \psi_e=a\Phi_e-\lambda_eu_e\quad(e\in E).
\]
For every $s\in V\cup E$ define the polynomial
\begin{equation}\label{weyl_raw_tau}
\tau_s=\sum_{t\in V\cup E}(\lambda_s-\lambda_t)q_{st}u_t
+(\vec h_1\lambda_s-c\lambda_s)u_s
-\frac12\sum_{v\in V}X_s\lambda_v\,u_v^2.
\end{equation}
Using the unreduced expression \eqref{rj_unreduced}, we can write the first layer of the fundamental algebraic system in the form $A^1\psi=\tau_V$. Differentiating this identity along the Hamiltonian vector field and using the vertical orbital equations, or, equivalently, substituting directly into \eqref{term.of.b.rec.form}, we obtain
\begin{equation}\label{weyl_raw_rec}
A^p\psi=d^p,\qquad d^1=\tau_V,\qquad
d^{p+1}=\vec h_1d^p-cd^p+A^p\tau_E.
\end{equation}
Here $d^p=b^p-\sum_e\lambda_eu_eA^p_e$. Note that the derivatives of the eigenvalues enter all these formulas explicitly; none of them is assumed to vanish. Moreover, since $q_{ve}=0$ unless $v,e$ belong to the same component, and in the latter case the corresponding difference of eigenvalues vanishes, $\tau_V$ depends only on horizontal variables and is quadratic in them.

{\bf Step 2: weights, rank, and rational consistency.} Let us give the horizontal variables weight $1$ and the central variables weight $2$, and let us freeze the base coefficients when taking the leading derivation
\[
\delta u_v=\sum_e q_{ve}u_e,\qquad \delta u_e=0,\qquad \delta f(q)=0.
\]
Then the Hamiltonian vector field raises the weight by at most $2$, and its part raising the weight by exactly $2$ is $\delta$. Indeed, differentiation of a base coefficient raises the weight by $1$, and a term $u_vu_t\partial_{u_s}$ raises it by $1+\mathrm{wt}(u_t)-\mathrm{wt}(u_s)$, which equals $2$ precisely for $s\in V,t\in E$.
Consequently, it follows by induction from \eqref{elem.A} and \eqref{weyl_raw_rec} that $A^p$ has weight at most $2p-1$ and $d^p$ has weight at most $2p$. Denoting by square brackets the homogeneous components of a given weighted degree, put
\[
\widehat A^p=(A^p)_{[2p-1]},\qquad
\widehat d^p=(d^p)_{[2p]},\qquad
\widehat\tau_e=(\tau_e)_{[3]}.
\]
Then these top-weight parts satisfy the relations
\begin{equation}\label{weyl_top_rec}
\begin{split}
&\widehat A^1=A^1,\qquad \widehat A^{p+1}=\delta\widehat A^p,\\
&\widehat d^1=\tau_V,\qquad
\widehat d^{p+1}=\delta\widehat d^p+\widehat A^p\widehat\tau_E,\\
&\widehat\tau_e=\sum_{f\in E}(\lambda_e-\lambda_f)q_{ef}u_f
+(\vec h_1\lambda_e-c\lambda_e)u_e.
\end{split}
\end{equation}
Note that the terms $-cd^p$ are of lower weight and therefore do not contribute here. We emphasize that the last term in \eqref{weyl_top_rec} contains the derivatives of the eigenvalues. Its vanishing is a consequence of the very statement we are proving, so this term cannot be neglected here without making the argument circular.

Let us fix a base point $q$. The matrix $\widehat A(q)$ is exactly the coefficient matrix of the fundamental algebraic system of the nilpotent approximation at $q$. Indeed, in the left-invariant frame of the nilpotent approximation all $q_{ef}$ with $e\in E$ vanish, and the Hamiltonian vector field acts on fiber polynomials with constant coefficients as $\delta$. By \cite[Theorem 2.10 and Proposition 3.11]{jmz2019}, ample covectors exist in an open dense subset of the fiber, and this matrix is injective there. Thus $\widehat A(q)$ has full column rank over $\mathbb R(u)$, the field of rational functions in the fiber variables. Note that this rank argument does not use any of the lemmas on graphs.

Further, by Proposition \ref{th:orb.to.proj}, an orbital diffeomorphism exists over a nonempty open subset of every base fiber. Hence all augmented maximal minors of $[A\mid d]$ vanish on this open subset and, being polynomials in the fiber variables, they vanish identically on that fiber. For a choice of rows in layers $p_1,\ldots,p_{|E|+1}$, the top possible weight of the corresponding minor is $\sum_j(2p_j-1)+1$, and its homogeneous component of this weight is the corresponding minor of $[\widehat A\mid\widehat d]$. Therefore, these minors vanish as well. Next, choose a nonzero square minor of $\widehat A$ formed by all of its columns and some of its rows, solve the corresponding equations by Cramer's rule, and append any further row. The vanishing of the resulting augmented minor shows that the same rational solution satisfies the equation corresponding to that row. Consequently, the system
\begin{equation}\label{weyl_top_system}
\widehat A^p\Psi=\widehat d^p\qquad(p\ge1)
\end{equation}
has a unique rational solution. Note that this argument is valid for the infinite matrix, because each additional row is tested by a finite minor.

{\bf Step 3: the first two internal layers.} Write $Z_i=\sum_{v\in V_i}u_vX_v$ and $\eta_i=\sum_{e\in E_i}u_e\varepsilon^e$. At the fixed point, define $J^i_{\eta_i}$ by means of the metric and the Levi bracket as in \eqref{J_eta}. Then, directly from $\delta Z_i=J^i_{\eta_i}Z_i$ and $\delta\eta_i=0$, we obtain by induction
\begin{equation}\label{weyl_operator}
\widehat A_i^p=(\ad_i((J^i_{\eta_i})^{p-1}Z_i))^*.
\end{equation}
In particular, this matrix is block diagonal. Since this calculation uses only the splitting of the symbol, it is valid in the present weakly adapted frame as well.

Call the fiber variables indexed by $V_i\cup E_i$ internal and all the other fiber variables external, and let $\iota_i$ be the operation of setting all the external variables equal to zero; the result of this operation applied to a polynomial is called its internal part. This operation acts on polynomials and commutes with $\delta$, because the derivation replaces each horizontal variable by a polynomial in the variables of its own component and annihilates the central variables. Put
\[
\beta_i=\tfrac12d(\alpha_i^2)|_{D_i(q)},\qquad
\sigma_i=\beta_i(Z_i),\qquad Y_i=J^i_{\eta_i}Z_i.
\]
Taking the internal part of \eqref{weyl_raw_tau} and using the identity $\iota_i(\vec h_1\lambda_i-c\lambda_i)=\sigma_i$, which follows by a direct computation from the formula of Step 1 given by the first divisibility condition, we obtain the following expressions for the internal parts of the first two layers:
\begin{equation}\label{weyl_internal_layers}
\iota_i(\widehat d_i^1)=\sigma_iZ_i^\flat-\|Z_i\|^2\beta_i,
\qquad
\iota_i(\widehat d_i^2)=2\sigma_iY_i^\flat+\beta_i(Y_i)Z_i^\flat.
\end{equation}
Here $\flat$ denotes metric duality. The first formula follows directly. To verify the second one, note first that the last line of \eqref{weyl_top_rec} gives $\iota_i(\widehat\tau_e)=\sigma_i u_e$ for $e\in E_i$. By the recursion for the top-weight parts, it then remains to differentiate the first formula by means of $\delta$, using
\[
\delta\sigma_i=\beta_i(Y_i),\qquad \delta Z_i=Y_i,
\qquad \delta\|Z_i\|^2=2\langle Z_i,Y_i\rangle=0,
\]
and to add $\sum_{e\in E_i}\widehat A^1_{ve}\sigma_i u_e=\sigma_i\langle Y_i,X_v\rangle$. This gives both coefficients in the second formula, including the factor $2$: one of the two equal contributions to its first term comes from the differentiation, and the other one comes from the added sum.

{\bf Step 4: contraction and conclusion.} Assume first that $W_i\ne0$, and contract the $i$th block of the second equation of \eqref{weyl_top_system} with $Y_i$. By the operator form of the coefficient matrix obtained in Step 3, on the domain of the rational solution the left-hand side of the resulting identity is equal to $\Psi_i([Y_i,Y_i])=0$. Thus the polynomial $\widehat d_i^2(Y_i)$ vanishes on a nonempty open set and, consequently, identically. Applying $\iota_i$ to this polynomial identity (and not to the rational solution, which may be singular where all the external variables vanish), and using \eqref{weyl_internal_layers} together with the orthogonality relation obtained in Step 3, we arrive at
\begin{equation}\label{weyl_key_identity}
d(\alpha_i^2)(Z_i)\,\|J^i_{\eta_i}Z_i\|^2=0.
\end{equation}
Note that the second factor here is a nonzero real polynomial whenever the bracket $\Lambda^2D_i\to W_i$ is nonzero. Indeed, in this case we can choose real $Z_i,Y$ with nonzero bracket and a real $\eta_i$ detecting it, i.e., not annihilating this bracket; then, by the definition of the corresponding endomorphism, its value on the first of these vectors has a nonzero inner product with the second one, and the positive definiteness of the metric gives a positive squared norm. Since the ring of real polynomials is an integral domain, we conclude that $d(\alpha_i^2)|_{D_i(q)}=0$.

It remains to consider the case $W_i=0$. Then all the first-layer rows of the corresponding block are zero. Hence the consistency of the top-weight system gives $\widehat d_i^1=0$, so that the first identity of \eqref{weyl_internal_layers} takes the form
\[
\beta_i(Z_i)Z_i^\flat=\|Z_i\|^2\beta_i\qquad\text{for all }Z_i\in D_i(q).
\]
When $\dim D_i\ge2$, as in the proposition, choose orthogonal unit vectors $Z_i,X_w$ and evaluate both sides on $X_w$. This gives $\beta_i(X_w)=0$ for every unit vector $X_w$, hence $\beta_i=0$. Since the point $q$ was arbitrary, this completes the proof of Proposition \ref{eigenconst}. Note that no equiregularity assumption on the eigendistributions has been used. \qed

\section{Vanishing of the correction terms}
\label{app:vanish}

\begin{Lemma}
\label{vanishlem}
Under the assumptions of Proposition \ref{eigenconst}, for every component $\Gamma_i$ of $\Gamma$ and every $w\in V_i$ the following statements hold:
\begin{enumerate}
\item $\mathcal E_w=0$; consequently, formula \eqref{RJ} holds and the first layer of the principal part is given by \eqref{tb1};
\item the additional term of \eqref{tbs_frak_full} vanishes identically; consequently, the recursive formula \eqref{tbs_frak} of Lemma \ref{recursive_tilde_b_frak} holds in the projective case as well.
\end{enumerate}
\end{Lemma}

\begin{proof}
We start with item (1). Suppose first that $|V_i|>1$. Then Proposition \ref{eigenconst} gives
\[
X_x\bigl(\alpha(w)^2\bigr)=X_x(\alpha_i^2)=0 ,
\qquad
x\in V_i,
\]
and substitution into \eqref{projective_residual} yields $\mathcal E_w=0$. If, on the other hand, $V_i=\{w\}$, then \eqref{projective_residual} reduces to
\[
\mathcal E_w=\frac12\left(u_wX_w\bigl(\alpha(w)^2\bigr)u_w-X_w\bigl(\alpha(w)^2\bigr)u_w^2\right)=0 .
\]
Thus $\mathcal E_w=0$ for every component of $\Gamma$. By Remark \ref{rjrem}, the expression for $R_w$ coming from \cite{jmz2019} then coincides with the right-hand side of \eqref{RJ}, and \eqref{tb1_full} reduces to \eqref{tb1}. This proves item (1).

Let us prove item (2). If $|V_i|>1$, then the first factor of the additional term of \eqref{tbs_frak_full} vanishes by Proposition \ref{eigenconst}; if $|V_i|=1$, then $E_i=\varnothing$ and the second factor is an empty sum. Hence the additional term vanishes for every component of $\Gamma$ and every $p\geq1$, so \eqref{tbs_frak_full} reduces to \eqref{tbs_frak}.
\end{proof}

\section{The characteristic tuple and the growth of Jacobi curves}
\label{app:jacobi}

In this appendix, we explain the geometric meaning of the characteristic tuple of Definition \ref{char_tuple_def}. This material is not needed for the proofs of the main results of the paper; it is used only in the proof of Proposition \ref{wc:criterion_prop}, in order to show that the non-ample set of a nilpotent approximation is a proper subset.

Let us first recall the relevant notions. Let $g_1$ be a sub-Riemannian metric on a distribution $D$ of rank $m$ on an $n$-dimensional manifold $M$, let $h_1$ be its Hamiltonian, and let $\lambda\in T^*M$ with $h_1(\lambda)\neq0$. Jacobi curves of extremals are defined in \cite{jac1}. As in \cite[Definition 2.5]{jmz2019}, we use the Jacobi curve of the normal extremal $\lambda(t)=e^{t\vec h_1}\lambda$ in the following form, without the reduction used in \cite{jac1}:
\[
J_\lambda(t)=e^{-t\vec h_1}_*\,\mathcal V_{\lambda(t)},
\]
where $\mathcal V_\mu$ denotes the tangent space at $\mu$ to the fiber of $T^*M$ through $\mu$; this is a curve of Lagrangian subspaces of the $2n$-dimensional symplectic space $T_\lambda(T^*M)$. The \emph{$i$th extension} $J^{(i)}_\lambda(t)$ of the Jacobi curve is the linear span of the derivatives of order at most $i$ at $t$ of all smooth curves $\tau\mapsto\ell(\tau)\in J_\lambda(\tau)$, and the subspaces
\[
J_\lambda(t)\subset J^{(1)}_\lambda(t)\subset J^{(2)}_\lambda(t)\subset\cdots
\]
form the \emph{osculating flag} of the Jacobi curve. The Jacobi curve is called \emph{ample} at $t$ if $J^{(k)}_\lambda(t)=T_\lambda(T^*M)$ for some $k$ (see \cite[Definition 2.9]{jmz2019}), and \emph{equiregular} at $t$ if the dimensions of all the subspaces $J^{(i)}_\lambda$ are constant near $t$. In the latter case, the numbers $\dim J^{(i)}_\lambda(t)-\dim J^{(i-1)}_\lambda(t)$, $i\geq1$, where $J^{(0)}_\lambda=J_\lambda$, do not increase with $i$, and they are the lengths of the columns of the \emph{Young diagram} of the Jacobi curve \cite{li2}. The length of the first column is the rank of the Jacobi curve, which is equal to $m$ for $J_\lambda$. By the \emph{growth} of the Jacobi curve, or of its osculating flag, we mean the sequence of the lengths of the columns of its Young diagram starting from the second column.

The Jacobi curve in the sense of \cite{jac1} is obtained from $J_\lambda$ by the reduction by the direction of $\vec h_1$ on the level set $\mathcal H$ of $h_1$ through $\lambda$: it is the curve $t\mapsto\bigl((J_\lambda(t)\cap T_\lambda\mathcal H)+\mathbb R\vec h_1(\lambda)\bigr)/\mathbb R\vec h_1(\lambda)$ in the Lagrangian Grassmannian of the $(2n-2)$-dimensional symplectic space $T_\lambda\mathcal H/\mathbb R\vec h_1(\lambda)$. This reduction decreases the length of the first column of the Young diagram by one, from $m$ to $m-1$, which is the rank of the Jacobi curves of sub-Riemannian normal extremals in \cite{li2}, and it does not change the other columns, i.e., it does not change the growth. Indeed, let $e$ be the Euler vector field of the fibers of $T^*M$, and set $\epsilon(\tau):=e^{-\tau\vec h_1}_*e\bigl(\lambda(\tau)\bigr)$. Then $\epsilon(\tau)\in J_\lambda(\tau)$, and $\epsilon(\tau)\notin T_\lambda\mathcal H$, because $dh_1(e)=2h_1\neq0$. Moreover, since $h_1$ is quadratic on the fibers, we have $[e,\vec h_1]=\vec h_1$, so that $\epsilon'(\tau)=-\vec h_1(\lambda)$ for all $\tau$. Consequently, $J^{(i)}_\lambda(0)=(J_\lambda\cap T_\lambda\mathcal H)^{(i)}(0)+\operatorname{span}\{\epsilon(0),\vec h_1(\lambda)\}$ for $i\geq1$, which shows that the reduction decreases $\dim J_\lambda(0)$ by one and $\dim J^{(i)}_\lambda(0)$, $i\geq1$, by two.

Now let $\mathfrak m=H\oplus W$ be a step $2$ fundamental graded Lie algebra with $m=\dim H$ and characteristic tuple $(r_0,\ldots,r_\omega)$, so that $\dim W=\sum_{j=0}^{\omega}r_j$ by \eqref{sum_r}, and let $\langle\cdot,\cdot\rangle$ be an inner product on $H$. The \emph{flat model} is the Carnot group of $\mathfrak m$ endowed with the left-invariant sub-Riemannian metric defined by $\langle\cdot,\cdot\rangle$. Given $\eta\in W^*$, let $J_\eta$ be the skew-symmetric operator on $H$ defined by $\langle J_\eta X,Y\rangle=\eta([X,Y])$, as in \eqref{wc:J}. Below, a property of covectors $\lambda\in T^*_qM$ is said to hold for a \emph{generic} $\lambda$ if it holds on a nonempty Zariski open subset of the vector space $T^*_qM$.

\begin{Lemma}
\label{Jacobi_lem}
Let $\mathfrak m$, $\langle\cdot,\cdot\rangle$, and $(r_0,\ldots,r_\omega)$ be as above, and let $n=\dim\mathfrak m$.
\begin{enumerate}
\item For the flat model and a generic covector $\lambda$, the Jacobi curve $J_\lambda$ is ample and equiregular at $t=0$, and its growth is $(r_0,r_1,\ldots,r_\omega)$. In other words, the Young diagram of $J_\lambda$ has columns of lengths $m,r_0,r_1,\ldots,r_\omega$, and the Young diagram of the Jacobi curve of \cite{jac1} has columns of lengths $m-1,r_0,r_1,\ldots,r_\omega$.
\item Let $g_1$ be a sub-Riemannian metric on a step $2$ distribution $D$ on an $n$-dimensional manifold $M$, and let $q\in M$ be a point such that the Tanaka symbol $\mathfrak m(q)$, together with the inner product $g_1(q)$ on $\mathfrak m_{-1}(q)=D(q)$, is isomorphic to $\mathfrak m$ together with $\langle\cdot,\cdot\rangle$. Then, for a generic covector $\lambda\in T^*_qM$, the Jacobi curve $J_\lambda$ is ample at $t=0$, and the increments $a_j:=\dim J^{(j+2)}_\lambda(0)-\dim J^{(j+1)}_\lambda(0)$, $j\geq0$, of its osculating flag satisfy
\begin{equation}
\label{jac_ineq}
a_0=r_0,\qquad \sum_{j=0}^{s}a_j\geq\sum_{j=0}^{s}r_j\quad\text{for all }s\geq0,\qquad \sum_{j\geq0}a_j=\sum_{j=0}^{\omega}r_j=n-m,
\end{equation}
where the characteristic tuple is padded by zeros. If, in addition, $J_\lambda$ is equiregular at $t=0$, which is the case for a generic $\lambda\in T^*_qM$ whenever $q$ belongs to a certain open dense subset of $M$, then $(a_0,a_1,\ldots)$ is the growth of $J_\lambda$, and \eqref{jac_ineq} means that the partition $(a_0,a_1,\ldots)$ of $n-m$ dominates the partition $(r_0,\ldots,r_\omega)$. In particular, the growth of the Jacobi curves of a nonflat metric can be faster than that of the flat model with the same symbol, but it is never slower: if $(a_j)\neq(r_j)$, then $a_j>r_j$ for the first index $j$ at which these sequences differ.
\end{enumerate}
\end{Lemma}

\begin{proof}
We proceed in three steps.

\emph{Step 1: the extension matrix.} Let $g_1$ be an arbitrary sub-Riemannian metric on a step $2$ distribution $D$. Fix a point $q$ and a local frame $\{X_s\}_{s\in V\cup E}$ such that $\{X_v\}_{v\in V}$ is $g_1$-orthonormal, and let $A=(A^p)_{p\geq1}$ be the matrix defined by the recursive formulas \eqref{elem.A}; this matrix depends only on $g_1$ and on the frame. By \cite[Lemmas 2.8 and 3.12]{jmz2019}, applied with $g_2=g_1$, we have $\dim J^{(1)}_\lambda(0)=n+m$ and
\begin{equation}
\label{jac_ext_matrix}
\dim J^{(s+2)}_\lambda(0)=n+m+\rank\begin{pmatrix}A^1(\lambda)\\ \vdots\\ A^{s+1}(\lambda)\end{pmatrix},\qquad s\geq0.
\end{equation}
Since $J_\lambda(t+\tau)=e^{-t\vec h_1}_*J_{\lambda(t)}(\tau)$, the osculating flag of $J_\lambda$ at time $t$ is the image of the osculating flag of $J_{\lambda(t)}$ at time $0$ under the linear isomorphism $e^{-t\vec h_1}_*$. Hence \eqref{jac_ext_matrix} holds at every $t$, with $\lambda$ replaced by $\lambda(t)$, and $J_\lambda$ is equiregular at $t=0$ if and only if all the ranks in \eqref{jac_ext_matrix} are constant along $\lambda(t)$ for $t$ near $0$.

Given $\lambda\in T_q^*M$, set $Z:=\sum_{v\in V}u_v(\lambda)X_v(q)$, and let $\eta$ be the functional on $T_qM/D(q)$ such that $\eta\bigl(X_e(q)\bmod D(q)\bigr)=u_e(\lambda)$ for $e\in E$. The entries of $A^p$ are homogeneous polynomials of degree $p$ in the quasi-impulses and have vertex degree at least one, i.e., degree at most $p-1$ in the edge variables $u_e$, $e\in E$ (see Remark \ref{homogen_total} and Lemma \ref{lin_quad_lem}). Their part of degree $p-1$ in the edge variables is the layer $\mathfrak A^p$ of the principal part, which, by \eqref{all_layer_operator} applied with $k=1$, is the matrix of $\bigl(\ad(J_\eta^{p-1}Z)\bigr)^*$, i.e., of the corresponding layer for the flat model of the symbol $\mathfrak m(q)$ (cf.\ \eqref{wc:Ap}). In other words, each step of the recursion \eqref{elem.A} raises the degree in the edge variables by at most one, and it does so only through the term $\delta u_v=\sum_{e\in E}q_{ve}u_e$ from \eqref{principal_delta}, which corresponds to $Z\mapsto J_\eta Z$; the differentiation of the structure functions or of the edge variables does not raise this degree. Note that all these statements use only the recursive formulas \eqref{elem.A}, so they hold for an arbitrary metric $g_1$.

\emph{Step 2: the flat model.} For the flat model, take a left-invariant frame such that $\{X_v\}_{v\in V}$ is an orthonormal basis of $H$ and $\{X_e\}_{e\in E}$ is a basis of $W$. Then the only nonzero structure constants are the $c^e_{xv}$ with $x,v\in V$ and $e\in E$, so the matrix $A$ coincides with its principal part. Since the rows of $A^p(\lambda)$ are the coordinates of the vectors $[J_\eta^{p-1}Z,X_v]$, $v\in V$, with respect to the basis $\{X_e\}_{e\in E}$ of $W$, we get
\[
\rank\begin{pmatrix}A^1(\lambda)\\ \vdots\\ A^{s+1}(\lambda)\end{pmatrix}=\dim\sum_{j=0}^s[J_\eta^jZ,H],\qquad s\geq0
\]
(compare with \cite[Remark 6.9]{NGZ}). We claim that, for generic $(Z,\eta)$, the dimensions of the successive increments of the subspaces $\sum_{j=0}^s[J_\eta^jZ,H]$, $s\geq0$, are equal to $r_0,r_1,\ldots,r_\omega$. To prove this, let us describe the successive generic choices in more detail. Set $Z_j=J_\eta^jZ$, $I_{-1}=0$, $K_{-1}=0$, and, whenever $Z_j$ is ad-maximal in $H/K_{j-1}\oplus W/I_{j-1}$, define
\[
I_j=I_{j-1}+[Z_j,H],\qquad
K_j=\{Y\in H:[Z_j,Y]\in I_{j-1}\}.
\]
Here all the quotients are realized as quotients of the original algebra $H\oplus W$: by induction on $j$, the subspaces $K_j$ and $I_j$ are the preimages in $H$ and in $W$ of the kernel and of the image of the operator $\ad Z_j\colon H/K_{j-1}\to W/I_{j-1}$ of the quotient algebra $H/K_{j-1}\oplus W/I_{j-1}$. By Lemma \ref{a1}, applied to this quotient algebra, $K_j\oplus I_j$ is an ideal of $H\oplus W$, so these subspaces define precisely the successive quotients that appear in the definition of the characteristic tuple. In particular, $[H,K_j]\subset I_j$; hence $\ad Z_{j+1}$ is well defined on $H/K_j$ modulo $I_j$, and $K_j\subset K_{j+1}$. We start with a generic ad-maximal $Z$. If $\nu\in I_{j-1}^{\perp}\subset W^*$, then replacing $\eta$ by $\eta+\nu$ does not change $Z_0,\ldots,Z_j$, since $J_\nu Z_a=0$ for $a<j$, whereas the next vector of the sequence changes exactly by $J_\nu Z_j$. The adjoint of the map $\nu\mapsto J_\nu Z_j$ is the map $Y\mapsto[Z_j,Y]\bmod I_{j-1}$, whose kernel is $K_j$. Hence, these changes cover the orthogonal complement $K_j^\perp$ of $K_j$ in $H$. Since the inner product is positive definite, we have $H=K_j\oplus K_j^\perp$, so $K_j^\perp$ projects isomorphically onto the next horizontal quotient $H/K_j$, i.e., the classes of these changes cover the entire quotient $H/K_j$. Consequently, they permit a generic choice at the next step that maximizes the rank, while all the preceding choices remain unchanged. Iterating this argument, we realize the generic characteristic tuple by a single pair $(Z,\eta)$. Moreover, all the rank conditions involved are expressed by the nonvanishing of polynomial minors, so they hold simultaneously on a nonempty Zariski-open set. On this set, a larger increment at the next step is impossible, because such an increment is bounded by the maximal rank of the adjoint operators in the corresponding quotient. This proves that the successive increments are equal to the entries of the characteristic tuple.

Since $\sum_{j=0}^{\omega}r_j=\dim W=n-m$, it follows from \eqref{jac_ext_matrix} that $J^{(\omega+2)}_\lambda(0)=T_\lambda(T^*G)$, i.e., the Jacobi curve is ample at $t=0$. Finally, the flat model is homogeneous, so the maximal values of the ranks in \eqref{jac_ext_matrix} on the fibers of $T^*G$ do not depend on the fiber. The set of covectors at which all these ranks attain their maximal values is therefore open in $T^*G$, so it contains $\lambda(t)$ for $t$ near $0$ whenever it contains $\lambda$. By Step 1, the Jacobi curve is equiregular at $t=0$ for a generic $\lambda$. Together with the discussion of the reduction above, this proves item (1).

\emph{Step 3: an arbitrary metric.} Let $g_1$ and $q$ be as in item (2), and let $\rho_s$ denote the rank of the matrix in \eqref{jac_ext_matrix}. By Step 1, under the rescaling $\eta\mapsto\sigma\eta$, $\sigma\in\mathbb R$, a nonzero minor of the corresponding matrix of the flat model is the leading coefficient in $\sigma$ of the corresponding minor of this matrix. Since the entries of $A$ are polynomials on $T^*_qM$, it follows from Step 2 that, for each $s$, the inequality $\rho_s(\lambda)\geq\sum_{j=0}^{s}r_j$ holds for a generic $\lambda\in T^*_qM$. By \eqref{jac_ext_matrix}, this proves the inequalities in \eqref{jac_ineq}. Since $\rho_s\leq n-m$ and $\sum_{j=0}^{\omega}r_j=n-m$, we also get $J^{(\omega+2)}_\lambda(0)=T_\lambda(T^*M)$, i.e., the Jacobi curve is ample at $t=0$, and $\sum_{j\geq0}a_j=n-m$. Moreover, the first layer $A^1$ coincides with $\mathfrak A^1$, so $a_0=\rank A^1(\lambda)$ is equal to the rank of $\ad Z$ in $\mathfrak m(q)$, which is $r_0$ for a generic $\lambda$. Finally, let $\hat\rho_s(q')$ denote the maximal value of $\rho_s$ on $T^*_{q'}M$. Since the rank of a continuous matrix-valued function is lower semicontinuous, so is $\hat\rho_s$, and, since $\hat\rho_s$ takes only finitely many values, it is locally constant on an open dense subset of $M$. If $q$ belongs to this subset and $\rho_s(\lambda)=\hat\rho_s(q)$ for all $s$, then, for $t$ near $0$, we have $\hat\rho_s(q)=\rho_s(\lambda)\leq\rho_s(\lambda(t))\leq\hat\rho_s\bigl(\pi(\lambda(t))\bigr)=\hat\rho_s(q)$, so, by Step 1, the Jacobi curve is equiregular at $t=0$. In this case, the increments $a_j$ are the lengths of the columns of the Young diagram of $J_\lambda$ starting from the second column. This proves item (2).
\end{proof}

\begin{rk}
\label{Jacobi_strict_rem}
The inequalities in item (2) of Lemma \ref{Jacobi_lem} can be strict. For example, let $E_{ij}=e_ie_j^{\mathrm T}-e_je_i^{\mathrm T}$, $1\leq i<j\leq5$, be the standard basis of $\mathfrak{so}(5)$, and consider the left-invariant sub-Riemannian metric on $SO(5)$ for which the left-invariant vector fields corresponding to $E_{15}$, $E_{25}$, $E_{35}$, and $E_{45}+E_{12}$ form an orthonormal frame of the distribution. These four elements and their six pairwise brackets form a basis of $\mathfrak{so}(5)$, so the Tanaka symbol at every point is the free step $2$ nilpotent Lie algebra with four generators, whose characteristic tuple is $(3,2,1)$. A direct computation based on \eqref{jac_ext_matrix} shows that, for a generic covector, the matrices formed by the first layer and by the first two layers of $A$ have ranks $3$ and $6$, i.e., $(a_0,a_1)=(3,3)$, whereas for the flat model the matrices formed by the first one, two, and three layers have ranks $3$, $5$, and $6$. Since this structure is homogeneous, the generic Jacobi curves are equiregular, and the Young diagram of a generic Jacobi curve in the sense of \cite{jac1} has columns of lengths $3,3,3$ instead of $3,3,2,1$. Note that $a_2=0<r_2=1$ in this example, so \eqref{jac_ineq} does not imply entrywise inequalities between the columns. On the other hand, for the horizontal subspace spanned by $E_{15},\ldots,E_{45}$, which corresponds to the symmetric pair $\bigl(\mathfrak{so}(5),\mathfrak{so}(4)\bigr)$, the Tanaka symbol is the same, and the growth is $(3,2,1)$, as for the flat model.
\end{rk}


\begin{thebibliography}{99}
\bibitem{ABB}
A.~Agrachev, D.~Barilari, and U.~Boscain,
\emph{A comprehensive introduction to sub-Riemannian geometry. From the Hamiltonian viewpoint}  (with an appendix by I. Zelenko), Cambridge Studies in Advanced Mathematics, 181. Cambridge University Press, Cambridge, 2020. xviii+745 pp. ISBN: 978-1-108-47635-5
\bibitem{AgrSach}
A.A. Agrachev, Yu. L.  Sachkov,
\emph{Control theory from the geometric viewpoint},
Encyclopaedia of Mathematical Sciences, 87. Control Theory and Optimization, II. Springer-Verlag, Berlin, 2004.
 \bibitem{jac1}
	A. Agrachev, I. Zelenko, Geometry of Jacobi curves - I., \emph{J. Dynamical and Control systems}, {\bf 8}, No.\ 1, 93--140, 2002.
\bibitem{Alekseevsky2020}
D. Alekseevsky,
\emph{Shortest and straightest geodesics in sub-Riemannian geometry},
J. Geom. Phys. 155 (2020), 103713, 22 pp.

\bibitem{BaudoinCho}
F.~Baudoin and G.~Cho,
\emph{The subelliptic heat kernel of the octonionic Hopf fibration},
Potential Anal. \textbf{55} (2021), 211--228.
\url{https://arxiv.org/abs/1904.08568}
\bibitem{deRham1952}
 G. de Rham, \emph{Sur la r\'{e}ductibilit\'{e} d'un espace de Riemann}, Commentarii Mathematici Helvetici, \textbf{26}, No.\ 1, 328--344, 1952.
\bibitem{Dini1870}
U.~Dini, \emph{Sopra un problema che si presenta nella teoria generale delle rappresentazioni geografiche di una superficie su di un'altra}, Ann. Mat. Ser.~2\textordfeminine{} \textbf{3} (1870), 269--287.
\bibitem{Eisenhart1923}
L.P. Eisenhart.
\emph{ Symmetric tensors of the second order whose first covariant
  derivatives are zero},
Trans. Amer. Math. Soc., \textbf{25} (2), 297--306, 1923.
\bibitem{jmz2019}
F. Jean, S. Maslovskaya and I. Zelenko, \emph{On projective and affine equivalence of sub-Riemannian metrics}, Geom. Dedicata 203 (2019), 279--319. \url{https://doi.org/10.1007/s10711-019-00437-1}
\bibitem{JMZ21}
F.~Jean, S.~Maslovskaya, and I.~Zelenko,
\emph{On Weyl's type theorems and genericity of projective rigidity in sub-Riemannian Geometry},
Geometriae Dedicata, \textbf{213} (2021), 295--314.
\url{https://doi.org/10.1007/s10711-020-00581-z}
\bibitem{kaplan}
A. Kaplan, \emph{ Fundamental solutions for a class of hypoelliptic PDE generated by composition of quadratic forms}, Trans. Amer. Math. Soc., 258(1):147--153, 1980.
\bibitem{LawsonMichelsohn}
H.~B.~Lawson and M.-L.~Michelsohn, \emph{Spin geometry}, Princeton Mathematical Series, 38. Princeton University Press, Princeton, NJ, 1989.
\bibitem{LeviCivita1896}
T.~Levi-Civita, \emph{Sulle trasformazioni delle equazioni dinamiche}, Ann. Mat. Ser.~2\textordfeminine{} \textbf{24} (1896), 255--300.
\bibitem{Zaifeng} Z. Lin, I. Zelenko. \emph {On Eisenhart's type theorem for sub-Riemannian metrics on step 2 distributions with ad-surjective Tanaka symbols},
Regul. Chaotic Dyn. 29 (2024), no. 2, 304–343.
\bibitem{Morimoto2008}
T. Morimoto, \emph{ Cartan connection associated with a subriemannian structure}, Differential Geom. Appl. 26 (2008), no. 1, 75--78.

\bibitem{NGZ}
S.~Nicolussi Golo and Y.~Zhang,
\emph{Curvature exponent and geodesic dimension on Sard-regular Carnot groups},
Analysis and Geometry in Metric Spaces \textbf{12} (2024), article 20240004.
\url{https://arxiv.org/abs/2308.15811}
\bibitem{Pont}
 L. S. Pontryagin,  V. G. Boltyanskii, R. V. Gamkrelidze, E. F. Mishchenko,  \emph{The mathematical theory of optimal processes}, Translated from the Russian by K. N. Trirogoff; edited by L. W. Neustadt Interscience Publishers John Wiley \& Sons, Inc., New York-London 1962 viii+360 pp
\bibitem{tan1} N.~Tanaka, \emph{On differential systems, graded Lie
  algebras and pseudo-groups}, J.\ Math.\ Kyoto.\ Univ., \textbf{10}
  (1970), pp.~1--82.
	 \bibitem{Z}
	I. Zelenko, \emph{ On the geodesic equivalence of Riemannian metrics and sub-Riemannian metrics on distributions of corank 1}, Sovrem. Mat. Prilozh. No.\ 21, Geom. Zadachi Teor. Upr., 79--105, 2004; Engl.transl.: J. Math. Sci. (N. Y.) {\bf 135} (4), 3168--3194, 2006.
\bibitem{zeltan} I. Zelenko, \emph{ On Tanaka's prolongation procedure for filtered structures of constant type}, Symmetry, Integrability, and Geometry: Methods and Applications (SIGMA), special issue ``\'{E}lie Cartan and Differential Geometry'', v. 5, 094,
2009, 21 pages.
\bibitem{li2}
	I. Zelenko, C. Li, \emph{Differential geometry of
		 curves in Lagrange Grassmannians with given Young diagram}, Differential Geometry and its Applications, {\bf 27}, 723--742, 2009.
\end{thebibliography}
\end{document}